\documentclass[11pt,onecolumn]{IEEEtran}

\usepackage{amsthm}
\usepackage{times,amssymb,amsmath,amsfonts,float,nicefrac,color,bbm,mathrsfs,caption,float}
\usepackage{algorithm,enumerate,multirow,caption,tikz,graphicx}
\usepackage{array}
\usepackage[mathscr]{eucal}
\usepackage{sidecap}
\usepackage{algpseudocode}
\usepackage{verbatim}
\usepackage{epstopdf}
\usepackage{textcomp}
\usetikzlibrary{shapes,arrows}
\usepackage{stmaryrd}
\usepackage{mathabx}
\usepackage[noadjust]{cite}
\usepackage{booktabs}
\usepackage[normalem]{ulem}

\usepackage[top=1in,bottom=1in,left=1in,right=1in]{geometry}
\allowdisplaybreaks

\newtheorem{theorem}{Theorem}[section]

\newtheorem{lemma}[theorem]{Lemma}
\newtheorem{proposition}[theorem]{Proposition}
\newtheorem{definition}[theorem]{Definition}

\newtheorem{conjecture}{Conjecture}

\newtheorem{cnstr}{Construction}

\newcounter{remark}[section]

   \newcounter{example}[section]

\newcommand{\RN}[1]{%
  \textup{\expandafter{\romannumeral#1}}%
}

\newcommand\remove[1]{}
\newcommand{\nc}{\newcommand}

\def\mathbi#1{{\textbf{\textit #1}}}
\nc\bfa{{\boldsymbol a}}\nc\bfA{{\boldsymbol A}}\nc\cA{{\mathscr A}}\nc\sA{{\mathscr A}}
\nc\bfb{{\boldsymbol b}}\nc\bfB{{\boldsymbol B}}\nc\cB{{\mathscr B}}\nc\sB{{\mathscr B}}
\nc\bfc{{\boldsymbol c}}\nc\bfC{{\boldsymbol C}}\nc\cC{{\mathscr C}}\nc\sC{{\mathscr C}}
\nc\bfd{{\boldsymbol d}}\nc\bfD{{\boldsymbol D}}\nc\cD{{\mathscr D}}
\nc\bfe{{\boldsymbol e}}\nc\bfE{{\boldsymbol E}}\nc\cE{{\mathscr E}}
\nc\bff{{\boldsymbol f}}\nc\bfF{{\boldsymbol F}}\nc\cF{{\mathscr F}}\nc\sF{{\mathscr F}}
\nc\bfg{{\boldsymbol g}}\nc\bfG{{\boldsymbol G}}\nc\cG{{\mathscr G}}
\nc\bfh{{\boldsymbol h}}\nc\bfH{{\boldsymbol H}}\nc\cH{{\mathscr H}}
\nc\bfi{{\boldsymbol i}}\nc\bfI{{\boldsymbol I}}\nc\cI{{\mathscr I}}\nc\sI{{\mathscr I}}
\nc\bfj{{\boldsymbol j}}\nc\bfJ{{\boldsymbol J}}\nc\cJ{{\mathscr J}}
\nc\bfk{{\boldsymbol k}}\nc\bfK{{\boldsymbol K}}\nc\cK{{\mathscr K}}
\nc\bfl{{\boldsymbol l}}\nc\bfL{{\boldsymbol L}}\nc\cL{{\mathscr L}}
\nc\bfm{{\boldsymbol m}}\nc\bfM{{\boldsymbol M}}\nc\cM{{\mathscr M}}
\nc\bfn{{\boldsymbol n}}\nc\bfN{{\boldsymbol N}}\nc\cN{{\mathcal N}}
\nc\bfo{{\boldsymbol o}}\nc\bfO{{\boldsymbol O}}\nc\cO{{\mathscr O}}
\nc\bfp{{\boldsymbol p}}\nc\bfP{{\boldsymbol P}}\nc\cP{{\mathscr P}}\nc\eP{{\EuScriptP}}\nc\fP{{\mathfrak P}}
\nc\bfq{{\boldsymbol q}}\nc\bfQ{{\boldsymbol Q}}\nc\cQ{{\mathscr Q}}
\nc\bfr{{\boldsymbol r}}\nc\bfR{{\boldsymbol R}}\nc\cR{{\mathscr R}}\nc\sR{{\mathscr R}}
\nc\bfs{{\boldsymbol s}}\nc\bfS{{\boldsymbol S}}\nc\cS{{\mathscr S}}
\nc\bft{{\boldsymbol t}}\nc\bfT{{\boldsymbol T}}\nc\cT{{\mathscr T}}
\nc\bfu{{\boldsymbol u}}\nc\bfU{{\boldsymbol U}}\nc\cU{{\mathscr U}}
\nc\bfv{{\boldsymbol v}}\nc\bfV{{\boldsymbol V}}\nc\cV{{\mathscr V}}\nc\sV{{\mathscr V}}
\nc\bfw{{\boldsymbol w}}\nc\bfW{{\boldsymbol W}}\nc\cW{{\mathscr W}}\nc\sW{{\mathscr W}}
\nc\bfx{{\boldsymbol x}}\nc\bfX{{\boldsymbol X}}\nc\cX{{\mathscr X}}
\nc\bfy{{\boldsymbol y}}\nc\bfY{{\boldsymbol Y}}\nc\cY{{\mathscr Y}}
\nc\bfz{{\boldsymbol z}}\nc\bfZ{{\boldsymbol Z}}\nc\cZ{{\mathscr Z}}

\DeclareMathOperator*{\argmin}{arg\!\min}

\DeclareMathOperator{\TV}{TV}
\DeclareMathOperator{\tr}{tr}
\DeclareMathOperator{\Var}{Var}

\DeclareMathOperator{\kl}{kl}

\begin{document}


\title{Asymptotically optimal private estimation under mean square loss}

\author{\IEEEauthorblockN{Min Ye} \hspace*{1in}
\and \IEEEauthorblockN{Alexander Barg}}

\maketitle
{\renewcommand{\thefootnote}{}\footnotetext{

\vspace{-.2in}
 
\noindent\rule{1.5in}{.4pt}

The authors are with Dept. of ECE and ISR, University of Maryland, College Park, MD 20742. Emails: yeemmi@gmail.com and abarg@umd.edu. Research supported by NSF grants CCF1422955 and CCF1618603.}
}
\renewcommand{\thefootnote}{\arabic{footnote}}
\setcounter{footnote}{0}

\begin{abstract}
We consider the minimax estimation problem of a discrete distribution with support size $k$ under locally differential privacy constraints. A privatization scheme is applied to each raw sample independently, and we need to estimate the distribution of the raw samples from the privatized samples. A positive number $\epsilon$ measures the privacy level of a privatization scheme. 

In our previous work (arXiv:1702.00610), we proposed a family of new privatization schemes and the corresponding estimator. We also 
proved that our scheme and estimator are order optimal in the regime $e^{\epsilon} \ll k$ under both $\ell_2^2$ and $\ell_1$ loss. 
In other words, for a large number of samples the worst-case estimation loss of our scheme was shown to differ from the optimal value by at most a constant factor. In this paper, we eliminate this gap by showing asymptotic optimality of the proposed scheme  and estimator 
under the $\ell_2^2$ (mean square) loss. More precisely, we show that for any $k$ and $\epsilon,$ the ratio between the worst-case estimation loss of our scheme and the optimal value approaches $1$ as the number of samples tends to infinity.
\end{abstract}

{\small\tableofcontents}

\vspace*{1in}
\section{Introduction} This paper continues our work \cite{Ye17}. The context of the problem that we consider is
related to a major challenge in the statistical analysis of user data, namely, the conflict between learning accurate 
statistics and protecting sensitive information about the individuals. 
As in \cite{Ye17}, we rely on a particular formalization of user privacy called {\em differential privacy}, introduced in \cite{Dwork06, Dwork08}.
Generally speaking, differential privacy requires that the adversary not be able to reliably infer an individual's data from public statistics even with access to all the other users' data.
The concept of differential privacy has been developed in two different contexts: the {\em global privacy}
context (for instance, when institutions release statistics related to groups of people) \cite{Ghosh12}, and the {\em local privacy} context when individuals disclose their personal data \cite{Duchi13}.

In this paper, we consider the minimax estimation problem of a discrete distribution with support size $k$ under locally differential privacy.
This problem has been studied in the non-private setting \cite{Kamath15, Lehmann06}, where we can learn the distribution from the raw samples.
In the private setting, we need to estimate the distribution of raw samples from the privatized samples which are generated independently from the raw samples according to a conditional distribution  $\mathbi{Q}$ (also called a {\em privatization scheme}).
Given a privacy parameter $\epsilon>0,$
we say that $\mathbi{Q}$ is $\epsilon$-locally differentially private if the probabilities of the same output conditional on different inputs differ by a factor of at most $e^{\epsilon}.$ Clearly, smaller $\epsilon$ means that it is more difficult to infer the original data from the privatized samples, and thus leads to higher privacy.
For a given $\epsilon,$ our objective is to find the optimal privatization scheme with $\epsilon$-privacy level to minimize the expected estimation loss for the worst-case distribution.
In this paper, we are mainly
concerned with the scenario where we have a large number of
samples, which captures the modern trend toward ``big data" analytics.

\subsection{Existing results}\label{sec:existing} The following two privatization schemes are the most well-known in the literature: the $k$-ary Randomized Aggregatable Privacy-Preserving Ordinal Response ($k$-RAPPOR) scheme \cite{Duchi13a, Erlingsson14}, and the $k$-ary Randomized Response ($k$-RR) scheme
\cite{Warner65,Kairouz14}.
The $k$-RAPPOR scheme is order optimal in the high privacy regime where $\epsilon$ is very close to $0,$
and the $k$-RR scheme is order optimal in the low privacy regime where $e^{\epsilon} \approx k$ \cite{Kairouz16}. 
Very recently, a family of privatization schemes and the corresponding estimators were proposed independently by Wang et al. \cite{Wang16} and the present authors \cite{Ye17}. In \cite{Ye17}, we further showed that
under both $\ell_2^2$ and $\ell_1$ loss, these privatization schemes and the corresponding estimators are order-optimal in the medium to high privacy regimes when $e^{\epsilon} \ll k.$

Duchi et al.~\cite{Duchi16} gave an order-optimal lower bound on the minimax private estimation loss for the high privacy regime where $\epsilon$ is very close to $0$. In \cite{Ye17}, we proved a stronger lower bound which is order-optimal in the whole region $e^{\epsilon} \ll k$. This lower bound implies that the schemes and the estimators proposed in \cite{Wang16,Ye17} are order optimal in this regime.   
Here order-optimal means that the ratio between the true value and the lower bound is upper bounded by a constant (larger than 1) when $n$ and $k/e^{\epsilon}$ both become large enough.

\subsection{Our contributions}
In this paper, we focus on the $\ell_2^2$ (mean square) loss. We prove an asymptotically tight lower bound on the minimax private estimation loss for all values of $k$ and $\epsilon$. In other words, for every $k$ and every $\epsilon$, the ratio between the true value and our lower bound goes to $1$ when $n$ goes to infinity. This is a huge improvement over the lower bounds in \cite{Ye17} and \cite{Duchi16} for the following two reasons: First, although the lower bounds in \cite{Ye17} and \cite{Duchi16} are order-optimal, they differ from the true value by a factor of several hundred. In practice, an improvement of several percentage points is already considered as a substantial advance (see for instance,~\cite{Kairouz16}). So these order-optimal bounds are far from satisfactory. Second, the bounds in \cite{Ye17} and \cite{Duchi16} only hold for certain regions of $k$ and $\epsilon$ while the lower bound in this paper holds for all values of $k$ and $\epsilon$.

Furthermore, as an immediate consequence of our lower bound, we show that the schemes and the estimators proposed in \cite{Wang16,Ye17} are asymptotically optimal! In other words, the ratio between
the lower bound and the worst-case estimation loss of these schemes and estimators goes to $1$ when $n$ goes to infinity.

\subsection{Organization of the paper}
In Section~\ref{Sect:pre}, we formulate the problem and give a more detailed review of the existing results.
Section~\ref{Sect:ovr} is devoted to an overview of the main results of this paper and to illustrating the main ideas behind the proof.
Since the proof is very long and technical, we include a short Section~\ref{Sect:sketch}, where
we explain the argument in formal terms, while skipping many details. The complete proof is given in 
Section~\ref{Sect:Main} which (with Appendices) takes the most of the length of the paper.
In Section~\ref{Sect:FW}, we point out two possible directions for future research.

\section{Problem formulation and existing results}\label{Sect:pre}
\textbf{Notation:}
Let $\cX=\{1,2,\dots,k\}$ be the source alphabet and let $\mathbi{p}=(p_1,p_2,\dots,p_k)$ be a probability distribution on $\cX.$
Denote by $\Delta_k=\{\mathbi{p}\in \mathbb{R}^k: p_i\ge 0 \text{~for~} i=1,2,\dots,k, \sum_{i=1}^k p_i=1\}$ the $k$-dimensional probability simplex. Let $X$ be a random variable (RV) that takes values on $\cX$ according to $\mathbi{p}$, so that $p_i=P(X=i).$ Denote by $X^n=(X^{(1)},X^{(2)},\dots,X^{(n)})$ the vector formed of $n$ independent copies of the RV $X.$

\subsection{Problem formulation}
In the classical (non-private) distribution estimation problem, we are given direct access to i.i.d. samples
$\{X^{(i)}\}_{i=1}^n$ drawn according to some unknown distribution $\mathbi{p}\in \Delta_k.$ Our goal is to estimate $\mathbi{p}$ based on the samples \cite{Lehmann06}. We define an estimator $\hat{\mathbi{p}}$ as a function
$\hat{\mathbi{p}}:\cX^n \to \mathbb{R}^k,$ and assess its quality in terms of the worst-case risk (expected loss)
$$
\sup_{\mathbi{p}\in \Delta_k}  \underset{X^n\sim \mathbi{p}^n}{\mathbb{E}} \ell(\hat{\mathbi{p}}(X^n), \mathbi{p}),
$$
where $\ell$ is some loss function. The minimax risk is defined as the solution of the following saddlepoint problem:
$$
r_{k,n}^{\ell}:= \inf_{\hat{\mathbi{p}}} \sup_{\mathbi{p}\in \Delta_k} 
\underset{X^n\sim \mathbi{p}^n}{\mathbb{E}} \ell(\hat{\mathbi{p}}(X^n), \mathbi{p}).
$$

In the private distribution estimation problem, we can no longer access the raw samples $\{X^{(i)}\}_{i=1}^n.$ Instead, we estimate the distribution $\mathbi{p}$ from the privatized samples $\{Y^{(i)}\}_{i=1}^n,$ obtained by applying a privatization mechanism $\mathbi{Q}$ independently to each raw sample $X^{(i)}.$ A {\em privatization mechanism} (also called privatization scheme) $\mathbi{Q}:\cX\to\cY$ is simply a conditional distribution $\mathbi{Q}_{Y|X}.$ The
privatized samples $Y^{(i)}$ take values in a set $\cY$ (the ``output alphabet'') that does not have to be the same as $\cX.$

The quantities $\{Y^{(i)}\}_{i=1}^n$ are i.i.d. samples drawn according to the marginal distribution $\mathbi{m}$ given by
\begin{equation}\label{eq:defm}
\mathbi{m}(S)=\sum_{i=1}^k \mathbi{Q}(S|i)p_i
\end{equation}
 for any $S\in \sigma(\cY),$ where $\sigma(\cY)$ denotes an appropriate $\sigma$-algebra on $\cY.$
In accordance with this setting, the estimator $\hat{\mathbi{p}}$ is a measurable function $\hat{\mathbi{p}}:\cY^n\to \mathbb{R}^k.$
We assess the quality of the privatization scheme $\mathbi{Q}$ and the corresponding estimator $\hat{\mathbi{p}}$ by the worst-case risk 
$$
r_{k,n}^{\ell} (\mathbi{Q}, \hat{\mathbi{p}}) := \sup_{\mathbi{p}\in \Delta_k} 
\underset{Y^n\sim \mathbi{m}^n}{\mathbb{E}} \ell(\hat{\mathbi{p}}(Y^n), \mathbi{p}),
$$
where $\mathbi{m}^n$ is the $n$-fold product distribution and $\mathbi m$ is given by \eqref{eq:defm}.
Define the {\em minimax risk} of the privatization scheme $\mathbi{Q}$ as
  \begin{equation}\label{eq:riskQ}
r_{k,n}^{\ell} (\mathbi{Q}):= \inf_{\hat{\mathbi{p}}} r_{k,n}^{\ell} (\mathbi{Q}, \hat{\mathbi{p}}).
   \end{equation}
\begin{definition}
For a given $\epsilon>0,$
a privatization mechanism $\mathbi{Q}:\cX\to\cY$ is said to be {\em $\epsilon$-locally differentially private} if 
\begin{equation}\label{eq:defep}
\sup_{S\in\sigma(\cY)} \frac{\mathbi{Q}(Y\in S|X=x)}{\mathbi{Q}(Y\in S|X=x')} \le e^{\epsilon}
\text{~for all~} x,x'\in\cX.
\end{equation}
\end{definition}

Denote by $\cD_{\epsilon}$ the set of all $\epsilon$-locally differentially private mechanisms. Given a privacy level 
$\epsilon,$ we seek to find the optimal $\mathbi{Q}\in\cD_{\epsilon}$ with the smallest possible minimax risk among all the 
$\epsilon$-locally differentially private mechanisms. Accordingly, define the $\epsilon$-private minimax risk as
 \begin{equation}
r_{\epsilon,k,n}^{\ell} := \inf_{\mathbi{Q}\in\cD_{\epsilon}}r_{k,n}^{\ell} (\mathbi{Q}).
  \label{eq:rekn}
  \end{equation}
  As already mentioned, we will limit ourselves to $\ell=\ell_2^2.$
  
\vspace*{.1in}
\noindent{\bf Main Problem:} \emph{Suppose that the cardinality $k$ of the source alphabet is known to the estimator. We would like to find the asymptotic growth rate of $r_{\epsilon,k,n}^{\ell_2^2}$ as $n\to\infty$ and to construct an asymptotically optimal privatization mechanism and a corresponding  estimator of $\mathbi{p}$ from the privatized samples.}
  
\vspace*{.1in}  It is this problem that we address---and resolve---in this paper. Specifically, we prove a lower bound
on $r_{\epsilon,k,n}^{\ell_2^2}$, which implies that the mechanism and the corresponding estimator
proposed in \cite{Ye17} are asymptotically optimal for the private estimation problem.

\subsection{Previous results}
In this section we briefly review known results that are relevant to our problem. In Sect.~\ref{sec:existing} we mentioned
several papers that have considered it, viz., \cite{Warner65,Duchi13a,Erlingsson14,Kairouz14,Kairouz16,Wang16,Duchi16}. 
In this section we discuss only the results of \cite{Ye17} since they subsume the (earlier) results of the mentioned references, and
since they are formulated in the form convenient for our presentation.

Let $\cD_{\epsilon,F}$ be the set of $\epsilon$-locally differentially private schemes with finite output alphabet. 
Let
  \begin{equation}\label{eq:DES}
\cD_{\epsilon,E}=\biggl\{ \mathbi{Q}\in\cD_{\epsilon,F}: 
\frac{\mathbi{Q}(y|x)}{\min_{x'\in\cX}\mathbi{Q}(y|x') } \in \{1,e^{\epsilon}\}
\text{~for all~} x\in\cX \text{~and all~} y\in\cY \biggr\}.
  \end{equation}
 In \cite[Theorem IV.5]{Ye17}, we have shown that
\begin{equation}\label{eq:red}
r_{\epsilon,k,n}^{\ell_2^2} = \inf_{\mathbi{Q}\in\cD_{\epsilon,E}} r_{k,n}^{\ell_2^2} (\mathbi{Q}).
\end{equation}
As a result, below we limit ourselves to schemes $\mathbi{Q}\in\cD_{\epsilon,E}$ in this paper. 
For such schemes, since the output alphabet is finite, we can write the marginal distribution $\mathbi{m}$  in \eqref{eq:defm} 
as a vector $\mathbi{m}=(\sum_{j=1}^k p_j \mathbi{Q}(y|j), y\in\cY).$ We will also use the shorthand notation
$\mathbi{m}=\mathbi{p}\mathbi{Q}$ to denote this vector.

In \cite{Ye17}, we introduced a family of privatization schemes
which are parameterized by the integer $d\in\{1,2,\dots,k-1\}.$ Given $k$ and $d,$ let the output alphabet be 
$\cY_{k,d}=\{y\in \{0,1\}^k: \sum_{i=1}^k y_i=d\},$ so $|\cY_{k,d}|=\binom{k}{d}.$ 

\begin{definition} [\cite{Ye17}] Consider the following privatization scheme:
   \begin{equation}\label{eq:defQ}
\mathbi{Q}_{k,\epsilon,d}(y|i)=\frac{e^\epsilon y_i+(1-y_i)}{\binom{k-1}{d-1}e^{\epsilon}+\binom{k-1}{d}} 
    \end{equation}
for all $y\in\cY_{k,d}$ and all $i\in\cX.$
The corresponding empirical estimator of $\mathbi{p}$ under $\mathbi{Q}_{k,\epsilon,d}$ is defined as follows:
\begin{equation}\label{eq:emp}
\hat{p_i}=\Big(\frac{(k-1)e^{\epsilon}+\frac{(k-1)(k-d)}{d}}{(k-d)(e^{\epsilon}-1)}\Big)\frac{T_i}{n}
-\frac{(d-1)e^{\epsilon}+k-d}{(k-d)(e^{\epsilon}-1)},
\end{equation}
where $T_i=\sum_{j=1}^n Y_i^{(j)}$ is the number of privatized samples whose $i$-th coordinate is $1$.
\end{definition}

It is easy to verify that $\mathbi{Q}_{k,\epsilon,d}$ is $\epsilon$-locally differentially private.
The estimation loss under $\mathbi{Q}_{k,\epsilon,d}$ and the empirical estimator is calculated in the following proposition.
\begin{proposition}\label{prop:risks}{\rm\cite[Prop.~III.1]{Ye17}} Suppose that the privatization scheme is 
 $\mathbi{Q}_{k,\epsilon,d}$ and the empirical estimator is given by \eqref{eq:emp}.
Let $\mathbi{m}=\mathbi{p}\mathbi{Q}_{k,\epsilon,d}.$
For all $\epsilon, n$ and $k,$ we have that
    \begin{equation}\label{eq:L2}
    \underset{Y^n\sim \mathbi{m}^n}{\mathbb{E}} \ell_2^2(\hat{\mathbi{p}}(Y^n),\mathbi{p})=\frac{1}{n}
     \Big(\frac{(d(k-2)+1)e^{2\epsilon}}{(k-d)(e^{\epsilon}-1)^2} +\frac{2(k-2)e^{\epsilon}}{(e^{\epsilon}-1)^2}
     +\frac{(k-2)(k-d)+1}{d(e^{\epsilon}-1)^2} -\sum_{i=1}^k p_i^2 \Big).
    \end{equation}
\end{proposition}

The sum $\sum_{i}p_i^2$ is maximized for $\mathbi{p}_U=(1/k,1/k,\dots,1/k)$, so the worst-case distribution is
the uniform one. Substituting $\mathbi{p}_U$ in \eqref{eq:L2}, we obtain
  \begin{equation}\label{eq:rd}
r_{k,n}^{\ell_2^2} (\mathbi{Q}_{k,\epsilon,d}, \hat{\mathbi{p}})  =
\underset{Y^n\sim \mathbi{m}_U^n}{\mathbb{E}} \ell_2^2(\hat{\mathbi{p}}(Y^n),\mathbi{p}_U)=
\frac{(k-1)^2}{nk(e^{\epsilon}-1)^2} \frac{(d e^{\epsilon} + k-d)^2 }{d(k-d)}.
  \end{equation}

Given $k$ and $\epsilon$, define
\begin{equation}\label{eq:dstar}
d^\ast = d^\ast(k,\epsilon)
:= \argmin_{1\le d \le k-1}  \frac{(d e^{\epsilon} + k-d)^2 }{d(k-d)},
\end{equation}
where the ties are resolved arbitrarily.
We obtain 
  \begin{equation}\label{eq:rdstar}
r_{k,n}^{\ell_2^2} (\mathbi{Q}_{k,\epsilon,d^\ast}, \hat{\mathbi{p}}) = 
\min_{1\le d \le k-1} r_{k,n}^{\ell_2^2} (\mathbi{Q}_{k,\epsilon,d}, \hat{\mathbi{p}}).
  \end{equation}
By differentiation in \eqref{eq:dstar} we find that $d^\ast$ can only take two possible values given in the next proposition.
\begin{proposition} 
$$
d^\ast=\lceil k/(e^{\epsilon}+1) \rceil \text{ or } \lfloor k/(e^{\epsilon}+1)\rfloor.
$$
\end{proposition}
Therefore, when $k/(e^{\epsilon}+1) \le 1$, $d^\ast=1$; when $k/(e^{\epsilon}+1) > 1$, a simple comparison can determine the value of $d^\ast$.

Clearly, $r_{k,n}^{\ell_2^2} (\mathbi{Q}_{k,\epsilon,d^\ast}, \hat{\mathbi{p}})$ serves as an upper bound on
$r_{\epsilon, k,n}^{\ell_2^2}$. In \cite{Ye17}, we also proved an order optimal lower bound on 
$r_{\epsilon, k,n}^{\ell_2^2}$ in the regime $e^{\epsilon} \ll k$ and $n$ large enough.
Combining the upper and lower bounds, we obtain the following theorem.

\begin{theorem}\label{Thm:last}{\rm \cite{Ye17}} Let $e^{\epsilon} \ll k$. For $n$ large enough,
$$
r_{\epsilon, k,n}^{\ell_2^2} =  \Theta \Big( \frac{k e^{\epsilon}}{n (e^{\epsilon}-1)^2} \Big).
$$
\end{theorem}

\section{Overview of the results and the main ideas of the proof}\label{Sect:ovr}
\subsection{Main result of the paper}
Let 
  \begin{equation}\label{eq:Mke}
     M(k,\epsilon):= \frac{(k-1)^2}{k(e^{\epsilon}-1)^2}\frac{(d^\ast e^{\epsilon} + k-d^\ast)^2 }{d^\ast(k-d^\ast)}
  \end{equation}
where $d^\ast$ is defined in \eqref{eq:dstar}.
Note that 
$r_{k,n}^{\ell_2^2} (\mathbi{Q}_{k,\epsilon,d^\ast}, \hat{\mathbi{p}})
=\frac{1}{n}M(k,\epsilon)$.
     \begin{theorem}\label{Thm:Main}
For every $k$ and $\epsilon$, there are a positive constant $C(k,\epsilon)>0$ and an integer $N(k,\epsilon)$ such that when $n \ge N(k,\epsilon)$,
\begin{equation}\label{eq:cnlb}
r_{\epsilon, k,n}^{\ell_2^2} \ge
\frac 1n M(k,\epsilon)
 - \frac{C(k,\epsilon)} {n^{14/13} }.
\end{equation}
This result together with \eqref{eq:rdstar} and \eqref{eq:rd} implies that
\begin{equation}\label{eq:evob}
r_{\epsilon, k,n}^{\ell_2^2} 
= r_{k,n}^{\ell_2^2} (\mathbi{Q}_{k,\epsilon,d^\ast}, \hat{\mathbi{p}})  -   O ( n^{-14/13} )  .
\end{equation}
\end{theorem}
This theorem completely determines the dominant term of $r_{\epsilon, k,n}^{\ell_2^2}$. It also gives an upper bound on the order of the remainder term. Moreover, this theorem also implies that our scheme
$\mathbi{Q}_{k,\epsilon,d^\ast}$ and the empirical estimator $\hat{\mathbi{p}}$ defined in \eqref{eq:emp}, both proposed in \cite{Ye17}, are asymptotically optimal for the Main Problem.

\subsection{Main ideas of the proof}
In this subsection we illustrate the main ideas that lead to determining the {\em dominant term} of $r_{\epsilon, k,n}^{\ell_2^2}$.
In view of \eqref{eq:rd}--\eqref{eq:rdstar} we need to show that
$
\lim_{n\to \infty} n r_{\epsilon,k,n}^{\ell_2^2} = M(k,\epsilon).
$
Clearly, $r_{\epsilon, k,n}^{\ell_2^2} \le
 r_{k,n}^{\ell_2^2} (\mathbi{Q}_{k,\epsilon,d^\ast}, \hat{\mathbi{p}})$
for all $n\ge 1$, implying that
$
\limsup_{n\to \infty} n r_{\epsilon,k,n}^{\ell_2^2} \le 
M(k,\epsilon).
$
Therefore, we only need to show the lower bound
\begin{equation}\label{eq:oblb}
\liminf_{n\to \infty} n r_{\epsilon,k,n}^{\ell_2^2} =
\liminf_{n\to \infty} n \inf_{\mathbi{Q} \in \cD_{\epsilon,E}} r_{k,n}^{\ell_2^2} (\mathbi{Q}) \ge 
M(k,\epsilon).
\end{equation}

Lower bounds on $r_{\epsilon,k,n}^{\ell_2^2}$ can be derived using Assouad's method; see \cite{Duchi16,Ye17}. 
More specifically, we can choose a finite subset of the probability simplex and assume that the probability distributions only come from 
this finite set, thereby reducing the original estimation problem to a hypothesis testing problem. This approach enables
us to derive the correct scaling order of $r_{\epsilon, k,n}^{\ell_2^2}$; see Theorem \ref{Thm:last}. Deriving the correct constant in front of the main
term is more problematic because {Assouad's method relies on reducing} a continuous domain (the probability simplex) to a finite set.

In this paper, we use a different approach to obtain an asymptotically tight lower bound in \eqref{eq:oblb}.
Since the worst-case estimation loss is always lower bounded by the average estimation loss, the minimax risk $r_{\epsilon, k,n}^{\ell_2^2}$ can be bounded below by the Bayes estimation loss. More specifically, we assume that the probability distributions are chosen uniformly from a small neighborhood of the uniform distribution $\mathbi{p}_U$. Surprisingly, the lower bound on the Bayes estimation loss turns out to be asymptotically the same as the worst-case estimation loss of our scheme and estimator. In other words, the ratio between these two quantities goes to $1$ when $n$ goes to infinity.

In order to obtain the lower bound on the Bayes estimation loss, we refine a classical method in asymptotic statistics, namely, local asymptotic normality (LAN) of the posterior distribution \cite{LeCam12,Ibrag81,Hajek72}.
We briefly describe the implications of LAN for our problem and explain why the classical approach does not directly apply to our problem.
Our objective is to estimate $\mathbi{p}$ from the privatized samples $Y^n$. In the Bayesian setup, we assume that $\mathbi{p}$ is drawn uniformly from $\cP$, a very small neighborhood of $\mathbi{p}_U$. Let 
$\mathbi{P}=[P_1,P_2,\dots,P_k]$ denote the random vector that corresponds to $\mathbi{p}$. Applying the LAN method of \cite{LeCam12,Ibrag81,Hajek72}, one can show that
when the radius of $\cP$ is of order $n^{-1/2}$, with large probability the conditional distribution of $\mathbi{P}$ given $Y^n$ 
approaches a jointly Gaussian distribution as $n$ goes to infinity, and the covariance matrix $\Sigma=\Sigma(n,\mathbi{Q})$ of this Gaussian distribution is completely determined by $n$ and the privatization scheme $\mathbi{Q}$. Note that $\Sigma$ is independent of the value of $Y^n$. 
We further note that the top-left $(k-1)\times (k-1)$ submatrix of $\Sigma$ is the inverse of the Fisher information matrix computed with
respect to the parameters $p_1,p_2,\dots,p_{k-1}$ (a detailed discussion of this issue appears in Appendix~\ref{ap:LAN}).
It is clear that the trace of $\Sigma$ serves an asymptotic lower bound on $r_{k,n}^{\ell_2^2} (\mathbi{Q})$. 
At the same time, for all $\mathbi{Q} \in \cD_{\epsilon,E}$, we can show that 
$$
\tr(\Sigma) \ge r_{k,n}^{\ell_2^2} (\mathbi{Q}_{k,\epsilon,d^\ast}, \hat{\mathbi{p}}),
$$
where $\mathbi{Q}_{k,\epsilon,d^\ast}$ and $\hat{\mathbi{p}}$ are given respectively in \eqref{eq:defQ} and \eqref{eq:emp}.
Therefore,
   $$
\liminf_{n\to \infty} n r_{k,n}^{\ell_2^2} (\mathbi{Q}) \ge M(k,\epsilon)
\text{~for all~} \mathbi{Q} \in \cD_{\epsilon,E}.
   $$
This inequality gives a strong intuition of why \eqref{eq:oblb} should hold. However, it does not imply \eqref{eq:oblb} because a pointwise asymptotic lower bound does not imply a uniform asymptotic lower bound. For this reason, we cannot directly use
the classical approach and must develop more delicate arguments that prove \eqref{eq:oblb} and \eqref{eq:evob}.
Another feature of our approach worth mentioning is that, unlike the methods in \cite{LeCam12,Ibrag81,Hajek72}, our proof is completely elementary.

\section{Sketch of the proof of Theorem~\ref{Thm:Main}}\label{Sect:sketch}
In the previous section we explained a general approach to the proof of   \eqref{eq:oblb}. The formal
proof that we will present is rather long and technical. For this reason, in this section we outline a ``road map'' to help the 
readers to follow our arguments. 

For reader's convenience we made a short table of our main notation;  see Table~\ref{definitions}.
Most, although not all, definitions from this table are also given in the main text.

The following conventions about our notation are made with reference to Table~\ref{definitions}. The vector $\mathbi{u}$ is chosen to be $(k-1)$-dimensional ($\mathbi{u}$ is a function of $\mathbi{p},$ but we omit the dependence for notational simplicity). At the same time, using the normalization condition, we
define  $u_k=u_k(\mathbi{u}):=-\sum_{j=1}^{k-1} u_j.$
Clearly, $u_k=p_k-1/k.$ The same convention is used for the random vector $\mathbi{U}$ and for the vector 
$\tilde{\mathbi{u}}$ both of which appear later in the paper.

The following obvious relations will be repeatedly used in the proof:
  \begin{equation}\label{eq:qu}
  \sum_{i=1}^L q_i = \sum_{i=1}^Lq_{ij}=1 \text{ for all }j\in[k], \quad \sum_{j=1}^k u_j=0, \quad \sum_{i=1}^L t_i(y^n)=n.
  \end{equation}

\begin{table}[H]
\centering
\large
\caption{\large Main notation}\label{definitions}
\begin{tabular}{| m{5cm} | m{10cm} |}
\hline
$\cX=\{1,2,\dots,k\}$ & source alphabet \\ \hline
$\mathbi{p}=(p_1,\dots,p_k)$ & distribution on $\cX$, where $p_j= P(X=j)$\\ \hline
$\cD_{\epsilon,E}$ & set of privatization mechanisms of the form \eqref{eq:DES}\\ \hline
$\mathbi{u}=(u_1,\dots,u_{k-1})$ & difference between the first $k-1$ coordinates of $\mathbi{p}$\\
$u_j:=p_j-\frac 1k, j=1,\dots,k-1$ & and the uniform distribution\\ \hline
$u_k=u_k(\mathbi{u}):=-\sum_{j=1}^{k-1} u_j$ & 
by definition $u_k=p_k - 1/k$ \\ \hline
$L'$ &  cardinality of the original output alphabet \\ \hline
$\cY$ &  original output alphabet. WLOG we assume $\cY=\{1,2,\dots,L'\}$ \\ \hline
$L$ &the number of equivalence classes $\{A_1,A_2,\dots,A_L\}$ in $\cY$ \\ \hline
$\{A_1,A_2,\dots,A_L\}$ &  equivalent output alphabet after symbol merging \\ \hline
$\mathbi{Q}=(\mathbi{Q}(i|j))_{i\in[L'],j\in[k]}$ 
&privatization mechanism (conditional distribution)\\ \hline
$q_{ij}, i\in[L],j\in[k]$ &  conditional probability of observing output symbols in $A_i$ if the raw sample is $j$ 
\remove{Here $A_i$ is the output symbol after the merging process.}\\ \hline
$q_i:=\frac 1k\sum_j q_{ij}, i\in[L]$ & 
by definition $q_i = P(Y\in A_i)$ when $\mathbi{p}$ is the uniform distribution over $\cX$ \\ \hline
$(t_i(y^n), i=1,\dots, L)$ &composition of the observed vector, 
$t_i(y^n)$ is the number of occurrences of symbol $A_i$ in $y^n$\\ \hline
$\mathbi{v} = (v_1,\dots,v_L)$  &  \\
$v_i:=t_i(y^n)-nq_i, i\in[L]$ & \\ \hline
$\mathbi {w}=\mathbi{w}(\mathbi{v},\mathbi{Q}) \in \mathbb{R}^{k-1}$ &vector with coordinates $\sum_{i=1}^L\frac{(q_{im}-q_{ik})v_i}{q_i},m=1,\dots,k-1$\\ \hline
$\mathbi{U}=(U_1,U_2,\dots,U_{k-1})$ & random vector corresponding to 
$\mathbi{u}=(u_1,\dots,u_{k-1})$\\ \hline
$U_k = U_k(\mathbi{U}) := -\sum_{j=1}^{k-1}U_j$ & random variable corresponding to $u_k$ \\ \hline
$\mathbi{V}=(V_1,\dots,V_L)$ & random vector corresponding to $\mathbi{v}=(v_1,\dots,v_L)$\\ \hline
$B(\alpha)$ &ellipsoid \eqref{eq:ellipsoid} of ``radius'' $\alpha$;\\
&$B_1:=B(n^{-5/13}),$  
 $B_2:=B(n^{-5/13}-3n^{-6/13}/\delta_0)$  \\
 \hline
\end{tabular}
\end{table} 

  Below we study distributions supported on ellipsoids, and we use the following generic notation: Given $\alpha>0,$ let us define an ellipsoid
  \begin{equation}\label{eq:ellipsoid}
B(\alpha)=\Big\{\mathbi{u}\in \mathbb{R}^{k-1}: 
 \sum_{i=1}^{k-1}  u_i^2 + \Big(\sum_{i=1}^{k-1}  u_i\Big)^2 < \alpha^2 \Big\}.
  \end{equation}
For future use we note that  $\mathbi{u}\in B(\alpha)$ if and only
$\mathbi{u}^T (I+J)\mathbi{u}<\alpha^2,$
where $I$ is the identity matrix and $J$ is the all-ones matrix.

\vspace*{.2in}
On account of \eqref{eq:red}, to prove \eqref{eq:cnlb} it suffices to show that for every $k$ and $\epsilon$, there are a positive constant $C(k,\epsilon)>0$ and an integer $N(k,\epsilon)$ such that when $n \ge N(k,\epsilon)$,
\begin{equation}\label{eq:indv}
\inf_{\mathbi{Q}\in \cD_{\epsilon,E}}r_{k,n}^{\ell_2^2}(\mathbi{Q}) \ge \frac 1n
M(k,\epsilon)
 - \frac{C(k,\epsilon)} {n^{14/13} }
\end{equation}
In the rest of this paper we will prove \eqref{eq:indv}. It is important to note that, because of the infimum on the distribution $\mathbi{Q}$, the constants
$C(k,\epsilon)$ and $N(k,\epsilon)$ should not depend on it. It is for this reason that the classical LAN approach does not directly apply. Below we provide more details about this.

\subsection{Output alphabet reduction}\label{Sect:OAR}
As mentioned above, we use Bayes estimation loss to bound $r_{\epsilon, k,n}^{\ell_2^2}$ below, and we assume that the probability distributions are chosen uniformly from a small neighborhood of the uniform distribution $\mathbi{p}_U$. Recalling the definition
of the vector $\mathbi{u}$ and $u_k$ in Table~\ref{definitions}, we note that estimating $\mathbi{p}$ is equivalent to estimating $(u_1,\dots,u_k).$
Let $\mathbi{U}$ be the random vector that corresponds to $\mathbi{u}$. We assume that $\mathbi{U}$ is uniformly distributed over the ellipsoid
   \begin{equation}\label{eq:B1}
B_1=B\Big(\frac{1}{n^{5/13}}\Big).
  \end{equation}

Suppose that the size of the original output alphabet $\cY$ is some integer $L'$, which is independent of $k$ and $\epsilon$, and can take an arbitrarily large value.
Recall that our objective is \eqref{eq:indv}.
 If $L'$ enters the estimates of the estimation loss, then we can not bound it by a function of $k$ and $\epsilon$, which
is something we would like to avoid. This can be done using the following simple observation.
Without loss of generality, we assume that $\cY = \{1,2,\dots,L'\}$.
Since $\mathbi{Q}\in \cD_{\epsilon,E}$ (see \eqref{eq:DES}), for every $i\in\{1,\dots,L'\}$,
the vector $(\mathbi{Q}(i|j),j=1,\dots,k)$ is proportional to one of the vectors in the set $\{1,e^\epsilon\}^k$. 
It is easy to see that we can merge into one symbol all the output symbols that correspond to proportional vectors, thereby
reducing the size of the output alphabet without affecting the Bayes estimation loss. Suppose that, upon merging all
such symbols, the size of the output alphabet becomes $L$. Clearly, $L\le 2^k$, and we will henceforth assume that
the output alphabet is $\{A_1,A_2,\dots,A_L\}$ and let $q_{ij}$ denote the conditional probability of observing $A_i$ if the raw sample is $j\in[k]$.

\subsection{Gaussian approximation of the posterior pdf $f_{U|Y^n}$}
\subsubsection{Approximation}
For $i=1,2,\dots,L$,
let $q_i:=\frac{1}{k}\sum_{j=1}^k q_{ij}$, and let $t_i=t_i(y^n)$ be the number of times that symbol $A_i$ appears in $y^n$. Define $v_i:=v_i(y^n)=t_i(y^n)-n q_i$ for $i=1,\dots,L.$  Then for $\mathbi{u}\in B_1$, we have
   \begin{align}
f_{\mathbi{U}|Y^n}(\mathbi{u}|y^n) &\;\propto\;
\prod_{i=1}^{L} (q_i + \sum_{j=1}^k u_j q_{ij} )^{t_i} \nonumber\\ 
&\;\propto\;
\exp\Big( \sum_{i=1}^L  \Big( (n q_i + v_i)
\log \Big( 1 + \sum_{j=1}^k \frac{u_j q_{ij}}{q_i} \Big) \Big) \Big). \label{eq:f}
    \end{align}
Since our objective is to estimate $u_1,u_2,\dots,u_k$, we view all the factors that {do not contain} $\mathbi{u}$ as constants in the formula above. Let us introduce a notation for the exponent of the right-hand side:
  \begin{equation}\label{eq:g}
g(\mathbi{u},y^n) = \sum_{i=1}^L  \Big( (n q_i + v_i)
\log \Big( 1 + \sum_{j=1}^k \frac{u_j q_{ij}}{q_i} \Big) \Big).
  \end{equation}
Let $\mathbi{V}$ be a random vector corresponding to the vector $\mathbi{v}=(v_1,\dots,v_L)$. Since $P(Y=A_i)=\sum_{j=1}^k p_j q_{ij},$ the expectation of $V_i$ equals
$$
\mathbb{E} V_i = n \sum_{j=1}^k p_j q_{ij}-\frac nk \sum_{j=1}^k q_{ij}=\sum_{j=1}^k n u_j q_{ij}.
$$
Assuming that $\mathbi{u}\in B_1$, we therefore conclude that
$\mathbb{E} V_i = O(n^{8/13})$. By definition, $\Var(V_i)=\Var(t_i(Y^n))$ for all $i$. 
 According to the Central Limit Theorem, when $n$ is large, $\Var(V_i)=O(n)$.
As a consequence, when $n$ is large, with large probability, $V_i = O(n^{8/13})$.
{This fact together with the relation}
  $$
  \log \Big( 1 + \sum_{j=1}^k \frac{u_j q_{ij}}{q_i} \Big)\approx
  \sum_{j=1}^k \frac{u_j q_{ij}}{q_i} - \frac{1}{2}\Big(\sum_{j=1}^k \frac{u_j q_{ij}}{q_i}\Big)^2
  $$
{gives us the following approximation of
$g(\mathbi{u},y^n)$:}
$$
g(\mathbi{u},y^n) \approx 
\sum_{i=1}^L  v_i \sum_{j=1}^k \frac{u_j q_{ij}}{q_i}
-\frac{1}{2} \sum_{i=1}^L  n q_i  \Big(\sum_{j=1}^k \frac{u_j q_{ij}}{q_i} \Big)^2 
 = - \frac{1}{2} h_{\mathbi{v}}(\mathbi{u})
 + \sum_{i=1}^L \frac{v_i^2}{2n q_i},
$$
where for $\mathbi{v} = (v_1,v_2,\dots,v_L)$, the function $h_{\mathbi{v}}: \mathbb{R}^{k-1} \to \mathbb{R}$ is defined as
  \begin{equation}\label{eq:defh}
h_{\mathbi{v}}(\mathbi{u}) = \sum_{i=1}^L \frac{n}{q_i} \Big(\sum_{j=1}^k u_j q_{ij} - \frac{v_i}{n} \Big)^2 \text{~for all~} \mathbi{u} \in \mathbb{R}^{k-1}.
  \end{equation}
Thus, for large $n$ the density function of the posterior distribution is approximately given by
   $$
   f_{\mathbi{U}|Y^n}(\mathbi{u}|y^n) \;\propto\;  \exp (-\frac{1}{2}h_{\mathbi{v}}(\mathbi{u})),
   $$
    and $h_{\mathbi{v}}(\mathbi{u})$ is a quadratic function of $\mathbi{u}$. Had $\mathbi{u}$ been taking values in the entire space $\mathbb{R}^{k-1}$, we would be able to conclude that the posterior distribution of $\mathbi{U}$ given $Y^n$ is approximately Gaussian. 
Supposing that this is the case, define the matrix
$$
\Phi = \Phi(n,\mathbi{Q}) := \sum_{i=1}^L
{\frac{n}{q_i}}\Big((q_{i,1} - q_{ik}),
\dots, (q_{ik-1} - q_{ik}) \Big)^T
\Big((q_{i,1} - q_{ik}),
\dots, (q_{i,k-1} - q_{ik}) \Big),
$$
and the vector $\mathbi{w} \in \mathbb{R}^{k-1}$ 
$$
\mathbi{w}=\mathbi{w}(\mathbi{v},\mathbi{Q}) := \Big( \sum_{i=1}^L  \frac{
(q_{i,1}-q_{ik}) v_i  }{q_i}, \sum_{i=1}^L  \frac{
(q_{i,2}-q_{ik}) v_i  }{q_i}, \dots, \sum_{i=1}^L  \frac{
(q_{ik-1}-q_{ik}) v_i  }{q_i} \Big)^T.
$$
Then the covariance and the vector of means of this Gaussian distribution are given by $\Phi^{-1}$ and $\Phi^{-1}\mathbi{w}.$
Note that $\Phi$ is independent of the value of $Y^n$. {At the same time, since $\mathbi{w}$ depends on the value of $\mathbi{v}$, and $\mathbi{v}$ is a function of $y^n$, the mean vector $\Phi^{-1}\mathbi{w}$ does vary with the realization of $Y^n$.}

\vspace*{.1in}\subsubsection{Large deviations}
The above argument does not directly apply because $\mathbi{u}$ is limited to the ellipsoid $B_1$. In order to claim that the posterior distribution can be indeed approximated as Gaussian, we need to show that the entire mass is concentrated in $B_1,$ i.e., that under the Gaussian distribution $\cN(\Phi^{-1}\mathbi{w}, \Phi^{-1})$ the probability $P(\mathbi{U}\not\in B_1)\approx0.$

To build intuition, let us consider the one-dimensional case. Let $h(x)=\frac{(x-\mu)^2}{2\sigma^2}$ be the (absolute value of the) exponent of the Gaussian pdf. By the Chernoff bound, for $X\sim \cN(\mu,\sigma^2)$ we have
   \begin{equation}\label{eq:Chernoff}
P(|X-\mu|\ge t) \le 2 \exp(-\frac{t^2}{2\sigma^2}).
  \end{equation}
This bound immediately implies that for any $\tilde{x} \in \mathbb{R}$
$$
P(\sqrt{h(X)}-\sqrt{h(\tilde{x})} \ge t) \le 2 \exp(-t^2).
$$
The situation is similar in the multi-dimensional case. Indeed, one can show that for $\mathbi{U} \sim \cN(\Phi^{-1}\mathbi{w}, \Phi^{-1})$ we have the following inequality: for any $\tilde{\mathbi{u}}\in\mathbb{R}^{k-1}$ 
and any $\alpha>0$
$$
P(\sqrt{h_{\mathbi{v}}(\mathbi{U})} - \sqrt{h_{\mathbi{v}}(\tilde{\mathbi{u}})} > n^\alpha)
\le \exp(-n^{\alpha})
$$
for a large enough $n.$
Now our task is to show that for almost all $y^n\in\cY^n$, we can find a $\tilde{\mathbi{u}}=\tilde{\mathbi{u}}(y^n)$ such that
\begin{equation}\label{eq:osk}
B_1^c \subseteq \{\mathbi{u}\in\mathbb{R}^{k-1}:
\sqrt{h_{\mathbi{v}}(\mathbi{u})} - \sqrt{h_{\mathbi{v}}(\tilde{\mathbi{u}})} > n^\alpha \}
\end{equation}
for some $\alpha>0$, or more precisely, that \eqref{eq:osk} holds for all $y^n$ in a subset $E_2\subseteq \cY^n$ such that $P(Y^n\in E_2)\approx 1$.

Toward that end, we define another ellipsoid
   \begin{equation}\label{eq:B2}
B_2=B\Big( \frac{1}{n^{5/13}} - \frac{3/\delta_0}{n^{6/13}} \Big)
   \end{equation}
(see \eqref{eq:ellipsoid}), where $\delta_0$ is a constant which will be specified later.
Observe that the ratio between the radii of $B_2$ and $B_1$ approaches 1 when $n$ is large. Thus, 
$P(\mathbi{u} \in B_2)\approx P(\mathbi{u} \in B_1),$ and since $P(\mathbi{u}\in B_1)=1,$ we have $P(\mathbi{u} \in B_2)\approx 1$.
Conditional on the event $\mathbi{U} = \tilde{\mathbi{u}} \in B_2$, we have\footnote{Recall the convention that $u_k=-\sum_{i=1}^{k-1}u_i$ and $\tilde{u}_k=-\sum_{i=1}^{k-1} \tilde{u}_i$.} 
   $$
\mathbb{E}\Big(\sum_{i=1}^L \frac{n}{q_i} \Big(\sum_{j=1}^k \tilde{u}_j q_{ij} - \frac{V_i}{n} \Big)^2 \Big)  = \frac{1}{n} \sum_{i=1}^L \frac{1}{q_i} \Var\Big(t_i(Y^n) \Big) = O(1),
   $$
so for large $n$
   $$
\sqrt{\sum_{i=1}^L \frac{n}{q_i} \Big(\sum_{j=1}^k \tilde{u}_j q_{ij} - \frac{V_i(Y^n)}{n} \Big)^2 }
<  n^{1/26}
$$
with large probability.
We can phrase this as follows: conditional on $\mathbi{U} \in B_2$, for almost all $y^n\in\cY^n$  we can find $\tilde{\mathbi{u}} =\tilde{\mathbi{u}}(y^n) \in B_2$ such that
   $$
\sqrt{\sum_{i=1}^L \frac{n}{q_i} \Big(\sum_{j=1}^k \tilde{u}_j q_{ij} - \frac{v_i(y^n)}{n} \Big)^2 }
<  n^{1/26}.
   $$
Since $P(\mathbi{U} \in B_2)\approx 1$, this is the same as saying that for almost all $y^n\in\cY^n,$  we can find $\tilde{\mathbi{u}} =\tilde{\mathbi{u}}(y^n) \in B_2$ such that
\begin{equation}\label{eq:spl1}
\sqrt{h_{\mathbi{v}}(\tilde{\mathbi{u}})} 
= \sqrt{\sum_{i=1}^L \frac{n}{q_i} \Big(\sum_{j=1}^k \tilde{u}_j q_{ij} - \frac{v_i(y^n)}{n} \Big)^2 }
<  n^{1/26}.
\end{equation}
By the triangle inequality, for any $\tilde{\mathbi{u}} \in B_2$ and any $\mathbi{u}\in B_1^c$, we have 
\begin{equation}\label{eq:skt1}
\sqrt{\sum_{i=1}^k (u_i - \tilde{u}_i)^2} > \frac{3/\delta_0}{n^{6/13}}.
\end{equation}
Our next goal is to use this inequality to bound below the quantity $\big(\sum_{i=1}^L \frac{1}{q_i} 
(\sum_{j=1}^k (u_j - \tilde{u}_j) q_{ij} ^2\big)^{1/2}.$ Let us introduce the quantity
  \begin{equation}\label{eq:skt2}
\delta=\delta(\mathbi{Q}) := \min_{{\mathbi{u}\in\mathbb{R}^{k-1}:\, \sum_{i=1}^k u_i^2=1}}
\Big(\sum_{i=1}^L \frac{1}{q_i} \Big(\sum_{j=1}^k u_j q_{ij} \Big)^2\Big)^{1/2}
  \end{equation}
Intuitively, $\delta$ measures how well $\mathbi{Q}$ can distinguish between different $\mathbi{u}$'s. 
Our argument proceeds differently depending on whether $\delta\ge \delta_0$ or not.

If $\delta$ is small, then there is a pair $\mathbi{u}$ and $\mathbi{u}'$ that are well-separated from each other 
by the $\ell_2$ distance, but the posterior probabilities of $\mathbi{u}$ and $\mathbi{u}'$ given $y^n$ are close to each other. 
Thus, the case $\delta<\delta_0$ for a sufficiently small constant $\delta_0$ can be handled
by a straightforward application of the Le Cam method, and the main obstacle is represented by the opposite case.

For $\delta\ge \delta_0$, according to \eqref{eq:skt1} and \eqref{eq:skt2}, for any $\tilde{\mathbi{u}} \in B_2$ and any $\mathbi{u}\in B_1^c$,
\begin{equation}\label{eq:spl2}
\begin{aligned}
\sqrt{\sum_{i=1}^L \frac{n}{q_i} \Big(\sum_{j=1}^k \tilde{u}_j q_{ij} - u_j q_{ij} \Big)^2 }
& = n^{1/2} \sqrt{\sum_{i=1}^L \frac{1}{q_i} \Big(\sum_{j=1}^k (\tilde{u}_j - u_j) q_{ij} \Big)^2 } \\
& \ge n^{1/2} \delta \sqrt{\sum_{i=1}^k (u_i - \tilde{u}_i)^2}\\
&\ge n^{1/2} \delta_0 \frac{3/\delta_0}{n^{6/13}}\\
& = 3 n^{1/26}.
\end{aligned}
\end{equation}
Combining \eqref{eq:spl1} and \eqref{eq:spl2} and using the triangle inequality, we conclude that for almost all $y^n\in\cY^n$
and any $\mathbi{u}\in B_1^c$
\begin{equation}\label{eq:spl3}
\sqrt{h_{\mathbi{v}}(\mathbi{u})} =
\sqrt{\sum_{i=1}^L \frac{n}{q_i} \Big(\sum_{j=1}^k u_j q_{ij} - \frac{v_i(y^n)}{n} \Big)^2 }
\ge 2  n^{1/26}.
\end{equation}
Now combining \eqref{eq:spl3} with \eqref{eq:spl1}, we deduce that
for almost all $y^n\in\cY^n$, we can find a $\tilde{\mathbi{u}}=\tilde{\mathbi{u}}(y^n)$ such that
$$
B_1^c \subseteq \{\mathbi{u}\in\mathbb{R}^{k-1}:
\sqrt{h_{\mathbi{v}}(\mathbi{u})} - \sqrt{h_{\mathbi{v}}(\tilde{\mathbi{u}})} > n^{1/26} \}.
$$
Once the details are filled in, this will establish \eqref{eq:osk}, and thus we will be able to conclude that for almost all $y^n\in\cY^n$, the posterior distribution of $\mathbi{U}$ given $Y^n=y^n$ is very close to $\cN(\Phi^{-1}\mathbi{w}, \Phi^{-1})$.

It is a standard fact that under $\ell_2^2$ loss, the optimal Bayes estimator for $\mathbi{u}$ is $\mathbb{E}(\mathbi{U}|Y^n)$. Therefore, the optimal Bayes estimation loss is 
$\sum_{i=1}^k \mathbb{E}(U_i - \mathbb{E}(U_i|Y^n))^2$, and this is equal to the sum of the variances of the posterior distributions of $U_i$ given $Y^n$.
We just showed that the posterior distribution is very close to $\cN(\Phi^{-1}\mathbi{w}, \Phi^{-1})$. 
Therefore, $\sum_{i=1}^k \mathbb{E}(U_i - \mathbb{E}(U_i|Y^n))^2$
can be approximated as $\tr(\Phi^{-1}) + \mathbi{1}^T \Phi^{-1} \mathbi{1}$, where $\mathbi{1}$ is the column vector of $k-1$ ones. 
More precisely, we can show that
$$
r_{k,n}^{\ell_2^2} (\mathbi{Q}) \ge
\tr(\Phi^{-1}) + \mathbi{1}^T \Phi^{-1} \mathbi{1} - \big| O \big( n^{-14/13} \big) \big|.
$$
Then we will prove that for all $\mathbi{Q}\in \cD_{\epsilon,E}$,
$$
\tr(\Phi^{-1}) + \mathbi{1}^T \Phi^{-1} \mathbi{1}
\ge \frac 1n M(k,\epsilon)
$$
and since by \eqref{eq:dstar},\,\eqref{eq:rdstar}, $r_{k,n}^{\ell_2^2} (\mathbi{Q}_{k,\epsilon,d^\ast}, \hat{\mathbi{p}})=\frac 1n M(k,\epsilon),$ this 
will prove Theorem~\ref{Thm:Main}.

\section{Asymptotically tight lower bound: Proof of Theorem~\ref{Thm:Main}}\label{Sect:Main}
In this section we develop the plan of attack outlined in Section~\ref{Sect:sketch}. Since the proof is rather long, we divide it into several steps, isolating each of them in its own subsection.

\subsection{Output alphabet reduction}\label{Sect:merge}
Here we fill in the details left out in Sec.~\ref{Sect:OAR}.
Given $\mathbi{Q}\in \cD_{\epsilon,E},$ assume without loss of generality that the output alphabet is $\cY=\{1,2,\dots,L'\}$ for some integer $L'.$  Define an equivalence relation
``$\equiv$" on $\cY$ as follows: 
for $i_1,i_2\in\{1,2,\dots,L'\},$ we say that $i_1 \equiv i_2$ if 
\begin{equation}\label{eq:eqv}
\frac{\mathbi{Q}(i_1|j)}{\sum_{j'=1}^k \mathbi{Q}(i_1|j')}
= \frac{\mathbi{Q}(i_2|j)}{\sum_{j'=1}^k \mathbi{Q}(i_2|j')}
\text{~for all~} j\in[k].
\end{equation}
In other words, we say that $i_1 \equiv i_2$ if the vectors $(\mathbi{Q}(i_1|1), \mathbi{Q}(i_1|2), \dots, \mathbi{Q}(i_1|k))$ and $(\mathbi{Q}(i_2|1), \mathbi{Q}(i_2|2), \linebreak[4]  \dots, \mathbi{Q}(i_2|k))$ are proportional to each other.
It is easy to verify that $\equiv$ is indeed an equivalence relation and therefore it induces a partition of $\cY$ into $L$ disjoint equivalence classes $\{A_i\}_{i=1}^L$, so $\cY=\cup_{i=1}^L A_i.$ By definition of $\cD_{\epsilon,E},$ for any $i\in\cY,$ the vector 
$(\mathbi{Q}(i|1), \mathbi{Q}(i|2), \dots, \mathbi{Q}(i|k))$ is proportional to one of the vectors in $\{1,e^{\epsilon}\}^k,$ which implies that $L\le 2^k.$ 

Next we will show that for the purposes of this proof,
the original output alphabet $\cY$ can be replaced with the alphabet $\{A_1,\dots,A_L\}$ with only minor notational changes.
We briefly discuss them below, with the overall goal of writing out the pdf $f_{U|Y^n}$ as in \eqref{eq:f}-\eqref{eq:g}.

For $i\in\{1,2,\dots,L'\}$ and $y^n\in\cY^n,$ let $t'_i(y^n)$ be the number of times that symbol $i$ appears in $y^n,$
and let 
\begin{equation}\label{eq:qpr}
q'_i=\frac{1}{k}\sum_{j=1}^k \mathbi{Q}(i|j).
\end{equation}
For $i\in\{1,2,\dots,L\}$ and $j\in\{1,2,\dots,k\},$ define 
   \begin{equation}\label{eq:tq}
   \begin{aligned}
   t_i(y^n)=\sum_{a\in A_i}t'_a(y^n),\quad  q_i=\sum_{a\in A_i} q'_a,\\
   q_{ij}= \mathbi{Q}(A_i|j) = \sum_{a\in A_i}\mathbi{Q}(a|j).
   \end{aligned}
\end{equation}
By definition of the equivalence relation \eqref{eq:eqv}, we have
   \begin{equation}\label{eq:eqclass}
\frac{\mathbi{Q}(a|j)}{q'_a}=\frac{q_{ij}}{q_i} \text{~for all~} a\in A_i \text{~and~} j\in[k].
\end{equation}
For $i\in\{1,2,\dots,L'\}$ and $y^n\in\cY^n,$ let 
$$
\mathbi{v}'(y^n)=(v'_1(y^n),v'_2(y^n),\dots,v'_{L'}(y^n))=(t'_1(y^n) - n q'_1, t'_2(y^n) - n q'_2,\dots, t'_{L'}(y^n) - n q'_{L'}).
$$
Since $\sum_{i=1}^{L'} t'_i(y^n) = \sum_{i=1}^{L'} n q'_i = n,$ we have $\sum_{i=1}^{L'} v'_i(y^n)=0$ for every $y^n\in\cY^n.$
We also define the random vector
$$
\mathbi{V}'=(V'_1,V'_2,\dots,V'_{L'}):=(t'_1(Y^n) - n q'_1, t'_2(Y^n) - n q'_2,\dots, t'_{L'}(Y^n) - n q'_{L'}).$$
For $i\in[L]$ and $y^n\in\cY^n,$ let 
$$
\mathbi{v}(y^n)=(v_1(y^n),v_2(y^n),\dots,v_L(y^n))=(t_1(y^n) - n q_1, t_2(y^n) - n q_2,\dots, t_L(y^n) - n q_L).
$$
Similarly, $\sum_{i=1}^L v_i(y^n)=0$ for every $y^n\in\cY^n.$ Moreover, by definition,
\begin{equation}\label{eq:vi}
v_i(y^n)=\sum_{a\in A_i} v'_a(y^n) \text{~for all~} y^n\in\cY^n.
\end{equation}
We also define the random vector
$$
\mathbi{V}=(V_1,V_2,\dots,V_L)=(t_1(Y^n) - n q_1, t_2(Y^n) - n q_2,\dots, t_L(Y^n) - n q_L).
$$ 
For the simplicity of notation, from now on we will write $\mathbi{v}'(y^n)$ and $\mathbi{v}(y^n)$ as $\mathbi{v}'$ and $\mathbi{v},$ respectively. Similarly, we abbreviate $v'_i(y^n)$ and $v_i(y^n)$ as $v'_i$ and $v_i,$ respectively.
In Proposition~\ref{prop:fg} below, we will show that in order to prove the lower bound in Theorem~\ref{Thm:Main}, we only need the quantities $q_i,v_i,q_{ij},i\in[L],j\in[k]$. Therefore abusing notation, we will write $\cY$ as $\{A_1,\dots,A_L\}$ and remove $L'$ and all the quantities associated with it from consideration in most parts of this proof.

Next let us switch our attention to the pdf $f_{U|Y^n}$.  Given $\mathbi{p}=(p_1,\dots,p_k)\in\Delta_k,$ we define
a $(k-1)$-dimensional vector $\mathbi{u}=(u_1,u_2,\dots,u_{k-1})=(p_1-1/k,p_2-1/k,\dots,p_{k-1}-1/k)$
and let $\mathbi{U}=(U_1,\dots,U_{k-1})$ be the corresponding random vector.
Clearly, $p_k=1/k-\sum_{i=1}^{k-1} u_i.$ As a result, estimating $\mathbi{p}$ is equivalent to estimating $(u_1,u_2,\dots,u_{k-1},-\sum_{i=1}^{k-1} u_i),$ so from now on we will focus on estimating the latter. 
Given $\mathbi{u}\in\mathbb{R}^{k-1}$, define $u_k(\mathbi{u})=-\sum_{i=1}^{k-1} u_i$ and define the corresponding random variable
$U_k(\mathbi{U})=-\sum_{i=1}^{k-1} U_i.$ Below we write these quantities as $u_k$ and $U_k$ respectively for simplicity of notation.

We only consider distributions of $\mathbi{U}$ supported on a small ellipsoid centered at $(\frac{1}{k},\frac{1}{k},\dots,\frac{1}{k}).$ We use Bayes estimation loss to bound below the minimax estimation loss $r_{k,n}^{\ell_2^2} (\mathbi{Q}).$ More specifically, 
we assume that the random vector $\mathbi{U}$ is uniformly distributed over the ellipsoid
$
B_1=B ( \frac{1}{n^{5/13}}).
$

\vspace*{.1in}
We have the following proposition.
\begin{proposition}\label{prop:fg} Assume that $\mathbi{U}$ is distributed over the ellipsoid $B_1$ and let $f_{\mathbi{U},Y^n}$ be
the density of the joint distribution.
For $\mathbi{u}\in B_1$ we have $f_{\mathbi{U},Y^n}(\mathbi{u},y^n)=C_{\mathbi{v}'}\exp(g(\bfu,y^n)),$
where
   \begin{equation}\label{eq:defg}
g(\mathbi{u},y^n) = \sum_{i=1}^L  \Big( (n q_i + v_i)
\log \Big( 1 + \sum_{j=1}^k \frac{u_j q_{ij}}{q_i} \Big) \Big),
\end{equation}
and $C_{\mathbi{v}'}=\frac{1}{\text{\rm Vol}(B_1)} \prod_{i=1}^{L'} (q'_i)^{n q'_a + v'_a}$.
\end{proposition}
\begin{proof}
We can write the joint distribution density function of the random vectors $\mathbi{U}$ and $Y^n$ as follows:
$$
f_{\mathbi{U},Y^n}(\mathbi{u},y^n)=  \begin{cases}
\frac{1}{C_0} \prod_{i=1}^{L'} (q'_i + \sum_{j=1}^k u_j \mathbi{Q}(i|j))^{t_i'(y^n)} & \mbox{~if~} \mathbi{u}\in B_1 \\
0 & \mbox{~if~}  \mathbi{u}\notin B_1
\end{cases},
$$
where $C_0$ is the volume of $B_1.$ For $\mathbi{u}\in B_1,$ we have
\begin{align*}
f_{\mathbi{U},Y^n}(\mathbi{u},y^n)
& =  \frac{1}{C_0} \prod_{i=1}^{L'} {(q'_i)}^{n q'_a + v'_a} 
\prod_{i=1}^{L'}  \Big(1 + \sum_{j=1}^k \frac{u_j \mathbi{Q}(i|j)}{q'_i}\Big)^{n q'_a + v'_a} \\
& = C_{\mathbi{v}'} \prod_{i=1}^{L'} \exp\Big\{  (n q'_a + v'_a)
\log \Big( 1 + \sum_{j=1}^k \frac{u_j \mathbi{Q}(i|j)}{q'_i} \Big) \Big\} \\
& = C_{\mathbi{v}'}  \exp\Big( \sum_{i=1}^{L'} \Big( (n q'_a + v'_a)
\log \Big( 1 + \sum_{j=1}^k \frac{u_j \mathbi{Q}(i|j)}{q'_i} \Big) \Big) \Big) \\
& = C_{\mathbi{v}'}  \exp\Big( \sum_{i=1}^L \sum_{a \in A_i} \Big( (n q'_a + v'_a)
\log \Big( 1 + \sum_{j=1}^k \frac{u_j \mathbi{Q}(a|j)}{q'_a} \Big) \Big) \Big) \\
& = C_{\mathbi{v}'}  \exp\Big( \sum_{i=1}^L \Big(  (n q_i + v_i)
\log \Big( 1 + \sum_{j=1}^k \frac{u_j q_{ij}}{q_i}  \Big) \Big) \Big),
\end{align*}
where $C_{\mathbi{v}'}=\frac{1}{C_0} \prod_{i=1}^{L'} (q'_i)^{n q'_a + v'_a}$ is a constant
that depends on $\mathbi{v}'$ but not on $\mathbi{u}$, and 
the last equality follows from \eqref{eq:eqclass} and \eqref{eq:vi}. 
\end{proof}

\subsection{Approximating $g(\mathbi{u},y^n)$}\label{sect:approx}
As already said, we will assume that the output alphabet of $\mathbi{Q}$ has the form $\{A_1,A_2,\dots,A_L\}$ and will use 
the auxiliary quantities (the composition, etc.) associated with it according to their definitions in \eqref{eq:tq} and \eqref{eq:vi}.
Let
\begin{equation}\label{eq:defg2}
g_2(\mathbi{u},y^n) := \sum_{i=1}^L  v_i \sum_{j=1}^k \frac{u_j q_{ij}}{q_i}
-\frac{1}{2} \sum_{i=1}^L  n q_i  \Big(\sum_{j=1}^k \frac{u_j q_{ij}}{q_i} \Big)^2.
\end{equation}
Further, let $E_1\subseteq\cY^n$ be defined as follows:
 $$
E_1= \Big\{y^n : \sum_{i=1}^L |v_i| <  2 k n^{8/13} \Big\}.
 $$
We will show that when $n$ is large, the difference between $g(\mathbi{u},y^n)$ and $g_2(\mathbi{u},y^n)$ is small for all $\mathbi{u}\in B_1$ and  $y^n\in E_1$.
\begin{proposition} \label{prop:g-g2} Let $g(\mathbi{u},y^n)$ be as defined in \eqref{eq:defg}.
 Then there is an integer $N(k,\epsilon)$ such that for every $y^n\in E_1,\mathbi{u}\in B_1$ and $n>N(k,\epsilon)$,
 \begin{equation}\label{eq:||}
| g(\mathbi{u},y^n) - g_2(\mathbi{u},y^n)|
< \frac{2k^3}{n^{2/13}} 
 .
\end{equation}
Consequently, for all such $y^n,\mathbi u$ and $n$,
\begin{equation}\label{eq:E1}
| \exp(g(\mathbi{u},y^n)) - \exp(g_2(\mathbi{u},y^n)) |\le  \frac{4k^3}{n^{2/13}}   \exp(g_2(\mathbi{u},y^n)). 
\end{equation}
\end{proposition}
\begin{proof}
{\bf 1.} We begin with approximating $g(\mathbi{u},y^n)$ as follows:
$$
g_1(\mathbi{u},y^n) = \sum_{i=1}^L   (n q_i + v_i)
 \Big(  \sum_{j=1}^k \frac{u_j q_{ij}}{q_i} -\frac{1}{2} \Big(\sum_{j=1}^k \frac{u_j q_{ij}}{q_i} \Big)^2 \Big) .
$$
For $\mathbi{u}\in B_1,$ we have $|u_j|<\frac{1}{n^{5/13}}$ for all $j\in\{1,2,\dots,k\}$ (see \eqref{eq:ellipsoid}), and so $|\sum_{j=1}^k u_j q_{ij}|< \frac{1}{n^{5/13}} \sum_{j=1}^k q_{ij}.$ Using this inequality together with the definition of $q_i$ (see Table~\ref{definitions})
we obtain 
   \begin{equation}\label{eq:bdr}
\Big|\sum_{j=1}^k \frac{u_j q_{ij}}{q_i} \Big| < \frac{k}{n^{5/13}}
, \quad i=1,2,\dots,L.
\end{equation}
Given $k,\epsilon,$ there is an integer $N(k,\epsilon)$ such that for all $n>N(k,\epsilon)$ we can bound the difference of $g(\mathbi{u},y^n)$ and $g_1(\mathbi{u},y^n)$ as follows: for all  $\mathbi{u}\in B_1$
and $y^n\in\cY^n$,
\begin{align}
|g(\mathbi{u},y^n) - g_1(\mathbi{u},y^n)|
& \le \sum_{i=1}^L  \Big( (n q_i + v_i)
\Big| \log \Big( 1 + \sum_{j=1}^k \frac{u_j q_{ij}}{q_i} \Big) - 
 \Big(  \sum_{j=1}^k \frac{u_j q_{ij}}{q_i} -\frac{1}{2} \Big(\sum_{j=1}^k \frac{u_j q_{ij}}{q_i} \Big)^2 \Big) \Big| \Big) \nonumber\\
& \overset{(a)}{\le} \sum_{i=1}^L   (n q_i + v_i)
\Big|  \sum_{j=1}^k \frac{u_j q_{ij}}{q_i}  \Big|^3 \nonumber \\
& < \frac{k^3}{n^{15/13}} \sum_{i=1}^L  (n q_i + v_i) = \frac{k^3}{n^{2/13}}
, \label{eq:diff1}
\end{align}
where $(a)$ follows from Prop.~\ref{Prop:log} (Appendix \ref{ap:log}). 

\vspace*{.1in}
{\bf 2.} Multiplying out in the definition of $g_1(\mathbi{u},y^n),$ we can simplify this expression as follows:
\begin{align*}
g_1(\mathbi{u},y^n) & = n \sum_{i=1}^L \sum_{j=1}^k u_j q_{ij}
+ \sum_{i=1}^L  v_i \sum_{j=1}^k \frac{u_j q_{ij}}{q_i}
-\frac{1}{2} \sum_{i=1}^L  n q_i  \Big(\sum_{j=1}^k \frac{u_j q_{ij}}{q_i} \Big)^2
-\frac{1}{2} \sum_{i=1}^L   v_i  \Big(\sum_{j=1}^k \frac{u_j q_{ij}}{q_i} \Big)^2 \\
& = \sum_{i=1}^L  v_i \sum_{j=1}^k \frac{u_j q_{ij}}{q_i}
-\frac{1}{2} \sum_{i=1}^L  n q_i  \Big(\sum_{j=1}^k \frac{u_j q_{ij}}{q_i} \Big)^2
-\frac{1}{2} \sum_{i=1}^L   v_i  \Big(\sum_{j=1}^k \frac{u_j q_{ij}}{q_i} \Big)^2,
\end{align*}
where the second equality follows because $\sum_{i=1}^L q_{ij}=1$ for all $j=1,2,\dots,k$ and $\sum_{j=1}^k u_j=0$ (see \eqref{eq:qu}), and so $\sum_{i=1}^L \sum_{j=1}^k u_j q_{ij}=0.$

We bound the difference of $g_1(\mathbi{u},y^n)$ and $g_2(\mathbi{u},y^n)$ as follows:
\begin{equation}\label{eq:allu}
\big| g_1(\mathbi{u},y^n) - g_2(\mathbi{u},y^n) \big|
 = \frac{1}{2} \Big| \sum_{i=1}^L   v_i  \Big(\sum_{j=1}^k \frac{u_j q_{ij}}{q_i} \Big)^2 \Big| 
< \frac{k^2}{2 n^{10/13}} \|\mathbi{v}\|_1 \text{~for all~}  \mathbi{u}\in B_1,
\end{equation}
where the inequality follows from \eqref{eq:bdr}.
As an immediate consequence, for every $y^n\in E_1$,
$$
\big| g_1(\mathbi{u},y^n) - g_2(\mathbi{u},y^n) \big|
\le \frac{k^3}{n^{2/13}}  \text{~for all~}  \mathbi{u}\in B_1.
$$
Combined with \eqref{eq:diff1}, we deduce that for every $y^n \in E_1,\mathbi{u}\in B_1$
$$
\Big| g(\mathbi{u},y^n) - g_2(\mathbi{u},y^n) \Big|
< \frac{2k^3}{n^{2/13}}.
$$
Therefore, for all $y^n \in E_1,\mathbi{u}\in B_1$
\begin{equation}
\begin{aligned}
\big| \exp(g(\mathbi{u},y^n)) - \exp(g_2(\mathbi{u},y^n)) \big|
& \le  \exp(g_2(\mathbi{u},y^n)) \max\Big\{\exp\Big(\frac{2k^3}{n^{2/13}} \Big)-1,1-\exp \Big(\frac{-2k^3}{n^{2/13}} \Big) \Big\} \\
& \le  \frac{4k^3}{n^{2/13}}   \exp(g_2(\mathbi{u},y^n)),
\end{aligned}
\end{equation}
where the last inequality holds for all $n\ge N(k,\epsilon)$ for a suitably chosen $N(k,\epsilon)$, and it follows from the fact that $|e^x-1|\le 2|x|$ for all $x\le 1/2$. 
\end{proof}

Next we show that $P(Y^n \in E_1)$ is close to $1$ when $n$ is large enough.
\begin{proposition}
There is an integer $N(k,\epsilon)$ such that for all $n>N(k,\epsilon)$,
\begin{equation}\label{eq:diffp}
P(Y^n \in E_1) 
> 1 - \frac{1}{n^{1/13}}.
\end{equation}
\end{proposition}

\begin{proof}
Given $\mathbi{u}\in B_1$, define the event$$
E_{\mathbi{u}}= \bigcap_{i=1}^L \Big\{y^n : \Big| v_i - \sum_{j=1}^k n u_j q_{ij} \Big| < \frac{k n^{8/13}}{2^k} \Big\}.
$$
 (Recall that $\mathbi{v}$ is a function of $y^n$).
To prove the proposition, we first show that 
$E_{\mathbi{u}} \subseteq E_1$ for all $\mathbi{u}\in B_1$. Then we show that
$P(Y^n\in (E_{\mathbi{u}})^c|\mathbi{U}=\mathbi{u})$ is small for all $\mathbi{u}\in B_1$.

By the triangle inequality, we have
\begin{align*}
\|\mathbi{v}\|_1 & \le  \sum_{i=1}^L \Big| \sum_{j=1}^k n u_j q_{ij} \Big| +
\sum_{i=1}^L \Big|v_i - \sum_{j=1}^k n u_j q_{ij}  \Big|.
\end{align*}
Further,
$$
\sum_{i=1}^L \Big| \sum_{j=1}^k n u_j q_{ij} \Big| \le \sum_{i=1}^L \sum_{j=1}^k n |u_j| q_{ij} \le n^{8/13} \sum_{i=1}^L \sum_{j=1}^k  q_{ij}   =kn^{8/13},
$$
and trivially,
$$ \sum_{i=1}^L \Big|v_i - \sum_{j=1}^k n u_j q_{ij}  \Big|\le 2^k \max_{i\in[L]} \Big| v_i - \sum_{j=1}^k n u_j q_{ij} \Big|  .     
$$
Combining these estimates, we obtain
  $$
  \|\mathbi{v}\|_1\le kn^{8/13} +2^k \max_{1\le i\le L} \Big| v_i - \sum_{j=1}^k n u_j q_{ij}\Big|\text{~for all~}  \mathbi{u}\in B_1.
  $$
As a result, for all $\mathbi{u}\in B_1$
\begin{equation}\label{eq:cond1}
E_{\mathbi{u}} \subseteq E_1.
\end{equation}
Note that, conditional on $\mathbi{U}=\mathbi{u},$ the random variable $t_i(Y^n)$ has binomial distribution $B(n, q_i + \sum_{j=1}^k u_j q_{ij})$. Using Hoeffding's inequality, for every $\mathbi{u}\in B_1, i=1,\dots,L$ we have
\begin{align}
P \Big\{ \Big| V_i - \sum_{j=1}^k n u_j q_{ij} \Big| \ge \frac{k n^{8/13}}{2^k} \,\Big|\, \mathbi{U}=\mathbi{u} \Big\}
& = P \Big\{ \Big| \frac{t_i(Y^n)}{n} - \Big(q_i + \sum_{j=1}^k u_j q_{ij} \Big) \Big| \ge \frac{k }{2^k n^{5/13}} \,\Big|\, \mathbi{U}=\mathbi{u} \Big\}\nonumber \\
& \le 2 \exp\Big(- \frac{k^2 n^{3/13}}{2^{2k-1}}\Big).\label{eq:cond2}
\end{align}
Combining \eqref{eq:cond1} and \eqref{eq:cond2}, for every $\mathbi{u}\in B_1$ we have
    \begin{align*}
P(Y^n\in (E_1)^c|\mathbi{U}=\mathbi{u}) & \le
      P(Y^n\in (E_{\mathbi{u}})^c|\mathbi{U}=\mathbi{u}) \\
&  \overset{(a)}{\le} \sum_{i=1}^L  P \Big\{ \Big| V_i - \sum_{j=1}^k n u_j q_{ij} \Big| \ge \frac{k n^{8/13}}{2^k} \,\Big|\, \mathbi{U}=\mathbi{u} \Big\} \\
& \le 2L \exp\Big\{- \frac{k^2 n^{3/13}}{2^{2k-1}}\Big\} \\
& \overset{(b)}{\le} 2^{k+1} \exp\Big\{- \frac{k^2 n^{3/13}}{2^{2k-1}}\Big\}
\end{align*}
where $(a)$ follows from the union bound and $(b)$ follows from the fact that $L\le 2^k.$
Therefore,
$$
P(Y^n \in E_1) 
\ge 1 - 2^{k+1} \exp\Big\{- \frac{k^2 n^{3/13}}{2^{2k-1}}\Big\}
> 1 - \frac{1}{n^{1/13}},
$$
where the last inequality holds for all $n\ge N(k,\epsilon)$ for a suitably chosen $N(k,\epsilon)$.
\end{proof}

\subsection{Gaussian approximation to the posterior distribution }
The expression \eqref{eq:defg2} for the function $g_2(\mathbi{u},y^n)$ can also be written as
\begin{align*}
g_2(\mathbi{u},y^n) = - \sum_{i=1}^L \frac{n}{2q_i} \Big(\sum_{j=1}^k u_j q_{ij} - \frac{v_i}{n} \Big)^2
 + \sum_{i=1}^L \frac{v_i^2}{2n q_i}.
\end{align*}
Given $\mathbi{v} \in \mathbb{R}^L,$ let us define a function $h_{\mathbi{v}}: \mathbb{R}^{k-1} \to \mathbb{R}$ and a constant $C_{\mathbi{v}}$ as
\begin{equation}\label{eq:hvcv}
h_{\mathbi{v}}(\mathbi{u}) = \sum_{i=1}^L \frac{n}{q_i} \Big(\sum_{j=1}^k u_j q_{ij} - \frac{v_i}{n} \Big)^2 \text{~for all~} \mathbi{u} \in \mathbb{R}^{k-1}, \quad
C_{\mathbi{v}}=\exp\Big(\sum_{i=1}^L \frac{v_i^2}{2n q_i} \Big).
\end{equation}
(see \eqref{eq:defh}; note that $C_\mathbi{v}$ does not depend on $\mathbi{u}$). Using this notation, we have
\begin{gather}\label{eq:g2}
g_2(\mathbi{u},y^n) = - \frac{1}{2} h_{\mathbi{v}}(\mathbi{u})
 + \sum_{i=1}^L \frac{v_i^2}{2n q_i}, \\
\exp(g_2(\mathbi{u},y^n)) = C_{\mathbi{v}} \exp ( -  h_{\mathbi{v}}(\mathbi{u})/2 ).\nonumber
\end{gather}
In order to estimate $\mathbi{u}$ from $y^n,$ we need to find the conditional distribution $f_{\mathbi{U}|Y^n}(\mathbi{u}|y^n).$ As a first step, we need to calculate 
$$
P_{Y^n}(y^n)= \int_{\mathbb{R}^{k-1}} f_{\mathbi{U},Y^n}(\mathbi{u},y^n) d\mathbi{u}
=C_{\mathbi{v}'} \int_{ B_1}   \exp(g(\mathbi{u},y^n)) d\mathbi{u},
$$
where $C_{\mathbi{v}'}$ is defined in Prop. \ref{prop:fg};
however this appears difficult (while it is possible to find the asymptotics of this integral using, for instance,
the multi-dimensional version of the Laplace method \cite{Wong2001}, controlling the error terms presents a problem, so we proceed in an ad-hoc way). According to \eqref{eq:E1} and \eqref{eq:diffp}, when $n$ is sufficiently large, with large probability the ratio between $\exp(g_2(\mathbi{u},y^n))$ and $\exp(g(\mathbi{u},y^n))$ is very close to $1.$ So we can use the following integral to approximate $P_{Y^n}(y^n)$:
\begin{equation}\label{eq:defG}
G(y^n)  = C_{\mathbi{v}'}  \int_{ B_1} \exp(g_2(\mathbi{u},y^n)) d\mathbi{u} = C_{\mathbi{v}'}  C_{\mathbi{v}}  \int_{ B_1} 
\exp\Big( - \frac{1}{2} h_{\mathbi{v}}(\mathbi{u}) \Big) d\mathbi{u}.
\end{equation}
The integrand in \eqref{eq:defG} is proportional to the probability density function of some Gaussian distribution. We will make use of this property to obtain an approximation to $G(y^n)$, which is in turn an approximation of $P_{Y^n}(y^n).$

Define 
\begin{equation}\label{eq:defdt}
\delta=\delta(\mathbi{Q})=\min_{\mathbi{u}\in\mathbb{R}^{k-1}: \sum_{i=1}^k u_i^2=1}
\Big(\sum_{i=1}^L \frac{1}{q_i} \Big(\sum_{j=1}^k u_j q_{ij} \Big)^2\Big)^{1/2}
\end{equation}
(repeated here from \eqref{eq:skt2}). Note that we are minimizing a continuous function over a compact set, so the minimum is attained.
It is easy to verify that for all $\mathbi{u}\in\mathbb{R}^{k-1}$
\begin{equation}\label{eq:defdelta}
\Big(\sum_{i=1}^L \frac{1}{q_i} \Big(\sum_{j=1}^k u_j q_{ij} \Big)^2\Big)^{1/2}
\ge \delta \Big(\sum_{i=1}^k u_i^2\Big)^{1/2}.
\end{equation}
Given $k$ and $\epsilon$, let define a constant
\begin{equation}\label{eq:defd0}
\delta_0 = \delta_0(k,\epsilon) := \sqrt{ \frac{1}{32M(k,\epsilon)}}
\end{equation}
where $M(k,\epsilon)$ is given by \eqref{eq:Mke}. 
The remainder of the proof of Theorem \ref{Thm:Main} depends on whether $\delta\ge \delta_0$ or not.
We divide the proof into these two cases because without this division, we can only prove a weaker version of \eqref{eq:indv}, where the constants $C(k,\epsilon)$ and $N(k,\epsilon)$ are replaced with constants that depend on $\mathbi{Q}$.

\subsection{\bf Case 1: $\delta\ge \delta_0$}

We use the Bayes estimation loss to bound below $r_{k,n}^{\ell_2^2} (\mathbi{Q}).$ It is well known that under $\ell_2^2$ loss, the optimal Bayes estimator for $u_i$ is $\mathbb{E}(U_i|Y^n)$. Therefore, the optimal Bayes estimation loss is 
  \begin{equation}\label{eq:BEL}
\sum_{i=1}^k \mathbb{E}(U_i - \mathbb{E}(U_i|Y^n))^2,
   \end{equation}
      and this is equal to the sum of the variances of the posterior distributions of $U_i$ given $Y^n$.
The main idea is to approximate the posterior distribution of $\mathbi{U}$ given $Y^n$ by a multivariate Gaussian distribution. More specifically, we define a $(k-1) \times (k-1)$ matrix $\Phi(n,\mathbi{Q}):=\sum_{i=1}^L \mathbi{z}_i^T\mathbi{z}_i,$  where
\begin{equation}\label{eq:altPhi}
\mathbi{z}_i=\sqrt{\frac{n}{q_i}}(q_{i,1}-q_{i,k}, q_{i,2}-q_{i,k}, \dots,q_{i,k-1}-q_{i,k}).
\end{equation}
Equivalently, $\Phi(n,\mathbi{Q})$ is defined by its associated quadratic form as follows:
\begin{equation}\label{eq:defPhi}
\mathbi{u}^T \Phi(n,\mathbi{Q}) \mathbi{u} =
\sum_{i=1}^L \frac{n}{q_i} \Big(\sum_{j=1}^k u_j q_{ij} \Big)^2
= \sum_{i=1}^L \frac{n}{q_i} \Big(\sum_{j=1}^{k-1} u_j (q_{ij}-q_{ik}) \Big)^2
\text{~for all~} \mathbi{u}\in\mathbb{R}^{k-1}.
\end{equation}
Since $\delta\ge \delta_0>0,$ we have $\mathbi{u}^T \Phi(n,\mathbi{Q}) \mathbi{u}>0$ for all $\mathbi{u} \neq \mathbi{0},$ 
which shows that $\Phi(n,\mathbi{Q})$ is positive definite. For simplicity of notation, we write it as $\Phi,$ omitting the
arguments. We will show that when $n$ is large enough, there is a set $E_2 \subseteq \cY^n$ with $P(Y^n \in E_2)$ close to $1,$ 
such that conditional on every $y^n \in E_1 \cap E_2,$ the covariance matrix of $\mathbi{U}$ is close to $\Phi^{-1}.$
This will enable us to conclude that $\sum_{i=1}^k \mathbb{E}(U_i - \mathbb{E}(U_i|Y^n))^2$
can be approximated as $\tr(\Phi^{-1}) + \mathbi{1}^T \Phi^{-1} \mathbi{1}$, where $\mathbi{1}$ is the all-ones column vector of length $k-1$.

More precisely, our proof of Case 1 consists of two steps, summarized in the following proposition. 
\begin{proposition}\label{prop:Case1}
There are two positive constants $C_3(k,\epsilon),C_4(k,\epsilon)$ and an integer $N(k,\epsilon)$ such that when $n \ge N(k,\epsilon)$,
the following lower bound holds for all $\mathbi{Q}\in \cD_{\epsilon,E}$:
\begin{equation}\label{eq:ob1} 
r_{k,n}^{\ell_2^2} (\mathbi{Q}) \ge
\Big(1  - \frac{C_4(k,\epsilon)}{n^{1/13}} \Big)
 ( \tr(\Phi^{-1}) + \mathbi{1}^T \Phi^{-1} \mathbi{1} ) 
 -  \frac{C_3(k,\epsilon)}{n^{14/13}} .
\end{equation}
Furthermore, for all $\mathbi{Q}\in \cD_{\epsilon,E}$ and all positive integer $n$,
   \begin{equation}\label{eq:tr}
\tr(\Phi^{-1}) + \mathbi{1}^T \Phi^{-1} \mathbi{1}
\ge \frac 1n M(k,\epsilon).
   \end{equation}
\end{proposition}

Once both \eqref{eq:ob1} and \eqref{eq:tr} are established, this will complete the proof of Theorem~\ref{Thm:Main} for Case 1 because of  \eqref{eq:Mke},\,\eqref{eq:cnlb} and \eqref{eq:indv}.

\vspace{.1in}\subsubsection{Proof of \eqref{eq:ob1}}
In order to prove inequality \eqref{eq:ob1}, we develop a refined version of the Local Asymptotic Normality approach \cite{Hajek72,Ibrag81}\footnote{A detailed discussion of the connection between our proof and LAN appears in Appendix~\ref{ap:LAN}.}. 
Consider the ellipsoid (see \eqref{eq:ellipsoid})
$$
B_2 = B \Big(  \frac{1}{n^{5/13}} - \frac{3/\delta_0}{n^{6/13}} \Big),
$$
and define a subset $E_2 \subseteq \cY^n$ as follows:
$$
E_2=\Big\{ y^n:  \exists \tilde{\mathbi{u}} \in B_2 \text{~such that~}
\sum_{i=1}^L \frac{n}{q_i} \Big(\sum_{j=1}^k \tilde{u}_j q_{ij} - \frac{v_i}{n} \Big)^2
< n^{1/13}  \Big\}.
$$
Recall our convention that $\tilde{u}_k=\tilde{u}_k(\tilde{\mathbi{u}})= -\sum_{i=1}^{k-1}\tilde{u}_i$, which stems from the fact that we are working with PMFs.
\begin{proposition}\label{Prop:E2}
 There is an integer $N(k,\epsilon)$ such that for all $n>N(k,\epsilon)$,
\begin{equation}\label{eq:pe2}
P(Y^n\in E_2) 
 \ge 1 - \frac{2^{k+1} + 3k/\delta_0}{n^{1/13}}.
\end{equation}
\end{proposition}
The proof is given in Appendix \ref{ap:E2}.

\vspace{.1in}
Our goal will be to show that for $y^n\in E_2$, the ratio between $G(y^n)$ in \eqref{eq:defG} and 
\begin{equation}\label{eq:H}
H(y^n) = C_{\mathbi{v}'}  C_{\mathbi{v}}  \int_{\mathbb{R}^{k-1}} 
\exp\Big( - \frac{1}{2} h_{\mathbi{v}}(\mathbi{u}) \Big) d\mathbi{u} 
\end{equation}
is very close to $1.$ Specifically, we have 
\begin{proposition} \label{prop:HGH}
 There is an integer $N(k,\epsilon)$ such that for all $n>N(k,\epsilon)$ and all
$y^n \in  E_2$,
  \begin{equation}\label{eq:HGH}
  \Big(1-\frac{256}{n^{2/13}}\Big)H(y^n)\le G(y^n)\le H(y^n)
  \end{equation}
\end{proposition}

The upper bound on $G(y^n)$ in \eqref{eq:HGH} is obvious. To prove the lower bound, we need several auxiliary definitions and propositions.

Given $\alpha>0$ and $\tilde{\mathbi{u}} \in \mathbb{R}^{k-1}$, define the set
   \begin{equation}\label{eq:Ev}
E_{\mathbi{v}}(\tilde{\mathbi{u}}, \alpha)=\{\mathbi{u} \in \mathbb{R}^{k-1}: \sqrt{h_{\mathbi{v}}(\mathbi{u})} - \sqrt{h_{\mathbi{v}}(\tilde{\mathbi{u}})} > \alpha \}.
  \end{equation}
\begin{proposition}\label{Prop:inc}
For every $y^n \in E_2$ and $\alpha\ge n^{-5/13},$ there exists $\tilde{\mathbi{u}} \in B_2$ such that the ellipsoid $B(\alpha)$ defined in \eqref{eq:ellipsoid}
satisfies
\begin{equation}\label{eq:BEag}
(B(\alpha))^c \subseteq E_{\mathbi{v}}(\tilde{\mathbi{u}}, \delta_0 n^{1/2} (\alpha - n^{-5/13} ) + n^{1/26}).
\end{equation}
\end{proposition}
The proof is given in Appendix \ref{ap:inc}. 
\begin{proposition}\label{Prop:U}
Let $\Phi$ be an $s\times s$ positive definite matrix, and let $\mathbi{t}$ be a column vector in $\mathbb{R}^s.$
Define a quadratic function $h:\mathbb{R}^s \to \mathbb{R}$ as follows:
\begin{equation}\label{eq:form}
h(\mathbi{u})=  (\mathbi{u}-\mathbi{t})^T \Phi  (\mathbi{u}-\mathbi{t}) + C,
\end{equation}
where $C \ge 0$ is a nonnegative constant.
For a given $\tilde{\mathbi{u}}\in \mathbb{R}^s,$ define the set
$$
E(\tilde{\mathbi{u}}, \alpha)=\{\mathbi{u} \in \mathbb{R}^s: \sqrt{h(\mathbi{u})} - \sqrt{h(\tilde{\mathbi{u}})} > \alpha \}.
$$
For all  $\alpha \ge \sqrt{s}$ and all $\tilde{\mathbi{u}} \in \mathbb{R}^s$ the following inequality holds true:
  \begin{equation}\label{eq:II}
\frac{\int_{E(\tilde{\mathbi{u}}, \alpha)} \exp(-\frac{1}{2}h(\mathbi{u})) d\mathbi{u}}{\int_{\mathbb{R}^s} \exp(-\frac{1}{2} h(\mathbi{u})) d\mathbi{u}}
\le e^{-\frac{(\alpha-\sqrt{s})^2}{2}}.
  \end{equation}
\end{proposition}
The proof is given in Appendix \ref{ap:concen}.

\vspace*{.1in}{\em {Proof of Proposition \ref{prop:HGH}}}:
Our plan is to use Prop.~\ref{Prop:U} for the function $h_{\mathbi{v}}$ defined in \eqref{eq:hvcv} in order to estimate $G(y^n).$
First let us show that $h_{\mathbi{v}}$ can be written as a quadratic form of the type \eqref{eq:form}.
Define a vector $\mathbi{w} \in \mathbb{R}^{k-1}$ as follows:
\begin{equation}\label{eq:defw}
\mathbi{w}=\mathbi{w}(\mathbi{v},\mathbi{Q})= \Big( \sum_{i=1}^L  \frac{
(q_{im}-q_{ik}) v_i  }{q_i}, m=1,\dots,k-1 \Big)^T.
\end{equation}
Then we have
   \begin{align}
h_{\mathbi{v}}(\mathbi{u}) & = \sum_{i=1}^L \frac{n}{q_i} \Big(\sum_{j=1}^k u_j q_{ij} \Big)^2
- 2 \sum_{i=1}^L  \sum_{j=1}^k   \frac{v_i q_{ij}}{q_i} u_j 
+ \sum_{i=1}^L \frac{ v_i^2}{nq_i}  \nonumber\\
& = \mathbi{u}^T \Phi \mathbi{u} - 2  \sum_{j=1}^k \Big(\sum_{i=1}^L  \frac{q_{ij} v_i }{q_i}\Big) u_j 
+ \sum_{i=1}^L \frac{ v_i^2}{nq_i} \nonumber\\
& = \mathbi{u}^T \Phi \mathbi{u} - 2  \sum_{j=1}^{k-1} \Big(\sum_{i=1}^L  \frac{
(q_{ij}-q_{ik}) v_i  }{q_i}\Big) u_j 
+ \sum_{i=1}^L \frac{ v_i^2}{nq_i} \nonumber\\
& = \mathbi{u}^T \Phi \mathbi{u} - \mathbi{w}^T \mathbi{u} - \mathbi{u}^T \mathbi{w}
+ \sum_{i=1}^L \frac{ v_i^2}{nq_i} \nonumber\\
& = (\mathbi{u} - \Phi^{-1}\mathbi{w})^T \Phi (\mathbi{u} - \Phi^{-1}\mathbi{w})
- \mathbi{w}^T \Phi^{-1} \mathbi{w} + \sum_{i=1}^L \frac{ v_i^2}{nq_i},\label{eq:hv}
\end{align}
where the second equality follows from the definition of $\Phi$ in \eqref{eq:defPhi}.

Since $h_{\mathbi{v}}(\mathbi{u}) \ge 0$ for all $\mathbi{u} \in \mathbb{R}^{k-1}$, the constant term
$C:=- \mathbi{w}^T \Phi^{-1} \mathbi{w} + \sum_{i=1}^L \frac{ v_i^2}{nq_i}$ is nonnegative. This shows that $h_{\mathbi{v}}$ satisfies the conditions in Prop.~\ref{Prop:U}.

To estimate $G(y^n)$ we first note that, by Prop.~\ref{Prop:inc}, for every $y^n \in E_2,$ there exists $\tilde{\mathbi{u}} \in B_2$ such that
$
(B_1)^c \subseteq E_{\mathbi{v}}(\tilde{\mathbi{u}}, n^{1/26}).
$
We then obtain
\begin{align*}
G(y^n)  & = C_{\mathbi{v}'}  C_{\mathbi{v}}  \int_{B_1} 
\exp\Big( - \frac{1}{2} h_{\mathbi{v}}(\mathbi{u})  \Big) d\mathbi{u} \nonumber\\
& \ge C_{\mathbi{v}'}  C_{\mathbi{v}} \Big( \int_{\mathbb{R}^{k-1}} 
\exp\Big( - \frac{1}{2} h_{\mathbi{v}}(\mathbi{u})  \Big) d\mathbi{u} -  \int_{E_{\mathbi{v}}(\tilde{\mathbi{u}}, n^{1/26})} 
\exp\Big( - \frac{1}{2} h_{\mathbi{v}}(\mathbi{u})  \Big)  d\mathbi{u} \Big) \\
& \ge H(y^n) \Big( 1 -  
\exp \Big(- \frac{1}{2} \big( n^{1/26} - \sqrt{k-1} \big)^2 \Big) \Big) \\
& \ge H(y^n) \Big( 1 -  
\frac{256}{n^{2/13}} \Big).\label{eq:bG}
\end{align*}
Here the  next-to-last step follows from Prop.~\ref{Prop:U}, and the last inequality holds for all $n\ge N(k,\epsilon)$
as long as we choose $N(k,\epsilon)$ large enough.
This proves the lower bound in \eqref{eq:HGH}.
\qed

\vspace*{.1in}
We have shown that $G(y^n)$ can be approximated by $H(y^n)$. Based on this, we can further show that $P_{Y^n}(y^n)$ can be approximated by $H(y^n)$.
 
\begin{lemma}\label{lemma:PH}  There is an integer $N(k,\epsilon)$ such that for all $n>N(k,\epsilon)$ and all
$y^n \in E_1\cap E_2$,
\begin{equation}\label{eq:PH}
\Big| P_{Y^n}(y^n)- H(y^n) \Big| 
 \le \frac{4k^3+256}{n^{2/13}} 
H(y^n).
\end{equation}
\end{lemma}
\begin{proof}
For every $y^n \in E_1$ we have
\begin{align*}
|P_{Y^n}(y^n)- G(y^n) | & = \Big| C_{\mathbi{v}'} \int_{ B_1}   \exp(g(\mathbi{u},y^n)) d\mathbi{u} - C_{\mathbi{v}'} \int_{ B_1}   \exp(g_2(\mathbi{u},y^n)) d\mathbi{u} \Big| \\
& \le C_{\mathbi{v}'} \int_{ B_1}   |\exp(g(\mathbi{u},y^n)) - \exp(g_2(\mathbi{u},y^n))| d\mathbi{u} \\
& \overset{(a)}{\le} \frac{4k^3}{n^{2/13}} C_{\mathbi{v}'} \int_{ B_1}   \exp(g_2(\mathbi{u},y^n)) d\mathbi{u} = \frac{4k^3}{n^{2/13}} G(y^n) \\
& \le \frac{4k^3}{n^{2/13}} H(y^n),
\end{align*}
where $(a)$ follows from \eqref{eq:E1}. On account of \eqref{eq:HGH}, this concludes the proof.
\end{proof} 

\vspace*{.1in}
Define
\begin{equation}\label{eq:f1}
f_1(\mathbi{u},y^n)= \frac{C_{\mathbi{v}'} \exp(g_2(\mathbi{u},y^n))}{H(y^n)}.
\end{equation}
For every $y^n \in E_1\cap E_2,$ we can approximate the conditional probability density function $f_{\mathbi{U}|Y^n}(\mathbi{u}|Y^n=y^n)$ as $f_1(\mathbi{u},y^n)$. At the same time, we will also show that for a fixed $y^n,$ $f_1(\mathbi{u},y^n)$ is the probability density function of a Gaussian random vector with mean vector $\Phi^{-1}\mathbi{w}$ and covariance matrix $\Phi^{-1}.$

The difference is bounded as follows: for all $\mathbi{u}\in B_1,$
  \begin{align}
  & |f_{\mathbi{U}|Y^n}(\mathbi{u}|Y^n=y^n) - f_1(\mathbi{u},y^n) | = 
\Big|\frac{C_{\mathbi{v}'} \exp(g(\mathbi{u},y^n))}{P_{Y^n}(y^n)} - 
\frac{C_{\mathbi{v}'} \exp(g_2(\mathbi{u},y^n))}{H(y^n)} \Big| \nonumber\\
&\overset{(a)}{\le}  \frac{C_{\mathbi{v}'} \exp(g_2(\mathbi{u},y^n)) }{H(y^n)}
\max\Big(\frac{1 +4k^3/(n^{2/13}) }{1 - (4k^3+256)/(n^{2/13}) } -1, 
1- \frac{1 - 4k^3/(n^{2/13}) }{1 + (4k^3+256)/(n^{2/13}) } \Big) \nonumber\\
&\overset{(b)}{\le}  \frac{16k^3 + 512}{n^{2/13} } f_1(\mathbi{u},y^n)\nonumber\\
&\overset{(c)}{=} \frac{C_1(k,\epsilon)}{n^{2/13} } f_1(\mathbi{u},y^n)
,\label{eq:ff1}
\end{align}
where $(a)$ follows from \eqref{eq:E1} and \eqref{eq:PH}; $(b)$ holds for all $n\ge N(k,\epsilon)$ as long as we take $N(k,\epsilon)>(8k^3+512)^{13/2}$;
in $(c)$ we define $C_1(k,\epsilon) := 16k^3 + 512$.
By definition,
   \begin{align}
f_1(\mathbi{u},y^n) & = \frac{C_{\mathbi{v}'}  C_{\mathbi{v}}
\exp\Big( - \frac{1}{2} h_{\mathbi{v}}(\mathbi{u})  \Big) }
{C_{\mathbi{v}'}  C_{\mathbi{v}}  \int_{\mathbb{R}^{k-1}} 
\exp\Big( - \frac{1}{2} h_{\mathbi{v}}(\mathbi{u})  \Big) d\mathbi{u} }\nonumber \\
& \overset{(a)}{=} \frac{ \exp\Big( - \frac{1}{2} (\mathbi{u} - \Phi^{-1}\mathbi{w})^T \Phi (\mathbi{u} - \Phi^{-1}\mathbi{w})  \Big) }
{  \int_{\mathbb{R}^{k-1}} 
\exp\Big( - \frac{1}{2} (\mathbi{u} - \Phi^{-1}\mathbi{w})^T \Phi (\mathbi{u} - \Phi^{-1}\mathbi{w})  
 d\mathbi{u} \Big)} \nonumber \\
& = \frac{\sqrt{|\Phi|}}{\sqrt{(2\pi)^{k-1}}}
 \exp\Big( - \frac{1}{2} (\mathbi{u} - \Phi^{-1}\mathbi{w})^T \Phi (\mathbi{u} - \Phi^{-1}\mathbi{w})  \Big),
 \label{eq:expf1}
\end{align}
where $(a)$ follows from \eqref{eq:hv}. Thus $f_1$ indeed represents a Gaussian pdf. 

\vspace*{.1in}
Recall that our goal is to estimate the Bayes estimation loss \eqref{eq:BEL}. Proceeding in this direction, we show next that for every $y^n \in E_1\cap E_2,$ the $\ell_1$ distance between $\mathbb{E}(\mathbi{U}|Y^n=y^n)$ and $\Phi^{-1}\mathbi{w}$ is very small.
\begin{lemma}\label{lemma:EPhi}
 There is an integer $N(k,\epsilon)$ such that for all $n>N(k,\epsilon)$ and all
$y^n \in E_1\cap E_2$,
  $$
  \|\mathbb{E}(\mathbi{U}|Y^n=y^n)-\Phi^{-1}\mathbi{w}\|_1\le \frac{C_2(k,\epsilon)}{n^{7/13}},
  $$
  where $C_2(k,\epsilon) := \frac{32k^4 e^{\epsilon}(k^3+32)}{\delta_0^2}
+ \sqrt{k-1}$.
\end{lemma}

\begin{proof}
  \begin{align}
&  \| \mathbb{E}(\mathbi{U}|Y^n=y^n) - \Phi^{-1}\mathbi{w} \|_1 \nonumber\\
&=  \Big\| \int_{ B_1} \mathbi{u} f_{\mathbi{U}|Y^n}(\mathbi{u}|Y^n=y^n) d\mathbi{u}
- \int_{\mathbb{R}^{k-1}} \mathbi{u} f_1(\mathbi{u},y^n) d\mathbi{u} \Big\|_1 \nonumber\\
&=  \sum_{i=1}^{k-1} \Big| \int_{ B_1} u_i f_{\mathbi{U}|Y^n}(\mathbi{u}|Y^n=y^n) d\mathbi{u}
- \int_{ \mathbb{R}^{k-1}} u_i f_1(\mathbi{u},y^n) d\mathbi{u} \Big| \nonumber\\
&\le  \sum_{i=1}^{k-1} \Big|\int_{ B_1} u_i (f_{\mathbi{U}|Y^n}(\mathbi{u}|Y^n=y^n) - f_1(\mathbi{u},y^n)) d\mathbi{u} \Big|
+ \sum_{i=1}^{k-1} \Big| \int_{ (B_1)^c} u_i f_1(\mathbi{u},y^n) d\mathbi{u} \Big| \nonumber\\
&\overset{(a)}{\le} \frac{16k^3 + 512}{n^{2/13} } \sum_{i=1}^{k-1} \int_{ B_1} |u_i| f_1(\mathbi{u},y^n) d\mathbi{u}
+ \sum_{i=1}^{k-1} \int_{ (B_1)^c} |u_i| f_1(\mathbi{u},y^n) d\mathbi{u}  \nonumber\\
& \le  \frac{16k^3 + 512}{n^{2/13} }  \int_{ \mathbb{R}^{k-1}} 
\|\mathbi{u} \|_1 f_1(\mathbi{u},y^n) d\mathbi{u}
+  \int_{ (B_1)^c} \|\mathbi{u}\|_1 f_1(\mathbi{u},y^n) d\mathbi{u} \nonumber\\
&\overset{(b)}{\le}  \frac{16k^3 + 512}{n^{2/13} } \biggl( \sqrt{k-1} \| \Phi^{-1}\mathbi{w} \|_2 + \sqrt{\frac{2(k-1)}{\pi} \tr(\Phi^{-1})} \biggr)
+  \int_{ (B_1)^c} \|\mathbi{u}\|_1 f_1(\mathbi{u},y^n) d\mathbi{u}, \label{eq:bdE}
  \end{align}
where $(a)$ follows from \eqref{eq:ff1}; and $(b)$ is justified in Appendix~\ref{ap:Gauss}.

Now let us write out $\mathbi{w}$ from its definition in \eqref{eq:defw}:
     \begin{align*}
\|\mathbi{w}\|_2 &\le \|\mathbi{w}\|_1 = 
        \sum_{j=1}^{k-1} \Big| \sum_{i=1}^L  \frac{
               (q_{ij}-q_{ik}) v_i  }{q_i} \Big|\\
&\le k e^{\epsilon} \sum_{j=1}^{k-1}  \sum_{i=1}^L  |v_i|\\
&\le 2k^2(k-1) e^{\epsilon} n^{8/13} \text{~for all~} y^n\in E_1,
\end{align*}
where the second line follows from the fact that the vector 
$(q_{i,1}, q_{i,2}, \dots, q_{ik})$ is proportional to one of the vectors in $\{1,e^{\epsilon}\}^k$ and
the third one follows directly from the definition of $E_1.$

Next we claim that all the eigenvalues of $\Phi^{-1}$ are no greater than $1/(n\delta_0^2).$ Indeed,
according to \eqref{eq:defdelta} and \eqref{eq:defPhi}, for all $\mathbi{u}\in\mathbb{R}^{k-1}$
   \begin{equation}\label{eq:Pn}
\mathbi{u}^T \Phi \mathbi{u} = n
\sum_{i=1}^L \frac{1}{q_i} \Big(\sum_{j=1}^k u_j q_{ij} \Big)^2
\ge n \delta^2 \Big( \sum_{i=1}^k u_i^2 \Big)
\ge n \delta_0^2  \|\mathbi{u}\|_2^2 .
\end{equation}
Using this, we have
 \begin{gather*}
\| \Phi^{-1}\mathbi{w} \|_2 \le \frac{1}{n \delta_0^2} \| \mathbi{w} \|_2
< \frac{2k^3 e^{\epsilon}}{n^{5/13} \delta_0^2}  \text{~for all~} y^n\in E_1,\\
\tr(\Phi^{-1}) \le \frac{k-1}{n\delta_0^2} < \frac{k}{n^{10/13}\delta_0^4},
  \end{gather*}
where the last inequality
 holds for all $n\ge N(k,\epsilon)$ as long as we set $N(k,\epsilon)$ to be large enough.
Combining the two inequalities above with \eqref{eq:bdE}, we deduce that for every $y^n \in E_1\cap E_2,$
\begin{align}
     \Big\| \mathbb{E}(\mathbi{U}|Y^n & =y^n) - \Phi^{-1}\mathbi{w} \Big\|_1 \nonumber\\
 & \le  \frac{16k^3 + 512}{n^{2/13} } \frac{2k^4e^{\epsilon}}{n^{5/13}\delta_0^2}
+ \int_{ (B_1)^c} \|\mathbi{u}\|_1 f_1(\mathbi{u},y^n) d\mathbi{u} \nonumber\\
 & =  \frac{32k^4 e^{\epsilon}(k^3+32)}{n^{7/13}\delta_0^2}
+ \int_{ (B_1)^c} \|\mathbi{u}\|_1 f_1(\mathbi{u},y^n) d\mathbi{u} \nonumber\\
 & \le  \frac{32k^4 e^{\epsilon}(k^3+32)}{n^{7/13}\delta_0^2}
+ \sqrt{k-1} \int_{ (B_1)^c} \|\mathbi{u}\|_2 f_1(\mathbi{u},y^n) d\mathbi{u} \nonumber\\
 & \le  \frac{32k^4 e^{\epsilon}(k^3+32)}{n^{7/13}\delta_0^2}
+ \sqrt{k-1} \int_{ (B_1)^c} \Big(\sum_{i=1}^k u_i^2 \Big)^{1/2} f_1(\mathbi{u},y^n) d\mathbi{u} \nonumber\\
& <  \frac{32k^4 e^{\epsilon}(k^3+32)}{n^{7/13}\delta_0^2}
+ \frac{\sqrt{k-1}}{n^{7/13}},\label{eq:EU}
\end{align}
where the last inequality follows by the estimate \eqref{eq:F} proved in Appendix~\ref{ap:f1} below.
This completes the proof of the lemma.
\end{proof}

\vspace*{.1in} In the next step we bound the conditional expectation $\sum_{i=1}^k \mathbb{E}  (U_i- \mathbb{E} (U_i | Y^n ) )^2$ when $Y^n\in E_1\cap E_2$. 

\begin{lemma}
 There is an integer $N(k,\epsilon)$ such that for all $n>N(k,\epsilon)$ and all
$y^n \in E_1\cap E_2$,
  \begin{equation}\label{eq:osa}
\sum_{i=1}^k \mathbb{E} [ (U_i- \mathbb{E} (U_i | Y^n=y^n ) )^2 | Y^n=y^n]\ge\Big(1 - \frac{C_1(k,\epsilon)}{n^{2/13} } \Big) 
\Big( \tr(\Phi^{-1}) + \mathbi{1}^T \Phi^{-1} \mathbi{1} \Big) 
  -  \frac{C_3(k,\epsilon)}{n^{14/13}},
  \end{equation}
  where $C_3(k,\epsilon) := 2(C_2(k,\epsilon))^2+1$.
\end{lemma}
\begin{proof}
As in Appendix~\ref{ap:f1}, given $y^n\in \cY^n,$ define $\bar{\mathbi{U}}(y^n)=(\bar{U}_1(y^n),\dots,\bar{U}_{k-1}(y^n)) $ to be
a $(k-1)$-dimensional Gaussian random vector with density function $f_{\bar{\mathbi{U}}(y^n)}(\cdot)=  f_1(\cdot,y^n)$, mean vector $\Phi^{-1}\mathbi{w}$ and covariance matrix $\Phi^{-1}.$
By Lemma \ref{lemma:EPhi}, for every $y^n \in E_1\cap E_2$ we have
$$
 \| \mathbb{E}(\mathbi{U}|Y^n=y^n) - \mathbb{E}\big( \bar{\mathbi{U}}(y^n) \big) \|_1 =
 \| \mathbb{E}(\mathbi{U}|Y^n=y^n) - \Phi^{-1}\mathbi{w}\|_1 
< \frac{C_2(k,\epsilon)}{n^{7/13}}.
$$
As usual, let $\bar{U}_k(y^n): = -\sum_{i=1}^{k-1} \bar{U}_i(y^n)$. For every $y^n \in E_1\cap E_2,$
\begin{align*}
\sum_{i=1}^k &( \mathbb{E}(U_i|Y^n=y^n) - \mathbb{E} \bar{U}_i(y^n)  )^2 \\
& =  \sum_{i=1}^{k-1} ( \mathbb{E}(U_i|Y^n=y^n) - \mathbb{E} \bar{U}_i(y^n)  )^2
+  \Big(\sum_{i=1}^{k-1} \mathbb{E}(U_i|Y^n=y^n) - 
\sum_{i=1}^{k-1}\mathbb{E} \bar{U}_i(y^n)  \Big)^2 \\
& \le  2  \| \mathbb{E}(\mathbi{U}|Y^n=y^n) - \mathbb{E} \bar{\mathbi{U}}(y^n) \|_1 ^2
< \frac{2(C_2(k,\epsilon))^2}{n^{14/13}}\\
& = \frac{C_3(k,\epsilon) - 1}{n^{14/13}}
\end{align*}
(on the second line we use the definition of $\bar U_k$).
As a result, for every $y^n \in E_1\cap E_2,$
\begin{align}
   \sum_{i=1}^k &\mathbb{E} [ (U_i - \mathbb{E} (U_i | Y^n=y^n ) )^2 | Y^n=y^n ] \nonumber\\
& \overset{(a)}{=}\sum_{i=1}^k \mathbb{E} [ (U_i- \mathbb{E}( \bar{U}_i(y^n) ) )^2 | Y^n=y^n ]
- \sum_{i=1}^k ( \mathbb{E}(U_i|Y^n=y^n) - \mathbb{E}( \bar{U}_i(y^n) ) )^2 \nonumber\\
&> \sum_{i=1}^k \mathbb{E} [ [(U_i- \mathbb{E}( \bar{U}_i(y^n) ) ])^2 | Y^n=y^n ]-  \frac{C_3(k,\epsilon)-1}{n^{14/13}}\nonumber \\
&=  \sum_{i=1}^k \int_{ B_1} (u_i- \mathbb{E}( \bar{U}_i(y^n) ) )^2 f_{\mathbi{U}|Y^n}(\mathbi{u}|Y^n=y^n) d\mathbi{u}
-  \frac{C_3(k,\epsilon) -1}{n^{14/13}} \nonumber\\
& \overset{(b)}{\ge} \Big(1 - \frac{C_1(k,\epsilon)}{n^{2/13} } \Big) \sum_{i=1}^k \int_{ B_1} (u_i- \mathbb{E}( \bar{U}_i(y^n) ) )^2 f_1(\mathbi{u},y^n) d\mathbi{u}
-  \frac{C_3(k,\epsilon)-1}{n^{14/13}} \nonumber\\
&\ge  \Big(1 - \frac{C_1(k,\epsilon)}{n^{2/13} } \Big) \sum_{i=1}^k \int_{ \mathbb{R}^{k-1}} (u_i- \mathbb{E}( \bar{U}_i(y^n) ) )^2 f_1(\mathbi{u},y^n) d\mathbi{u}
-  \frac{C_3(k,\epsilon)-1}{n^{14/13}} \nonumber\\
  & - \sum_{i=1}^k \int_{ (B_1)^c} (u_i- \mathbb{E}( \bar{U}_i(y^n) ) )^2 f_1(\mathbi{u},y^n) d\mathbi{u} \nonumber\\
& \overset{(c)}{=} \Big(1 - \frac{C_1(k,\epsilon)}{n^{2/13} } \Big) 
\sum_{i=1}^k \Var( \bar{U}_i(y^n) ) 
  -  \frac{C_3(k,\epsilon)-1}{n^{14/13}}\nonumber\\
&\hspace*{2in}   -  \int_{ (B_1)^c} \sum_{i=1}^k (u_i- \mathbb{E}( \bar{U}_i(y^n) ) )^2 f_1(\mathbi{u},y^n) d\mathbi{u}, \label{eq:Cov}
\end{align}
where $(a)$ follows by a straightforward rewriting; in $(b)$ we use \eqref{eq:ff1}; $(c)$ follows from the fact that the probability density function of $\bar{\mathbi{U}}(y^n)$ is $f_1(\cdot,y^n).$ 
Further, since $\bar{U}_k (y^n) = -\sum_{i=1}^{k-1} \bar{U}_i  (y^n),$ we have
  \begin{align}
  \sum_{i=1}^k \Var( \bar{U}_i(y^n) )&=\sum_{i=1}^{k-1}\Var( \bar{U}_i(y^n) )+\Var\Big(\sum_{i=1}^{k-1}\bar{U}_i(y^n)\Big )\nonumber\\
  &=\tr(\Phi^{-1})+\mathbi{1}^T \Phi^{-1} \mathbi{1}.\label{eq:Cov1}
  \end{align}
  Finally, the integral in \eqref{eq:Cov} is estimated in \eqref{eq:I} in Appendix \ref{ap:vr}. Using this estimate together
  with \eqref{eq:Cov1} in \eqref{eq:Cov}, we arrive at the claimed inequality \eqref{eq:osa}.
\end{proof}

Now let us take the final step toward our goal of proving \eqref{eq:ob1}.
Combining \eqref{eq:osa} with \eqref{eq:diffp} and \eqref{eq:pe2}, we bound below the optimal Bayes estimation loss 
as follows:
\begin{align}
 \sum_{i=1}^k \mathbb{E}  \Big(U_i&- \mathbb{E} (U_i | Y^n ) \Big)^2   \nonumber\\
\ge & P(Y^n \in E_1\cap E_2 ) 
\Big( \Big(1 - \frac{C_1(k,\epsilon)}{n^{2/13} } \Big) 
( \tr(\Phi^{-1}) + \mathbi{1}^T \Phi^{-1} \mathbi{1} ) 
  -  \frac{C_3(k,\epsilon)}{n^{14/13}}  \Big) \nonumber\\
\ge & \Big(1 - \frac{1 + 2^{k+1} + 3k/\delta_0}{n^{1/13}} \Big)
\Big( \Big(1 - \frac{C_1(k,\epsilon)}{n^{2/13} } \Big) 
( \tr(\Phi^{-1}) + \mathbi{1}^T \Phi^{-1} \mathbi{1} ) 
  -  \frac{C_3(k,\epsilon)}{n^{14/13}}  \Big) \nonumber\\
\ge & \Big(1  - \frac{1 + 2^{k+1} + 3k/\delta_0}{n^{1/13}}
 - \frac{C_1(k,\epsilon)}{n^{2/13} } \Big) ( \tr(\Phi^{-1}) + \mathbi{1}^T \Phi^{-1} \mathbi{1} ) 
 -  \frac{C_3(k,\epsilon)}{n^{14/13}}  \nonumber\\
> & \Big(1  - \frac{C_4(k,\epsilon)}{n^{1/13}} \Big)
 ( \tr(\Phi^{-1}) + \mathbi{1}^T \Phi^{-1} \mathbi{1} ) 
 -  \frac{C_3(k,\epsilon)}{n^{14/13}}, \nonumber
\end{align}
where in the last inequality we define 
$C_4(k,\epsilon) := 1 + 2^{k+1} + 3k/\delta_0 + C_1(k,\epsilon)$.
This establishes \eqref{eq:ob1} and completes the first step of the proof of Proposition \ref{prop:Case1}.

\vspace*{.1in}
\subsubsection{Proof of \eqref{eq:tr}}

We use the Cauchy-Schwarz inequality to bound $\tr(\Phi^{-1}) + \mathbi{1}^T \Phi^{-1} \mathbi{1}$. The result
is given in the following proposition, which is proved in Appendix~\ref{ap:bdtr}.
\begin{proposition}\label{Prop:bdtr}
\begin{equation}\label{eq:itp}
\tr(\Phi^{-1}) + \mathbi{1}^T \Phi^{-1} \mathbi{1}
\ge (k-1)
\Big( \frac{\sum_{i=1}^{k-1}\Phi_{ii}}{k}-\frac{\sum_{i\neq j}\Phi_{ij} }{k(k-1)} \Big)^{-1}.
\end{equation}
\end{proposition}

Now let us further bound the term on right-hand side of \eqref{eq:itp}.
Recalling the definition of $\Phi$ in \eqref{eq:altPhi}, we write it out explicitly and
perform a series of straightforward manipulations:
  \begin{align}
  \frac{\sum_{i=1}^{k-1}\Phi_{ii}}{k}&-\frac{\sum_{i\neq j}\Phi_{ij} }{k(k-1)} \nonumber\\
& = \frac{n}{k(k-1)} \sum_{i=1}^L \frac{1}{q_i} 
\Big( (k-1)  \sum_{j=1}^{k-1} (q_{ij}-q_{ik})^2  - 
\Big(\sum_{j=1}^{k-1} \Big(\sum_{1\le j'\le k-1, j'\neq j} (q_{ij'}-q_{ik}) \Big) (q_{ij}-q_{ik}) \Big)  \Big)\nonumber \\
  & = \frac{n}{k(k-1)} \sum_{i=1}^L \frac{1}{q_i} 
\Big(  k  \sum_{j=1}^{k-1} (q_{ij}-q_{ik})^2  - 
\Big(\sum_{j=1}^{k-1} \Big(\sum_{j'=1}^{k-1} (q_{ij'}-q_{ik}) \Big) (q_{ij}-q_{ik}) \Big)  \Big)\nonumber \\
& = \frac{n}{k(k-1)} \sum_{i=1}^L \frac{1}{q_i} 
\Big(  k  \sum_{j=1}^{k-1} (q_{ij}-q_{ik})^2  - 
\Big(\sum_{j=1}^{k-1}  (q_{ij}-q_{ik}) \Big)^2  \Big) \nonumber\\
& = \frac{n}{k(k-1)} \sum_{i=1}^L \frac{1}{q_i} 
\Big\{  k  \sum_{j=1}^{k-1} q_{ij}^2 - 2kq_{ik} \sum_{j=1}^{k-1} q_{ij} 
+ k(k-1) q_{ik}^2  \nonumber \\
 & \hspace*{1in}   - \Big( \Big( \sum_{j=1}^{k-1} q_{ij} \Big)^2 
-2(k-1) q_{ik} \sum_{j=1}^{k-1} q_{ij} + (k-1)^2 q_{ik}^2 \Big) \Big\} \nonumber\\
& =   \frac{n}{k(k-1)} \sum_{i=1}^L \frac{1}{q_i} 
\Big(  k  \sum_{j=1}^{k-1} q_{ij}^2 - 2q_{ik} \sum_{j=1}^{k-1} q_{ij} 
+ (k-1) q_{ik}^2 
 -  \Big( \sum_{j=1}^{k-1} q_{ij} \Big)^2  \Big)\nonumber \\
& =  \frac{n}{k(k-1)} \sum_{i=1}^L \frac{1}{q_i} 
\Big(  k  \sum_{j=1}^{k} q_{ij}^2 
 -  \Big( \sum_{j=1}^{k} q_{ij} \Big)^2  \Big) \nonumber\\
& = \frac{n}{k(k-1)} \sum_{i=1}^L  q_i 
\Big(  k  \sum_{j=1}^{k} \Big( \frac{q_{ij}}{q_i} \Big)^2 
 -  k^2  \Big)  \label{eq:rhp1}\\
& \le \frac{n}{k(k-1)} k^2 (e^{\epsilon} - 1)^2  
 \frac{d^\ast (k-d^\ast)} {(d^\ast e^{\epsilon} + k - d^\ast)^2}
 \sum_{i=1}^L  q_i \label{eq:rhp2} \\
& = \frac{n}{k-1} k (e^{\epsilon} - 1)^2  
 \frac{d^\ast (k-d^\ast)} {(d^\ast e^{\epsilon} + k - d^\ast)^2} \label{eq:rhp3}
 \end{align}
  where \eqref{eq:rhp1} follows by definition of $q_i$ in Table~\ref{definitions}, and the inequality used to arrive at
\eqref{eq:rhp2} is proved in \eqref{eq:L}, Appendix~\ref{app:rhp1}.

Substituting \eqref{eq:rhp3} into \eqref{eq:itp}, we obtain inequality \eqref{eq:tr}.

Now we are ready to prove \eqref{eq:indv} using \eqref{eq:ob1} and \eqref{eq:tr}. For this, we set 
$$
C(k,\epsilon) :=
C_4(k,\epsilon) M(k,\epsilon)
+ C_3(k,\epsilon).
$$
We obtain
$$
r_{k,n}^{\ell_2^2}(\mathbi{Q}) 
\ge  \sum_{i=1}^k \mathbb{E}  (U_i- \mathbb{E} (U_i | Y^n ) )^2  \ge
\frac 1n M(k,\epsilon)
 - \frac{C(k,\epsilon)} {n^{14/13} }.
$$
This completes the proof of \eqref{eq:indv} for the case $\delta \ge \delta_0$.

\subsection{\bf Case 2: $\delta < \delta_0$}

This case is much easier to handle: We can rely on a straightforward application of Le Cam's method \cite{LeCam12} (see also 
\cite[Lemma 1]{Yu97}). 
 
Our goal is again to prove \eqref{eq:indv}. We use standard distance functions on distributions defined on finite sets $\cY.$ The {\em KL divergence} between
two such distributions $\mathbi{m}_1$ and $\mathbi{m}_2$ is defined as
$$
D_{\kl}(\mathbi{m}_1 \| \mathbi{m}_2) := \sum_{y\in\cY}  \mathbi{m}_1(y) \log \frac{\mathbi{m}_1(y)}
{\mathbi{m}_2(y)}.
$$
The {\em total variation distance} between $\mathbi{m}_1$ and $\mathbi{m}_2$ is defined as
\begin{equation}\label{eq:deftv}
\| \mathbi{m}_1 - \mathbi{m}_2 \|_{\TV} := \max_{A\subseteq \cY} |\mathbi{m}_1(A) - \mathbi{m}_2(A)|
=\frac{1}{2}\sum_{y\in\cY} |\mathbi{m}_1(y) - \mathbi{m}_2(y)|.
\end{equation}

According to \eqref{eq:defdt}, there is $\tilde{\mathbi{u}} \in\mathbb{R}^{k-1}$ such that\footnote{we again use the convention that
$\tilde{u}_k = -\sum_{i=1}^{k-1} \tilde{u}_i$.}
\begin{equation}\label{eq:tud}
\sum_{i=1}^k \tilde{u}_i^2=1 \text{~and~}
\Big(\sum_{i=1}^L \frac{1}{q_i} \Big(\sum_{j=1}^k \tilde{u}_j q_{ij} \Big)^2\Big)^{1/2} = \delta.
\end{equation}
With this $\tilde{\mathbi{u}}$ in mind, let
  $$
\mathbi{p}_1 = \mathbi{p}_U \text{~and~}
\mathbi{p}_2 = \mathbi{p}_U + \tilde{\mathbi{u}}/\sqrt{n \delta^2}
  $$
where as before, $\mathbi{p}_U=(1/k,1/k,\dots,1/k)$ denotes the uniform pmf.

Note that some of the coordinates of $\tilde{\mathbi{u}}$ can be negative, but we can ensure that
$\mathbi{p}_2 \in \Delta_k$ for all $n\ge N(k,\epsilon)$ as long as $N(k,\epsilon)$ is sufficiently large. Consequently,
$$
r_{k,n}^{\ell_2^2} (\mathbi{Q}) = \inf_{\hat{\mathbi{p}}} \sup_{\mathbi{p}\in \Delta_k} 
\underset{Y^n\sim (\mathbi{p} \mathbi{Q})^n}{\mathbb{E}} \ell_2^2(\hat{\mathbi{p}}(Y^n), \mathbi{p})
\ge \frac{1}{2} \inf_{\hat{\mathbi{p}}}
\Big( \underset{Y^n\sim \mathbi{m}_1^n}{\mathbb{E}} \ell_2^2(\hat{\mathbi{p}}(Y^n), \mathbi{p}_1) +
\underset{Y^n\sim \mathbi{m}_2^n}{\mathbb{E}} \ell_2^2(\hat{\mathbi{p}}(Y^n), \mathbi{p}_2 ) \Big),
$$
where $\mathbi{m}_1=\mathbi{p}_1 \mathbi{Q}$ and 
$\mathbi{m}_2 = \mathbi{p}_2 \mathbi{Q}$.

For a given estimator $\hat{\mathbi{p}}:\cY^n\to\mathbb{R}^k$,  define the set 
   $$
   \cK_1:= \{y^n\in\cY^n: 
\ell_2(\hat{\mathbi{p}}(y^n),\mathbi{p}_1)
\ge \ell_2(\hat{\mathbi{p}}(y^n),\mathbi{p}_2)\}
   $$
By the triangle inequality, we have
\begin{align*}
\ell_2(\hat{\mathbi{p}}(y^n),\mathbi{p}_1) \ge \frac{1}{2}
\ell_2(\mathbi{p}_1,\mathbi{p}_2) = \frac{1}{2\sqrt{n\delta^2}}
\text{~for all~} y^n\in \cK_1, \\
\ell_2(\hat{\mathbi{p}}(y^n),\mathbi{p}_2) \ge \frac{1}{2}
\ell_2(\mathbi{p}_1,\mathbi{p}_2) = \frac{1}{2\sqrt{n\delta^2}}
\text{~for all~} y^n\in \cK_1^c.
\end{align*}
Therefore, for any estimator $\hat{\mathbi{p}}$,
\begin{align*}
E&:=\underset{Y^n\sim \mathbi{m}_1^n}{\mathbb{E}} \ell_2^2(\hat{\mathbi{p}}(Y^n), \mathbi{p}_1) +
\underset{Y^n\sim \mathbi{m}_2^n}{\mathbb{E}} \ell_2^2(\hat{\mathbi{p}}(Y^n), \mathbi{p}_2 )\\ 
&\ge  \frac{1}{4 n\delta^2} \Big(\mathbi{m}_1^n (Y^n\in \cK_1)
+ \mathbi{m}_2^n (Y^n\in \cK_1^c) \Big) \\
&=  \frac{1}{4 n\delta^2} \Big(1 - \mathbi{m}_2^n (Y^n\in \cK_1) + \mathbi{m}_1^n (Y^n\in \cK_1) \Big) \\
&\ge   \frac{1}{4 n\delta^2} \Big(1 - 
\sup_{\cK\subseteq \cY^n}\Big(\mathbi{m}_2^n (Y^n\in \cK) - \mathbi{m}_1^n (Y^n\in \cK) \Big) \Big) \\
&= \frac{1}{4 n\delta^2} \Big(1 - 
\| \mathbi{m}_2^n - \mathbi{m}_1^n \|_{\TV} \Big) ,
\end{align*}
where the last step follows by definition \eqref{eq:deftv}. Using Pinsker's inequality\footnote{Pinsker's inequality 
asserts that $\|P_1-P_2\|_{\TV}^2\le \frac12 D_{\kl}(P_1\|P_2)$ for any two probability measures $P_1,P_2$.}, we obtain
\begin{align*}
E&\ge \frac{1}{4 n\delta^2} \Big(1 - 
\sqrt{\frac{1}{2}D_{\kl}(\mathbi{m}_2^n \| \mathbi{m}_1^n)} \Big) 
= \frac{1}{4 n\delta^2} \Big(1 - 
\sqrt{\frac{n}{2}D_{\kl}(\mathbi{m}_2 \| \mathbi{m}_1)} \Big).
\end{align*}
Let us write out this expression explicitly, recalling the fact that the original output alphabet is $\cY=\{1,2,\dots,L'\}$ and using in succession \eqref{eq:qpr}, \eqref{eq:eqclass}, \eqref{eq:tq}:
    \begin{align*}
E&\ge \frac{1}{4 n\delta^2} \Biggl(1 - 
\sqrt{\frac{n}{2} \sum_{i=1}^{L'} \Biggl\{\Big(
\sum_{j=1}^k \mathbi{Q}(i|j) \Big(\frac{1}{k}+\frac{\tilde{u}_j}{\sqrt{n\delta^2}} \Big)  \Big)
\log \frac{\sum_{j=1}^k \mathbi{Q}(i|j) \Big(\frac{1}{k}+\frac{\tilde{u}_j}{\sqrt{n\delta^2}} \Big)}
{\sum_{j=1}^k \mathbi{Q}(i|j) \frac{1}{k} }  \Biggr\}  } \Biggr) \\
&= \frac{1}{4 n\delta^2} \Biggl(1 - 
\sqrt{\frac{n}{2} \sum_{i=1}^{L'} \Biggl\{
\Big(q_i'+\sum_{j=1}^k \mathbi{Q}(i|j) \frac{\tilde{u}_j}{\sqrt{n\delta^2}} \Big) 
\log \frac{q_i'+\sum_{j=1}^k \mathbi{Q}(i|j) \frac{\tilde{u}_j}{\sqrt{n\delta^2}}  }
{q_i' }  \Biggr\}  } \Biggr) \\
&= \frac{1}{4 n\delta^2} \Biggl(1 - 
\sqrt{\frac{n}{2} \sum_{i=1}^{L'} \Biggl\{ q_i'
\Big(1 + \sum_{j=1}^k \frac{\mathbi{Q}(i|j)}{q_i'} \frac{\tilde{u}_j}{\sqrt{n\delta^2}} \Big) 
\log \Big( 1+\sum_{j=1}^k \frac{\mathbi{Q}(i|j)}{q_i'} \frac{\tilde{u}_j}{\sqrt{n\delta^2}}  \Big)  \Biggr\}  } \Biggr) \\
&= \frac{1}{4 n\delta^2} \Biggl(1 - 
\sqrt{\frac{n}{2} \sum_{i=1}^{L} \Biggl\{ \Big(\sum_{a\in A_i} q_a' \Big)
\Big(1 + \sum_{j=1}^k \frac{q_{ij}}{q_i} \frac{\tilde{u}_j}{\sqrt{n\delta^2}} \Big) 
\log \Big( 1+\sum_{j=1}^k \frac{q_{ij}}{q_i} \frac{\tilde{u}_j}{\sqrt{n\delta^2}}    \Big) \Biggr\} } \Biggr) \\
&= \frac{1}{4 n\delta^2} \Biggl(1 - 
\sqrt{\frac{n}{2} \sum_{i=1}^{L}  \Biggl\{
\Big(q_i + \sum_{j=1}^k q_{ij} \frac{\tilde{u}_j}{\sqrt{n\delta^2}} \Big) 
\log \Big( 1+\sum_{j=1}^k \frac{q_{ij}}{q_i} \frac{\tilde{u}_j}{\sqrt{n\delta^2}}  \Big)  \Biggr\}  } \Biggr) .
\end{align*}
Bounding above the logarithm by $\log(1+1)\le x$, we further obtain
\begin{align}
E&\ge  \frac{1}{4 n\delta^2} \Biggl(1 - 
\sqrt{\frac{n}{2} \sum_{i=1}^{L}  \Biggl\{
\Big(q_i + \sum_{j=1}^k q_{ij} \frac{\tilde{u}_j}{\sqrt{n\delta^2}} \Big) 
 \Big( \sum_{j=1}^k \frac{q_{ij}}{q_i} \frac{\tilde{u}_j}{\sqrt{n\delta^2}}  \Big)  \Biggr\}  } \Biggr) \nonumber\\
&= \frac{1}{4 n\delta^2} \Biggl(1 - 
\sqrt{\frac{n}{2} \Big( \sum_{i=1}^{L}\sum_{j=1}^k q_{ij} \frac{\tilde{u}_j}{\sqrt{n\delta^2}} 
+  \sum_{i=1}^{L} \frac{1}{q_i} \Big(\sum_{j=1}^k q_{ij} \frac{\tilde{u}_j}{\sqrt{n\delta^2}} \Big)^2  \Big) } \Biggr) \nonumber\\
& = \frac{1}{4 n\delta^2} \Biggl(1 - 
\sqrt{\frac{1}{2\delta^2}   \sum_{i=1}^{L} \frac{1}{q_i} \Big(\sum_{j=1}^k q_{ij} \tilde{u}_j \Big)^2   } \Biggr)
\label{eq:h} \\
&= \frac{1}{4 n\delta^2} \Big(1 - \sqrt{\frac{1}{2}} \Big)
> \frac{1}{16 n\delta^2} > \frac{1}{16 n\delta_0^2} \label{eq:i}\\
&= \frac{2}{n} M(k,\epsilon) \label{eq:j} \\
&{\ge \frac 1n M(k,\epsilon)- \frac{C(k,\epsilon)} {n^{14/13} }}
 , \nonumber
\end{align}
where \eqref{eq:h} follows from \eqref{eq:qu}, namely from the equality 
$\sum_{i=1}^{L}\sum_{j=1}^k q_{ij} \tilde{u}_j   = 0$; \eqref{eq:i} follows from the definition of $\tilde{\mathbi{u}}$ in
\eqref{eq:tud}; \eqref{eq:j}
follows from \eqref{eq:defd0}. (In the last step we make a somewhat arbitrary transition to match the inequality \eqref{eq:ob1} 
established for the case $\delta\ge \delta_0$.)

This completes the proof of \eqref{eq:indv} for the case $\delta < \delta_0$.

\section{Outlook: Open questions}\label{Sect:FW}
\subsection{Uniqueness of the optimal privatization scheme}
We believe that our privatization scheme $\mathbi{Q}_{k,\epsilon,d^\ast}$ is essentially the unique optimal choice.
More precisely, recall that we reduced the output alphabet from the original set to the set of equivalence classes; see 
Sec.~\ref{Sect:merge}. We conjecture that, under the assumption that $\mathbi{Q} \in \cD_{\epsilon,E}$ and after
the alphabet reduction, any scheme $\mathbi{Q}$ that is different from $\mathbi{Q}_{k,\epsilon,d^\ast},$
will entail a strictly larger estimation loss value. In other words, for $n$ large enough
$$
r_{k,n}^{\ell_2^2} (\mathbi{Q}) \ge r_{k,n}^{\ell_2^2} (\mathbi{Q}_{k,\epsilon,d^\ast}, \hat{\mathbi{p}}),
$$
where $\hat{\mathbi{p}}$ is given in \eqref{eq:emp}.
Formally, we can phrase this conjecture as follows: 
\begin{conjecture}\label{cj:u1}
For $\mathbi{Q} \in \cD_{\epsilon,E}$,
\begin{equation}\label{eq:fut}
\lim_{n\to \infty} n r_{k,n}^{\ell_2^2} (\mathbi{Q}) =
 M(k,\epsilon)
\end{equation}
if and only if $\mathbi{Q} = \mathbi{Q}_{k,\epsilon,d^\ast}$ (after accounting for the alphabet reduction).
\end{conjecture}

Let us list some necessary conditions for this to hold. 

\vspace*{.1in}$(i)$ From the proof in Section~\ref{Sect:Main} we know that
\begin{equation}\label{eq:unb}
\liminf_{n\to \infty} n r_{k,n}^{\ell_2^2} (\mathbi{Q}) \ge
n ( \tr(\Phi^{-1}) + \mathbi{1}^T \Phi^{-1} \mathbi{1} ).
\end{equation}
Therefore,  a necessary condition for \eqref{eq:fut} to hold is
\begin{equation}\label{eq:futA}
 \tr(\Phi^{-1}) + \mathbi{1}^T \Phi^{-1} \mathbi{1} =\frac 1n  M(k,\epsilon).
\end{equation}
According to \eqref{eq:rhp2}, \eqref{eq:L}, equality in \eqref{eq:futA} implies that any (asymptotically) optimal privatization scheme satisfies the condition that each vector $(Q(i|j), j=1,\dots,k)$ is proportional to one of the vectors in the set 
$\{(v_1,v_2,\dots,v_k) \in \{1, e^\epsilon\}^k: v_1+\dots+v_k=d^\ast e^\epsilon + k-d^\ast \}$, i.e., after normalizing, it contains exactly $d^\ast$ entries of $e^\epsilon$ and $k-d^\ast$ entries $1$. 

\vspace*{.1in}
$(ii)$
The converse part of Prop.~\ref{Prop:Zac} (specifically, the claim in \eqref{eq:eqcon}) gives another set of necessary conditions for \eqref{eq:futA} to hold.

\vspace*{.1in}
$(iii)$
Note that \eqref{eq:unb} is obtained by choosing $\mathbi{p}$ from a neighborhood of the uniform distribution $\mathbi{p}_U$. A similar bound can be obtained by choosing $\mathbi{p}$ to be in the neighborhood of any point in the probability simplex $\Delta_k$.
Formally speaking, the observable random variables in our problem are $Y^n$, and the unknown parameters are
$(p_1,p_2,\dots,p_{k-1})$. Denoting by $I(p_1,p_2,\dots,p_{k-1})$ the Fisher information matrix. As shown in Appendix~\ref{ap:LAN}, $$
\Phi(n,\mathbi{Q}) = I(1/k,1/k,\dots,1/k).
$$
Similarly to \eqref{eq:unb}, one can show that
   $$
\liminf_{n\to \infty} n r_{k,n}^{\ell_2^2} (\mathbi{Q}) \ge
n  \tr( (I(p_1,p_2,\dots,p_{k-1}))^{-1}  + \mathbi{1}^T (I(p_1,p_2,\dots,p_{k-1}))^{-1} \mathbi{1} )
   $$
for all $(p_1,p_2,\dots,p_{k-1},1-p_1-\dots - p_{k-1})\in \Delta_k$.
Therefore another set of necessary conditions for \eqref{eq:fut} to hold is
$$
\tr( (I(p_1,p_2,\dots,p_{k-1}))^{-1} ) + \mathbi{1}^T (I(p_1,p_2,\dots,p_{k-1}))^{-1} \mathbi{1} 
\le \frac 1n M(k,\epsilon)
$$
for all $(p_1,p_2,\dots,p_{k-1},1-p_1-\dots - p_{k-1})\in \Delta_k$.

We believe that the conditions listed above imply the uniqueness claim of the privatization mechanism 
$\mathbi{Q}_{k,\epsilon,d^\ast}$.

We conclude by suggesting an even stronger conjecture which we also believe to be true.
\begin{conjecture}
For all $\mathbi{Q}$ with finite output alphabet,
$$
\lim_{n\to \infty} n r_{k,n}^{\ell_2^2} (\mathbi{Q}) =
 M(k,\epsilon)
$$
if and only if $\mathbi{Q} = \mathbi{Q}_{k,\epsilon,d^\ast}$ (after accounting for the alphabet reduction).
\end{conjecture}

\subsection{Asymptotically tight lower bound for the $\ell_1$ loss}
Another open question is to find an asymptotically optimal privatization mechanism/estimation procedure
for the $\ell_1$ loss. We similarly believe that our privatization scheme $\mathbi{Q}_{k,\epsilon,d^\ast}$ and the empirical estimator $\hat{\mathbi{p}}$ given by \eqref{eq:emp} are asymptotically optimal in this case as well. 
More precisely, in \cite{Ye17}, we have shown that
$$
r_{k,n}^{\ell_1} (\mathbi{Q}_{k,\epsilon,d}, \hat{\mathbi{p}})
=   \frac{k-1}{e^{\epsilon}-1}  \sqrt{\frac{2}{\pi n}}
\sqrt{\frac{(d e^{\epsilon} + k-d)^2 }{d(k-d)} } + o\Big(\frac 1{\sqrt n}\Big).
$$
It is clear that
$$
r_{k,n}^{\ell_1} (\mathbi{Q}_{k,\epsilon,d^\ast}, \hat{\mathbi{p}}) = 
\min_{1\le d \le k-1} r_{k,n}^{\ell_1} (\mathbi{Q}_{k,\epsilon,d}, \hat{\mathbi{p}}).
$$
Our conjecture is
\begin{equation}\label{eq:cj2}
\lim_{n\to \infty} \sqrt{n} r_{\epsilon,k,n}^{\ell_1}
= \frac{k-1}{e^{\epsilon}-1}  \sqrt{\frac{2}{\pi}}
\sqrt{\frac{(d^\ast e^{\epsilon} + k-d^\ast)^2 }{d^\ast(k-d^\ast)} }.
\end{equation}

Within the frame of our approach, the obstacle in the way of proving this conjecture can be described as follows. Given a positive definite matrix $M$, let $S(M):=\sqrt{\sum_i M_{ii}}.$ 
Similarly to \eqref{eq:unb}, one can show that
$$
\liminf_{n\to \infty} \sqrt{n} r_{k,n}^{\ell_1} (\mathbi{Q}) \ge
\sqrt{n} ( S(\Phi^{-1}) + (\mathbi{1}^T \Phi^{-1} \mathbi{1})^{1/2} ).
$$
To prove \eqref{eq:cj2}, we need a lower bound on the quantity
$ S(\Phi^{-1}) + (\mathbi{1}^T \Phi^{-1} \mathbi{1})^{1/2} $, which is similar to
\eqref{eq:itp}. So far we have not been able to find a useful bound of this kind.

\section*{acknowledgement}

We are grateful to our colleague Itzhak Tamo for his help with the proof of Prop.~\ref{Prop:Zac}.

\clearpage

\begin{center}{\Large\bf Appendices}\end{center}

\appendices
\addtocontents{toc}{\protect\setcounter{tocdepth}{0}}

\section{A calculus inequality}\label{ap:log}
\begin{proposition}\label{Prop:log}
For all $x\ge -\frac{2}{3},$ we have the following inequality:
$$
\Big| \log(1+x) - (x - \frac{x^2}{2}) \Big| \le \Big| x^3 \Big|
$$
\end{proposition}
\begin{proof}
Let 
$$
h_1(x)=\log(1+x) - (x - \frac{x^2}{2}), \text{~and~} 
h_2(x) = \Big| \log(1+x) - (x - \frac{x^2}{2}) \Big| - \Big| x^3 \Big|.
$$
Then
$$
h'_1(x)=\frac{x^2}{x+1} \ge 0 \text{~for all~} x>-1.
$$
Since $h_1(0)=0,$ we have $h_1(x)\ge 0$ for all $x>0,$ and $h_1(x) \le 0$ for all $-1<x<0.$
Consequently,
$$
h_2(x)=h_1(x)-x^3 \text{~for~} x > 0, \text{~and~} h_2(x)=x^3-h_1(x) \text{~for~} -1 < x<0.
$$
As a result,
$$
h'_2(x)=\frac{x^2}{x+1}- 3x^2 < 0 \text{~for~} x > 0, \text{~and~} h'_2(x)=3x^2 - \frac{x^2}{x+1} \ge 0 \text{~for~} -\frac{2}{3} \le x<0.
$$
Since $h_2(0)=0,$ we conclude that $h_2(x) \le 0$ for all $x\ge -\frac{2}{3}.$ 
\end{proof}

\section{Proof of Proposition \ref{Prop:E2}}\label{ap:E2}
We need the following simple proposition about the volume of ellipsoids.
\begin{proposition}\label{Prop:ellips}
Let $\Lambda$ be an $s\times s$ positive definite matrix. Let 
$
\{\mathbi{u} \in \mathbb{R}^s: \mathbi{u}^T \Lambda \mathbi{u} \le \alpha_1^2\}
$
and
$
\{\mathbi{u} \in \mathbb{R}^s: \mathbi{u}^T \Lambda \mathbi{u} \le \alpha_2^2\}
$
be two ellipsoids. Then the ratio of their volumes equals $(\alpha_1 / \alpha_2)^s$.
\end{proposition}
\begin{proof} Consider the ellipsoid $
B_{\Lambda}(\alpha):=\{\mathbi{u} \in \mathbb{R}^s: \mathbi{u}^T \Lambda \mathbi{u} \le \alpha^2\}.
$
Let $\lambda_1,\dots,\lambda_s$ be the eigenvalues of $\Lambda$, taken with multiplicities.
Since $\Lambda$ is positive definite, we can diagonalize it as 
$\Lambda=P^T DP$, where $P$ is an orthogonal matrix and $D=\text{diag}(\lambda_1,\dots,\lambda_s)$.
Then
\begin{align}
\text{Vol}(B_{\Lambda}(\alpha))&=\int_{\mathbi{u}^T \Lambda \mathbi{u} \le \alpha^2} d\mathbi{u} 
= \int_{\mathbi{v}^T \mathbi{v} \le 1} |\alpha P^T D^{-1/2}| d\mathbi{v} 
= \alpha^s |D^{-1/2}| \frac{\pi^{s/2}}{\Gamma(\frac{s}{2}+1)} \nonumber\\
&= \alpha^s \frac{\pi^{s/2}}{\Gamma(\frac{s}{2}+1)} \prod_{i=1}^s \lambda_i^{-1/2},
\label{eq:vols}
\end{align}
where the second equality follows by a change of variable $\mathbi{v}=\frac{1}{\alpha}D^{1/2}P\mathbi{u}$, and the third one 
uses the expression for the volume of the unit sphere in $\mathbb{R}^s.$
The proposition follows immediately from \eqref{eq:vols}.
\end{proof}
Now we are ready to prove Prop.~\ref{Prop:E2}.

{\em Proof of Prop.~\ref{Prop:E2}.}
As mentioned before, conditional on $\mathbi{U}=\mathbi{u},$ the random variable $t_i(Y^n)$ has binomial distribution $B(n, q_i + \sum_{j=1}^k u_j q_{ij})$.
With this in mind, using \eqref{eq:qu} several times, we have for all $\mathbi{u}\in B_1$ the following upper bound:
\begin{align}
 \mathbb{E} &\Big[ \sum_{i=1}^L \frac{n}{q_i} \Big(\sum_{j=1}^k u_j q_{ij} - \frac{V_i}{n} \Big)^2
\Big| \mathbi{U}=\mathbi{u} \Big] \nonumber\\
& = \mathbb{E} \Big[ \sum_{i=1}^L \frac{n}{q_i} \Big(\frac{t_i(Y^n)}{n} - \Big(q_i + \sum_{j=1}^k u_j q_{ij} \Big) \Big)^2
\Big| \mathbi{U}=\mathbi{u} \Big] \nonumber\\
&= \sum_{i=1}^L \frac{1}{q_i} \Big(q_i + \sum_{j=1}^k u_j q_{ij} \Big) \Big( 1 - q_i - \sum_{j=1}^k u_j q_{ij} \Big) \nonumber\\
& \le  \sum_{i=1}^L \frac{1}{q_i} q_i \Big( 1 - q_i - \sum_{j=1}^k u_j q_{ij} \Big)
+ \sum_{i=1}^L \frac{1}{q_i} \Big|\sum_{j=1}^k u_j q_{ij} \Big| \Big( 1 - q_i - \sum_{j=1}^k u_j q_{ij} \Big) \nonumber\\ 
& =L-1 + \sum_{i=1}^L  \Big|\sum_{j=1}^k \frac{u_j q_{ij}}{q_i} \Big| \Big( 1 - q_i - \sum_{j=1}^k u_j q_{ij} \Big) \nonumber\\
&\overset{(a)}{\le}  L-1 + \frac{k}{ n^{5/13}}\sum_{i=1}^L  \Big( 1 - q_i - \sum_{j=1}^k u_j q_{ij} \Big) \nonumber\\
& = (L-1) \Big( 1 + \frac{k}{n^{5/13}} \Big) < 2L \le 2^{k+1}, \label{eq:expL}
\end{align}
where $(a)$ follows from \eqref{eq:bdr}.

For all $\mathbi{u} \in B_2,$ we have
\begin{align}
P(Y^n\in E_2 | \mathbi{U}=\mathbi{u})
& \ge P \Big( \sum_{i=1}^L \frac{n}{q_i} \Big(\sum_{j=1}^k u_j q_{ij} - \frac{V_i}{n} \Big)^2
< n^{1/13} \Big| \mathbi{U}=\mathbi{u} \Big) \nonumber\\
& = 1 - P \Big( \sum_{i=1}^L \frac{n}{q_i} \Big(\sum_{j=1}^k u_j q_{ij} - \frac{V_i}{n} \Big)^2
\ge n^{1/13} \Big| \mathbi{U}=\mathbi{u} \Big) \nonumber\\
& \ge 1 - \frac{2^{k+1}}{n^{1/13}}, \label{eq:13}
\end{align}
where the last step follows by the Markov inequality and \eqref{eq:expL}.
By our assumption, $P_U(B_1)=1,$ so from \eqref{eq:B1},\eqref{eq:B2}, and Prop.~\ref{Prop:ellips} we obtain
  \begin{equation}\label{eq:PB2}
P(\mathbi{U}\in B_2) = \frac{\text{Vol}(B_2)}{\text{Vol}(B_1)}=\Big( 1 - \frac{3/\delta_0}{n^{1/13}} \Big)^{k-1} \ge 1 - \frac{3(k-1)/\delta_0}{n^{1/13}},
  \end{equation}
where the inequality follows from the fact that $(1-x)^k \ge 1-kx$ for all $0\le x < 1.$
Therefore the inequality holds for all $n\ge N(k,\epsilon)$ as long as we set 
$N(k,\epsilon)>(3/\delta_0)^{13}$.
Now using \eqref{eq:13} and \eqref{eq:PB2} we obtain
\begin{align*}
P(Y^n\in E_2) & = \int_{ B_1} 
P(Y^n\in E_2 | \mathbi{U}=\mathbi{u}) f_{\mathbi{U}}(\mathbi{u}) d \mathbi{u}
\ge \int_{ B_2} 
P(Y^n\in E_2 | \mathbi{U}=\mathbi{u}) f_{\mathbi{U}}(\mathbi{u}) d \mathbi{u} \\
& \ge \Big( 1 - \frac{2^{k+1}}{n^{1/13}} \Big) \int_{ B_2} 
 f_{\mathbi{U}}(\mathbi{u}) d \mathbi{u} = \Big( 1 - \frac{2^{k+1}}{n^{1/13}} \Big) P(\mathbi{U}\in B_2) \\
& \ge 1 - \frac{2^{k+1} + 3k/\delta_0}{n^{1/13}}.
\end{align*}
This completes the proof of Prop.~\ref{Prop:E2}.

\section{Proof of Proposition~\ref{Prop:inc}}\label{ap:inc}
For every $y^n \in E_2,$ there exists $\tilde{\mathbi{u}} \in B_2$ such that
\begin{equation}\label{eq:tiluag}
 \sqrt{h_{\mathbi{v}}(\tilde{\mathbi{u}})}= \Big(\sum_{i=1}^L \frac{n}{q_i} \Big(\sum_{j=1}^k \tilde{u}_j q_{ij} - \frac{v_i}{n} \Big)^2 \Big)^{1/2} < n^{1/26}.
\end{equation}
By definition of $B_2,$ we have $\Big(\sum_{i=1}^k \tilde{u}_i^2 \Big)^{1/2} < \frac{1}{n^{5/13}} - \frac{3/\delta_0}{n^{6/13}}.$ According to triangle inequality,
$$
\Big(\sum_{i=1}^k (u_i - \tilde{u}_i)^2 \Big)^{1/2} \ge 
\Big(\sum_{i=1}^k u_i^2 \Big)^{1/2} - \Big(\sum_{i=1}^k \tilde{u}_i^2 \Big)^{1/2}
> \alpha - \frac{1}{n^{5/13}}+ \frac{3/\delta_0}{n^{6/13}} \text{~for all~}
\mathbi{u} \notin B(\alpha).
$$
By \eqref{eq:defdelta},
$$
\Big(\sum_{i=1}^L \frac{1}{q_i} \Big(\sum_{j=1}^k (u_j - \tilde{u}_j) q_{ij} \Big)^2\Big)^{1/2}
\ge \delta_0 \Big(\sum_{i=1}^k (u_i - \tilde{u}_i)^2 \Big)^{1/2}
> \delta_0 (\alpha - \frac{1}{n^{5/13}} ) + \frac{3}{n^{6/13}} \text{~for all~}
\mathbi{u} \notin B(\alpha).
$$
As a result,
$$
\Big(\sum_{i=1}^L \frac{n}{q_i} \Big(\sum_{j=1}^k (u_j - \tilde{u}_j) q_{ij} \Big)^2\Big)^{1/2}
> \delta_0 n^{1/2} (\alpha - n^{-5/13} ) + 3 n^{1/26} \text{~for all~}
\mathbi{u} \notin B(\alpha).
$$
Again by triangle inequality,
\begin{equation}\label{eq:utriag}
\begin{aligned}
 \sqrt{h_{\mathbi{v}}(\mathbi{u})} = & \Big(\sum_{i=1}^L \frac{n}{q_i} \Big(\sum_{j=1}^k u_j q_{ij} - \frac{v_i}{n} \Big)^2 \Big)^{1/2} \\
\ge & \Big(\sum_{i=1}^L \frac{n}{q_i} \Big(\sum_{j=1}^k (u_j - \tilde{u}_j) q_{ij} \Big)^2\Big)^{1/2} -
\Big(\sum_{i=1}^L \frac{n}{q_i} \Big(\sum_{j=1}^k \tilde{u}_j q_{ij} - \frac{v_i}{n} \Big)^2 \Big)^{1/2} \\
> & \delta_0 n^{1/2} (\alpha - n^{-5/13} ) + 2 n^{1/26} \text{~for all~}
\mathbi{u} \notin B(\alpha).
\end{aligned}
\end{equation}
Combining \eqref{eq:tiluag} and \eqref{eq:utriag}, we have
$$
 \sqrt{h_{\mathbi{v}}(\mathbi{u})} - \sqrt{h_{\mathbi{v}}(\tilde{\mathbi{u}})} >
\delta_0 n^{1/2} (\alpha - n^{-5/13} ) + n^{1/26}  \text{~for all~}
\mathbi{u} \notin B(\alpha).
$$
Thus we conclude that for every $y^n \in E_2,$ there exists $\tilde{\mathbi{u}} \in B_2$ such that
$$
(B(\alpha))^c \subseteq E_{\mathbi{v}}(\tilde{\mathbi{u}}, \delta_0 n^{1/2} (\alpha - n^{-5/13} ) + n^{1/26}).
$$
This completes the proof of Prop.~\ref{Prop:inc}.

\section{Proof of Proposition \ref{Prop:U}}\label{ap:concen}
We will rely on a concentration result for random Gaussian vectors.
\begin{proposition}\label{Prop:Z}
Let $\mathbi{Z}\sim \cN(\mathbi{0}, I_s)$ be a standard Gaussian random vector. We have
\begin{equation}\label{eq:Z2}
P(\|\mathbi{Z}\|_2 \ge \alpha) \le e^{-\frac{(\alpha-\sqrt{s})^2}{2}} \text{~for all~} \alpha \ge \sqrt{s}.
\end{equation}
\end{proposition}
This is a very special case of the following general concentration inequality due to Sudakov and Tsirel'son \cite{Sudakov78} and Borell \cite{Borell75} (see also Pisier and Maurey \cite[page 176]{Pisier86}).
Recall that a function $f:\mathbb{R}^s \to \mathbb{R}$ is called $\rho$-Lipschitz if
$$
|f(\mathbi{z})-f(\mathbi{z}')| \le \rho \| \mathbi{z} - \mathbi{z}' \|_2 
\text{~for all~} \mathbi{z},\mathbi{z}'\in\mathbb{R}^s.
$$
\begin{theorem}[\cite{Sudakov78,Borell75}]
For a $\rho$-Lipschitz function $f:\mathbb{R}^s \to \mathbb{R}$
and a standard Gaussian random vector $\mathbi{Z} \sim \cN(\mathbi{0}, I_s),$ the random variable $f(\mathbi{Z})$ satisfies the following concentration inequality:
\begin{equation}\label{eq:conc}
P(f(\mathbi{Z})-\mathbb{E} f(\mathbi{Z}) \ge \alpha) \le e^{-\frac{\alpha^2}{2\rho^2}} \text{~for all~} \alpha \ge 0.
\end{equation}
\end{theorem}
{\em Proof of Prop.~\ref {Prop:Z}:} 
Take $f$ above to be $f(\mathbi{z})=\|\mathbi{z}\|_2$ for all $\mathbi{z}\in \mathbb{R}^s,$ 
and note that it is $1$-Lipschitz by the triangle inequality, so \eqref{eq:conc} holds true. We have
$$
P(\|\mathbi{Z}\|_2-\mathbb{E} \|\mathbi{Z}\|_2 \ge \alpha) \le e^{-\frac{\alpha^2}{2}} \text{~for all~} \alpha \ge 0.
$$
By Jensen's inequality,
$$
\mathbb{E} \|\mathbi{Z}\|_2 = \mathbb{E} \sqrt{Z_1^2+\dots+Z_s^2} \le \sqrt {\mathbb{E} (Z_1^2+\dots+Z_s^2)} =\sqrt{s}.
$$
Therefore,
$$
P(\|\mathbi{Z}\|_2- \sqrt{s} \ge \alpha) \le
P(\|\mathbi{Z}\|_2-\mathbb{E} \|\mathbi{Z}\|_2 \ge \alpha) \le e^{-\frac{\alpha^2}{2}} \text{~for all~} \alpha \ge 0.
$$
\qed


{\em Proof of Prop.~\ref{Prop:U}.}
Observe that $\sqrt{h(\tilde{\mathbi{u}})} \ge \sqrt{h(\mathbi{t})} = \sqrt{C}$ for all $\tilde{\mathbi{u}} \in \mathbb{R}^s,$
and so $E(\tilde{\mathbi{u}}, \alpha) \subseteq E(\mathbi{t}, \alpha).$ 
Define the following set:
$$
E(\alpha) = \{\mathbi{u} \in \mathbb{R}^s: \sqrt{(\mathbi{u}-\mathbi{t})^T \Phi  (\mathbi{u}-\mathbi{t})} >  \alpha \}.
$$
Since $\sqrt{h(\mathbi{u})} \le \sqrt{(\mathbi{u}-\mathbi{t})^T \Phi  (\mathbi{u}-\mathbi{t})} + \sqrt{C}$
for all $\mathbi{u} \in \mathbb{R}^s,$ we have $E(\mathbi{t}, \alpha) \subseteq E(\alpha).$ 
Let $\mathbi{U} \sim \cN(\mathbi{t}, \Phi^{-1})$ be an $s$-dimensional Gaussian random vector with mean vector $\mathbi{t}$ and covariance matrix $\Phi^{-1}.$ 
Let $\mathbi{Z}=\Phi^{1/2} (\mathbi{U}-\mathbi{t}).$ Then $\mathbi{Z}\sim \cN(\mathbi{0}, I_s)$. Indeed, the mean vector 
of $\mathbi{Z}$ is trivially zero, and the covariance matrix is found as 
$$
\mathbb{E} (\mathbi{Z} \mathbi{Z}^T) = \mathbb{E} (\Phi^{1/2} (\mathbi{U}-\mathbi{t})  (\mathbi{U}-\mathbi{t})^T \Phi^{1/2}) = \Phi^{1/2} \Phi^{-1} \Phi^{1/2} = I_s,
$$
respectively. We obtain, for all $\alpha\ge \sqrt s$,
\begin{align*}
\frac{\int_{E(\tilde{\mathbi{u}}, \alpha)} \exp(-\frac{1}{2}h(\mathbi{u})) d\mathbi{u}}{\int_{ \mathbb{R}^s} \exp(-\frac{1}{2}h(\mathbi{u})) d\mathbi{u}}
 &= \frac{\int_{E(\tilde{\mathbi{u}}, \alpha)} \exp(-\frac{1}{2} (\mathbi{u}-\mathbi{t})^T \Phi  (\mathbi{u}-\mathbi{t})) d\mathbi{u}}{\int_{\mathbb{R}^s} \exp(-\frac{1}{2} (\mathbi{u}-\mathbi{t})^T \Phi  (\mathbi{u}-\mathbi{t})) d\mathbi{u}} \\
 &=  P(\mathbi{U}\in E(\tilde{\mathbi{u}}, \alpha)) 
\le P(\mathbi{U}\in E(\alpha))\\
&=  P (\sqrt{(\mathbi{U}-\mathbi{t})^T \Phi  (\mathbi{U}-\mathbi{t})} >  \alpha) \\
&=  P (\sqrt{\mathbi{Z}^T \mathbi{Z}} > \alpha)\\
&\le  e^{-\frac{(\alpha-\sqrt{s})^2}{2}}.
\end{align*}
where the last step follows by \eqref{eq:Z2}. \qed

\section{Proof of Step \text{\rm (b)} in Eq.~\eqref{eq:bdE}}\label{ap:Gauss}

Let $\Lambda$ be an $s\times s$ positive definite matrix, and let $\mathbi{t}$ be a column vector in $\mathbb{R}^s.$
It is easily seen that for a Gaussian random variable $X\sim \cN(\mu, \sigma^2),$
$$
\mathbb{E}|X| \le |\mu| + \sqrt{\frac{2 \sigma^2}{\pi}} .
$$
Therefore, we have
\begin{align*}
\|\mathbi{U}\|_1 &\le \|\mathbi{t}\|_1 + \sqrt{\frac{2}{\pi}} \sum_{i=1}^s \sqrt{\Var(U_i)} \le \sqrt{s} \|\mathbi{t}\|_2 + \sqrt{\frac{2s}{\pi}} \Big(\sum_{i=1}^s \Var(U_i) \Big)^{1/2} \\
& = \sqrt{s} \|\mathbi{t}\|_2 + \sqrt{\frac{2s}{\pi} \tr(\Lambda)}
\end{align*}
which is what we used to obtain the last line in \eqref{eq:bdE}.

\section{Proof of \eqref{eq:EU}}\label{ap:f1}
To prove the last estimate in \eqref{eq:EU}, we will show that there exists an integer $N(k,\epsilon)$ such that for every $n\ge N(k,\epsilon)$ and every $y^n \in E_2$,
  \begin{equation}\label{eq:F}
 \int_{ (B_1)^c} \Big(\sum_{i=1}^k u_i^2 \Big)^{1/2} f_1(\mathbi{u},y^n) d\mathbi{u} 
<  \frac{1}{n^{7/13}}.
  \end{equation}

Given $y^n\in \cY^n,$ define $\bar{\mathbi{U}}(y^n)=(\bar{U}_1(y^n),\dots,\bar{U}_{k-1}(y^n)) $ as a $(k-1)$-dimensional Gaussian random vector with density function $f_{\bar{\mathbi{U}}(y^n)}(\cdot)=  f_1(\cdot,y^n)$, mean vector $\Phi^{-1}\mathbi{w}$ and covariance matrix $\Phi^{-1}.$
Note $\Phi^{-1}\mathbi{w}$ depends on $y^n$.
   According to \eqref{eq:hvcv}, \eqref{eq:g2}, \eqref{eq:H}, and \eqref{eq:f1},
   $$
f_1(\bar{\mathbi{u}},y^n)=
\frac{ \exp(-\frac{1}{2}h_{\mathbi{v}}(\bar{\mathbi{u}} )) }{\int_{ \mathbb{R}^{k-1}} \exp(-\frac{1}{2} h_{\mathbi{v}}(\mathbi{u})) d\mathbi{u}}.
  $$
With this, we can use the result of Prop.~\ref{Prop:U}, taking $s=k-1.$ Namely, for all $\alpha\ge \sqrt{k-1}, 
\tilde{\mathbi{u}}\in {\mathbb{R}}^{k-1},$ inequality \eqref{eq:II} gives the estimate
$$
P  \big( \bar{\mathbi{U}}(y^n) \in E_{\mathbi{v}}(\tilde{\mathbi{u}}, \alpha)  \big) =
\int_{ E_{\mathbi{v}}(\tilde{\mathbi{u}}, \alpha)} f_1(\bar{\mathbi{u}},y^n) d\bar{\mathbi{u}}
\le \exp\Big(-\frac{1}{2} (\alpha-\sqrt{k-1})^2 \Big) .$$
For a given $y^n\in E_2$, take $\tilde{\mathbi{u}}$ whose existence is established in Prop.~\ref{Prop:inc}. Using \eqref{eq:BEag}, we deduce that for every $y^n \in E_2$ and $ \alpha \ge n^{-5/13}$
   \begin{equation}\label{eq:PUB}
P  \big(\bar{\mathbi{U}}(y^n) \in (B(\alpha))^c  \big) \le
\exp\Big\{-\frac{1}{2} \Big( \delta_0 n^{1/2} (\alpha - n^{-5/13} ) + n^{1/26}-\sqrt{k-1} \Big)^2 \Big\} .
    \end{equation}
At this point we wish to ensure that $n^{1/26}-\sqrt{k-1}>0$. Since $n$ is taken to be sufficiently large, in particular,
$n\ge N(k,\epsilon)$, at it suffices to ensure that  $N(k,\epsilon)>k^{13},$ which entails no loss of generality.

Define a random variable
$$
\bar{Z}(y^n)= \Big(\sum_{i=1}^{k-1} \big(\bar{U}_i(y^n) \big)^2
+ \big( \sum_{i=1}^{k-1}\bar{U}_i(y^n) \big)^2\Big)^{1/2}.
$$
Observe that 
  $$
  \{\bar{\mathbi{U}}(y^n) \in (B(\alpha))^c\}=\{\bar{Z}(y^n) \ge \alpha\}$$
  and thus by \eqref{eq:PUB}, for every $y^n \in E_2$ and $ \alpha \ge n^{-5/13},$
  \begin{equation}\label{eq:ubz}
P  \Big(\bar{Z}(y^n) \ge \alpha \Big) \le
\exp\Big\{-\frac{1}{2} \Big( \delta_0 n^{1/2} (\alpha - n^{-5/13} ) + n^{1/26}-\sqrt{k-1} \Big)^2 \Big\} .
\end{equation}
Note that this implies that $\lim_{\bar{z}\to\infty} \bar{z} P(\bar{Z}(y^n) \ge \bar{z})=0.$
Then we have, for every $y^n \in E_2,$
\begin{align}
\int_{ (B_1)^c} \Big(\sum_{i=1}^k u_i^2 \Big)^{1/2} &f_1(\mathbi{u},y^n) d\mathbi{u}  =  
{\int_{\bar{\mathbi{u}} \in (B_1)^c} 
\Big(\sum_{i=1}^{k-1}\bar{u}_i^2
+\Big(\sum_{i=1}^{k-1}\bar{u}_i\Big)^2\Big)^{1/2}
f_{\bar{\mathbi{U}}(y^n)}( \bar{\mathbi{u}} ) d\bar{\mathbi{u}} }\nonumber \\
& = \int_{\bar{z} \ge n^{-5/13}} 
\bar{z}
f_{\bar{Z}(y^n) }( \bar{z} ) d\bar{z}
= \int_{\bar{z} \ge n^{-5/13}} 
\bar{z} dP(\bar{Z}(y^n) \le \bar{z}) \nonumber \\
& = - \int_{\bar{z} \ge n^{-5/13}} 
\bar{z} dP(\bar{Z}(y^n) \ge \bar{z}) \nonumber \\
& = - \Big(  \bar{z} P(\bar{Z}(y^n) \ge \bar{z}) \Big|_{n^{-5/13}}^{\infty} -  \int_{\bar{z} \ge n^{-5/13}} 
 P(\bar{Z}(y^n) \ge \bar{z}) d \bar{z}  \Big) \nonumber \\
& = n^{-5/13} P(\bar{Z}(y^n) \ge n^{-5/13}) + \int_{\bar{z} \ge n^{-5/13}} 
 P(\bar{Z}(y^n) \ge \bar{z}) d \bar{z}
- \lim_{\bar{z}\to\infty} \bar{z} P(\bar{Z}(y^n) \ge \bar{z}) \nonumber \\
&\le  n^{-5/13} \exp\Big(-\frac{1}{2} \Big(  n^{1/26}-\sqrt{k-1} \Big)^2 \Big)\nonumber \\
&\hspace*{.5in}+ \int_{\bar{z} \ge n^{-5/13}}  \exp\Big(-\frac{1}{2} \Big( \delta_0 n^{1/2} (\bar{z} - n^{-5/13} ) + n^{1/26}-\sqrt{k-1} \Big)^2 \Big) d \bar{z}, \label{eq:cmp}
\end{align}
where in the last two steps we used \eqref{eq:ubz}. 

To bound the second term on the last line of \eqref{eq:cmp}, we interpret it as an integral of a (univariate) Gaussian random variable $Z$ with 
mean $n^{-5/13} - \delta_0^{-1}n^{-6/13} + \delta_0^{-1}n^{-1/2}\sqrt{k-1}$ and variance $\delta_0^{-2}n^{-1}.$ Then
   \begin{align}
 \int_{\bar{z} \ge n^{-5/13}}  &\exp\Big(-\frac{1}{2} \Big( \delta_0 n^{1/2} (\bar{z} - n^{-5/13} ) + n^{1/26}-\sqrt{k-1} \Big)^2 \Big) d \bar{z} \nonumber\\
& = \frac{\sqrt{2\pi}}{\delta_0 \sqrt{n}}
\int_{z \ge n^{-5/13}}  \frac{\delta_0 \sqrt{n}}{\sqrt{2\pi}}  \exp\Big(-\frac{\delta_0^2 n}{2} \Big( z - \big( n^{-5/13} - \delta_0^{-1} n^{-6/13} + \delta_0^{-1} n^{-1/2}\sqrt{k-1} \big) \Big)^2 \Big) d z \nonumber\\
& = \frac{\sqrt{2\pi}}{\delta_0 \sqrt{n}} P(Z\ge n^{-5/13})
= \frac{\sqrt{2\pi}}{\delta_0 \sqrt{n}} P\Big( Z - \mathbb{E} Z \ge \delta_0^{-1}n^{-6/13} - \delta_0^{-1}n^{-1/2}\sqrt{k-1} \Big)\nonumber\\ & {\le}\frac{\sqrt{2\pi}}{\delta_0 \sqrt{n}} 
\exp \Big(-\frac{\delta_0^2 n}{2} \big( \delta_0^{-1}n^{-6/13} - \delta_0^{-1}n^{-1/2}\sqrt{k-1} \big)^2  \Big) \nonumber\\
& = \frac{\sqrt{2\pi}}{\delta_0 \sqrt{n}} 
\exp \Big(-\frac{ 1 }{2} \big( n^{1/26} - \sqrt{k-1} \big)^2  \Big),\label{eq:aux}
    \end{align}
where the inequality follows from the Gaussian tail bound (a one-sided version of \eqref{eq:Chernoff}).
It remains to use \eqref{eq:aux} in \eqref{eq:cmp}: We obtain that for every $y^n \in E_2,$
   \begin{align*}
 \int_{ (B_1)^c} \Big(\sum_{i=1}^k u_i^2 \Big)^{1/2} f_1(\mathbi{u},y^n) d\mathbi{u} 
   &\le  \Big(\frac{1}{n^{5/13}} + \frac{\sqrt{2\pi}}{\delta_0 \sqrt{n}} \Big)
\exp \Big(-\frac{ 1 }{2} \big( n^{1/26} - \sqrt{k-1} \big)^2  \Big) \\
<    \frac{1}{n^{7/13}},
    \end{align*}
where the last inequality holds for all $n\ge N(k,\epsilon)$ as long as $N(k,\epsilon)$ is large enough.

\section{Bounding the integral in \eqref{eq:Cov}}\label{ap:vr}
Let 
$$
I:=\int_{ (B_1)^c} \sum_{i=1}^k \Big(u_i- \mathbb{E}\big( \bar{U}_i(y^n) \big) \Big)^2 f_1(\mathbi{u},y^n) d\mathbi{u}
$$
Here we prove the following bound:

{\em There exists an integer $N(k,\epsilon)$ such that
for every $n\ge N(k,\epsilon)$ and every $y^n \in E_2$,
   \begin{equation}\label{eq:I}
I  
\le   \frac{1}{n^{14/13}}.
   \end{equation}
}

Using $\bar U_k=-\sum_{i=1}^{k-1} \bar U_{i},$ we have
\begin{align*}
I = \int_{\bar{\mathbi{u}} \in (B_1)^c} \sum_{i=1}^{k-1} (\bar{u}_i &- \mathbb{E}( \bar{U}_i(y^n) ) )^2 f_1(\bar{\mathbi{u}},y^n) d\bar{\mathbi{u}}\\
&+  \int_{\bar{\mathbi{u}} \in (B_1)^c}  \Big(\sum_{i=1}^{k-1}  (\bar{u}_i- \mathbb{E}( \bar{U}_i(y^n) )) \Big)^2 f_1(\bar{\mathbi{u}},y^n) d\bar{\mathbi{u}}
\end{align*}
Since $(\sum_{i=1}^m x_i)^2\le m\sum_{i=1}^m x_i^2,$ we obtain
  $$
  I\le k \int_{\bar{\mathbi{u}} \in (B_1)^c} \sum_{i=1}^{k-1} \Big(\bar{u}_i- \mathbb{E}( \bar{U}_i(y^n) ) \Big)^2 f_1(\bar{\mathbi{u}},y^n) d\bar{\mathbi{u}}
  $$ 
Since $\mathbb{E}(\bar{\mathbi{U}}(y^n))=\Phi^{-1}\mathbi{w},$ we can use inequality \eqref{eq:Pn} which says that for every $\mathbi{z}\in {\mathbb R}^{k-1},$ $\mathbi{z}^T\Phi \mathbi{z}\ge n\delta_0^2\|\mathbi{z}\|^2.$ Taking $\mathbi{z}=(\bar{\mathbi{u}} - \Phi^{-1}\mathbi{w})$ we can continue as follows:
  \begin{align}
I&\le   k
\int_{\bar{\mathbi{u}} \in (B_1)^c} (\bar{\mathbi{u}} - \Phi^{-1}\mathbi{w})^T  
(\bar{\mathbi{u}} - \Phi^{-1}\mathbi{w}) f_1(\bar{\mathbi{u}},y^n) d\bar{\mathbi{u}} \nonumber\\
&\le \frac{k}{n \delta_0^2}
\int_{\bar{\mathbi{u}} \in (B_1)^c} (\bar{\mathbi{u}} - \Phi^{-1}\mathbi{w})^T  \Phi
(\bar{\mathbi{u}} - \Phi^{-1}\mathbi{w}) f_1(\bar{\mathbi{u}},y^n) d\bar{\mathbi{u}}, \label{eq:100}
\end{align}

To bound the integral in \eqref{eq:100}, for a given $\alpha \ge 0,$ we define another set
$$
E_{\mathbi{v}}(\alpha) = \{\mathbi{u} \in \mathbb{R}^{k-1}: (\mathbi{u}- \Phi^{-1}\mathbi{w} )^T \Phi  (\mathbi{u}- \Phi^{-1}\mathbi{w} ) >  \alpha \}.
$$
Now turning to the remark after \eqref{eq:hv}, we note that $h_{\mathbi{v}}(\Phi^{-1}\mathbi{w})=C\ge0.$ Thus for 
every $\tilde{\mathbi{u}} \in \mathbb{R}^{k-1}$ we obtain
\begin{align*}
(\mathbi{u}- \Phi^{-1}\mathbi{w} )^T \Phi  (\mathbi{u}- \Phi^{-1}\mathbi{w} ) 
& = h_{\mathbi{v}}(\mathbi{u}) - h_{\mathbi{v}}(\Phi^{-1}\mathbi{w}) \\
& \ge h_{\mathbi{v}}(\mathbi{u}) - h_{\mathbi{v}}(\tilde{\mathbi{u}}) 
\ge \big( \sqrt{h_{\mathbi{v}}(\mathbi{u}) } - \sqrt{ h_{\mathbi{v}}(\tilde{\mathbi{u}}) }  \big)^2.
\end{align*}
As a result, for every $\tilde{\mathbi{u}} \in \mathbb{R}^{k-1}$,
$$
E_{\mathbi{v}}(\tilde{\mathbi{u}}, \alpha) \subseteq E_{\mathbi{v}}(\alpha^2) 
, \quad\alpha \ge 0.
$$
According to Prop.~\ref{Prop:inc}, for every $y^n \in E_2,$ there exists $\tilde{\mathbi{u}} \in B_2$ such that
$(B_1)^c \subseteq E_{\mathbi{v}}(\tilde{\mathbi{u}}, n^{1/26})$.
Therefore  for all $y^n \in E_2,$ $(B_1)^c \subseteq E_{\mathbi{v}}( n^{1/13 })$.
Consequently, for every $y^n \in E_2,$
      \begin{align}
 \int_{(B_1)^c} &(\bar{\mathbi{u}} - \Phi^{-1}\mathbi{w})^T  \Phi
(\bar{\mathbi{u}} - \Phi^{-1}\mathbi{w}) f_1(\bar{\mathbi{u}},y^n) d\bar{\mathbi{u}} \nonumber\\
&\le  \int_{E_{\mathbi{v}}( n^{1/13 }) } (\bar{\mathbi{u}} - \Phi^{-1}\mathbi{w})^T  \Phi
(\bar{\mathbi{u}} - \Phi^{-1}\mathbi{w}) f_1(\bar{\mathbi{u}},y^n) d\bar{\mathbi{u}}  \nonumber\\
&\overset{(a)}{=}   \int_{E_{\mathbi{v}}( n^{1/13 }) } (\bar{\mathbi{u}} - \Phi^{-1}\mathbi{w})^T  \Phi
(\bar{\mathbi{u}} - \Phi^{-1}\mathbi{w})  \sqrt{\frac{{|\Phi|}}{{(2\pi)^{k-1}}}}
 \exp\Big( - \frac{1}{2} (\bar{\mathbi{u}} - \Phi^{-1}\mathbi{w})^T \Phi (\bar{\mathbi{u}} - \Phi^{-1}\mathbi{w})  \Big) d\bar{\mathbi{u}} \nonumber \\
&\overset{(b)}{=}  \frac{1}{\sqrt{(2\pi)^{k-1}}}\int_{\mathbi{z}: \;\mathbi{z}^T \mathbi{z} > n^{1/13 } }  \mathbi{z}^T \mathbi{z} 
 \exp\Big( - \frac{1}{2} \mathbi{z}^T \mathbi{z} \Big) d\mathbi{z},\label{eq:p2}
       \end{align}
where $(a)$ follows from \eqref{eq:expf1}; $(b)$ follows from the change of variable 
$\mathbi{z} = \Phi^{1/2}(\bar{\mathbi{u}} - \Phi^{-1}\mathbi{w})$ and the fact that 
$|\Phi^{1/2}| = |\Phi|^{1/2}$.
Let $\mathbi{Z} \sim \cN(\mathbi{0}, I_{k-1})$ be a standard Gaussian random vector. 
   We proceed similarly to \eqref{eq:cmp}:
\begin{align}
 \frac{1}{\sqrt{(2\pi)^{k-1}}}&\int_{\{\mathbi{z}: \;\mathbi{z}^T \mathbi{z} > n^{1/13 } \} }  \mathbi{z}^T \mathbi{z} 
   \exp\Big( - \frac{1}{2} \mathbi{z}^T \mathbi{z} \Big) d\mathbi{z} \nonumber\\
    &= \int_{a = n^{1/13 }}^\infty a d P( \mathbi{Z}^T\mathbi{Z} \le a) \nonumber\\
& = - \int_{n^{1/13 }}^\infty a d P( \mathbi{Z}^T\mathbi{Z} > a)
= - \Big(  a P( \mathbi{Z}^T\mathbi{Z} > a) \Big|_{n^{1/13 }}^{\infty} 
- \int_{n^{1/13 }}^\infty P( \mathbi{Z}^T\mathbi{Z} > a) da \Big) \nonumber\\
& = n^{1/13 } P( \mathbi{Z}^T\mathbi{Z} > n^{1/13 })
+ \int_{n^{1/13 }}^\infty P( \mathbi{Z}^T\mathbi{Z} > a) da - \lim_{a\to\infty}a P( \mathbi{Z}^T\mathbi{Z} > a) \nonumber\\
&\overset{(a)}{\le}  n^{1/13 } \exp\Big(-\frac{1}{2}(n^{1/ 26 } - \sqrt{k-1})^2 \Big)
+  \int_{n^{1/13 }}^\infty  \exp\Big(-\frac{1}{2}( \sqrt{a} - \sqrt{k-1})^2 \Big) da \nonumber\\
&\overset{(b)}{\le}  n^{1/13 } \exp\Big(-\frac{1}{2}(n^{1/ 26 } - \sqrt{k-1})^2 \Big)
+ \int_{n^{1/13 }}^\infty  \exp\Big(-\frac{a}{8} \Big) da \nonumber\\
& = n^{1/13 } \exp\Big(-\frac{1}{2}(n^{1/ 26 } - \sqrt{k-1})^2 \Big)
+ 8 \exp \Big(-\frac{n^{1/13 }}{8} \Big), \label{eq:p3}
\end{align}
where in $(a)$ we used Prop.~\ref{Prop:Z} and the fact that $\lim_{a\to\infty}a P( \mathbi{Z}^T\mathbi{Z} > a)=0$, and in $(b)$ we take a sufficiently large $n$, greater than a suitably chosen value $N(k,\epsilon).$

Finally, using \eqref{eq:p2} and \eqref{eq:p3} in \eqref{eq:100}, we conclude that for every $y^n \in E_2,$
\begin{align*}
I\le &  \frac{k}{n^{12/13} \delta_0^2} \exp\Big(-\frac{1}{2}(n^{1/ 26 } - \sqrt{k-1})^2 \Big)
+ \frac{8k}{n \delta_0^2} \exp \Big(-\frac{n^{1/13 }}{8} \Big) \\
< & \frac{1}{n^{14/13}},
\end{align*}
and this implies \eqref{eq:I} for all $n>N(k,\epsilon)$ as long as we take a sufficiently large $N(k,\epsilon)$.

\section{Proof of Proposition \ref{Prop:bdtr}}\label{ap:bdtr}
In this section we prove a somewhat stronger result which immediately implies Prop.~\ref{Prop:bdtr}.
\begin{proposition}\label{Prop:Zac}
Let $A=(a_{ij})$ be a $(k-1)\times (k-1)$ positive definite matrix and let $B=A^{-1}.$
Then
\begin{equation}\label{eq:obj}
( \tr(B) + \mathbi{1}^T B \mathbi{1} ) 
\Big( \frac{\sum_{i=1}^{k-1}a_{ii}}{k}-\frac{\sum_{i\neq j}a_{ij} }{k(k-1)} \Big) \ge k-1.
\end{equation}
The equality holds if and only if 
\begin{equation}\label{eq:eqcon}
a_{ij}=\begin{cases} 2a &\text{if } i=j\\
a&\text{otherwise}.
\end{cases}
\end{equation} 
for some $a>0$.
\end{proposition}
\begin{proof}
Let $\lambda_1,\dots,\lambda_{k-1}$ be the eigenvalues of $A$ and let $v_1,v_2,\dots,v_{k-1}$ be the corresponding eigenvectors. 
Since $A\succ 0$, the vectors $v_1,\dots,v_{k-1}$ form an orthogonal basis for $\mathbb{R}^{k-1}.$
Let us expand the all-ones vector in $\mathbb{R}^{k-1}$ in this basis:
   $$
\mathbi{1}=\sum_{i=1}^{k-1} \alpha_i v_i.
   $$
It is clear that
\begin{equation}\label{eq:alp}
\sum_{i=1}^{k-1} \alpha_i^2 = k-1
\end{equation}
and the coefficients $\alpha_i$ satisfy $0\le \alpha_i^2\le k-1, i=1,\dots, k-1.$
 Furthermore, $A=\sum_{i=1}^{k-1}\lambda_i v_i v_i^T$ and $B=\sum_{i=1}^{k-1}\frac{1}{\lambda_i} v_i v_i^T.$
We obtain
   $$
   \tr(B) + \mathbi{1}^T B \mathbi{1}=\sum_{i=1}^{k-1} \frac{1}{\lambda_i} + \sum_{i=1}^{k-1} \frac{\alpha_i^2}{\lambda_i}
   $$
and since
  $$
   \frac{\sum_{i=1}^{k-1}a_{ii}}{k}-\frac{\sum_{i\neq j}a_{ij} }{k(k-1)} =\frac{\tr(A)}{k-1}-\frac{\mathbi{1}^T A \mathbi{1} }{k(k-1)}
=\frac{1}{k(k-1)}\Big(k\sum_{i=1}^{k-1} \lambda_i - \sum_{i=1}^{k-1} \lambda_i\alpha_i^2 \Big),
  $$
the left-hand side of \eqref{eq:obj} can be written as
\begin{align}
 \Big( \sum_{i=1}^{k-1} \frac{1}{\lambda_i} + \sum_{i=1}^{k-1} \frac{\alpha_i^2}{\lambda_i} \Big)
\Big( \frac{\sum_{i=1}^{k-1}\lambda_i}{k-1}&-\frac{\sum_{i=1}^{k-1} \lambda_i\alpha_i^2 }{k(k-1)} \Big)
  = \frac{1}{k(k-1)} \Big( \sum_{i=1}^{k-1} \frac{1 + \alpha_i^2}{\lambda_i}  \Big)
\Big( \sum_{i=1}^{k-1}\lambda_i (k-\alpha_i^2) \Big) \nonumber \\
&\ge  \frac{1}{k(k-1)} \Big( \sum_{i=1}^{k-1} \sqrt{(1 + \alpha_i^2)(k-\alpha_i^2)} \Big)^2 \label{eq:inea} \\
&\ge  \frac{1}{k(k-1)} \Big( \sum_{i=1}^{k-1} \sqrt{k} \Big)^2 = k-1,  \label{eq:ineb}
\end{align}
where \eqref{eq:inea} follows from the Cauchy-Schwarz inequality and \eqref{eq:ineb} follows from the fact that $(1+x)(k-x)\ge k$ for all $0\le x\le k-1.$ 
This completes the proof of \eqref{eq:obj}.

It is easy to verify that the equality in \eqref{eq:obj} holds if \eqref{eq:eqcon} holds. Let us prove the ``only if" part. Note that equality in \eqref{eq:ineb}
holds only if $\alpha_i^2 = 0$ or $k-1$ for all $i=1,2,\dots,k-1$.
By \eqref{eq:alp} this means that one of the $\alpha_i$'s is $\sqrt{k-1}$, and all the other $\alpha_i$'s are $0$.
Without loss of generality, suppose that $\alpha_1=\sqrt{k-1}$ and $\alpha_i=0$ for $2\le i\le k-1,$ and thus
    \begin{equation}\label{eq:1}
v_1=\frac{1}{\sqrt{k-1}} \mathbi{1}.
    \end{equation}
 Moreover, if \eqref{eq:inea} holds with equality, then
$$
\frac{\lambda_1^2 (k-\alpha_1^2)}{1+ \alpha_1^2}
= \frac{\lambda_2^2 (k-\alpha_2^2)}{1+ \alpha_2^2}
= \dots =
\frac{\lambda_{k-1}^2 (k-\alpha_{k-1}^2)}{1+ \alpha_{k-1}^2}
$$
and thus
$$
\lambda_2=\lambda_3=\dots=\lambda_{k-1} = \frac{\lambda_1}{k}.
$$
Let $\tilde J$ be a $(k-1) \times (k-1)$ matrix given by
$$
\tilde J=\frac{\lambda_1}{k}
\left[ \begin{array}{cccc}
1 & 1 & \dots & 1 \\
1 & 1 & \dots & 1 \\
\vdots & \vdots & \vdots & \vdots \\
1 & 1 & \dots & 1 
\end{array} \right]
$$
By \eqref{eq:1}, the basis vectors $v_i, i=2,3,\dots,k-1$ are orthogonal to $\mathbi{1}.$ This implies that 
   $$
   (A-\tilde J)v_i=\frac{\lambda_1}{k} v_i, \quad i=2,3,\dots,k-1.
   $$
On the other hand, $\tilde J v_1=\frac{(k-1)\lambda_1}{k} v_1$. Therefore,
$(A-\tilde J) v_1=\frac{\lambda_1}{k} v_1$.
As a result, $(A-\tilde J)v_i=\frac{\lambda_1}{k} v_i$ for all
$i=1,2,\dots,k-1$. Since $v_1,\dots,v_{k-1}$ span $\mathbb{R}^{k-1}$, we deduce that 
$A-\tilde J=\frac{\lambda_1}{k} I_{k-1}$, where $I_{k-1}$ is the identity matrix.
This implies that
  $$
 a_{ij}=  \begin{cases}
2\frac{\lambda_1}{k} &\text{if } i=1,2,\dots,k-1, \\
\frac{\lambda_1}{k} &\text{otherwise.}
  \end{cases}
  $$
Since $\lambda_1>0$, this concludes the proof.
\end{proof}

\section{Proof of \eqref{eq:rhp2}} \label{app:rhp1}
{\em Let $d^\ast$ be as defined in \eqref{eq:dstar}. To prove \eqref{eq:rhp2}, we only need to prove the following inequality:
  \begin{equation}\label{eq:L}
\sum_{j=1}^{k} \Big( \frac{q_{ij}}{q_i} \Big)^2 
\le k \Big( 1 +   (e^{\epsilon} - 1)^2  
 \frac{d^\ast (k-d^\ast)} {(d^\ast e^{\epsilon} + k - d^\ast)^2}  \Big), \;\; i=1,2,\dots, L.
  \end{equation}
}

\begin{proof}
Since $\mathbi{Q}\in \cD_{\epsilon,E}$, by definition we have 
$$
\frac{q_{ij}}{\min_{j'\in[k]}q_{ij'}} \in \{1, e^{\epsilon}\}
$$
for all $i\in[L]$. Given $i\in[L]$, define
$$
d_i = |\{j:\frac{q_{ij}}{\min_{j'\in[k]}q_{ij'}} = e^{\epsilon} \}|.
$$
Then
$$
|\{j: q_{ij} = \min_{j'\in[k]}q_{ij'}  \}| = k-d_i.
$$
Therefore, 
   \begin{align*}
\sum_{j=1}^{k} \big( \frac{q_{ij}}{q_i} \big)^2 
 & = d_i \big( \frac{k e^{\epsilon}}{d_i e^{\epsilon} + k - d_i} \big)^2
 + (k-d_i) \big( \frac{k }{d_i e^{\epsilon} + k - d_i} \big)^2  \\
 & = k^2 \frac{d_i e^{2\epsilon} + k - d_i}{(d_i e^{\epsilon} + k - d_i)^2} \\
 & =  k \frac{(d_i e^{2\epsilon} + k - d_i) (d_i+k-d_i) }{(d_i e^{\epsilon} + k - d_i)^2} \\
 & = k \frac{ d_i^2 e^{2\epsilon} + d_i(k-d_i)(e^{2\epsilon}+1) + (k-d_i)^2 }
{ d_i^2 e^{2\epsilon} + 2d_i(k-d_i)e^{\epsilon} + (k-d_i)^2 }  \\
 & = k \Big( 1 + \frac{  d_i(k-d_i)(e^{\epsilon} - 1)^2  }
{ d_i^2 e^{2\epsilon} + 2d_i(k-d_i)e^{\epsilon} + (k-d_i)^2 }  \Big) \\
 & = k \Big( 1 +   (e^{\epsilon} - 1)^2  
 \Big( 2e^{\epsilon} + \frac{d_i}{k-d_i} e^{2\epsilon}  + \frac{k-d_i}{d_i} \Big)^{-1}  \Big) \\
 & \overset{(a)}{\le} k \Big( 1 +   (e^{\epsilon} - 1)^2  
 \Big( 2e^{\epsilon} + \frac{d^\ast}{k-d^\ast} e^{2\epsilon}  + \frac{k-d^\ast}{d^\ast} \Big)^{-1}  \Big) \\
 & = k \Big( 1 +   (e^{\epsilon} - 1)^2  
 \frac{d^\ast (k-d^\ast)} {(d^\ast e^{\epsilon} + k - d^\ast)^2}  \Big)  ,
    \end{align*}
where $(a)$ follows directly from the definition of $d^\ast$ in \eqref{eq:dstar}.
\end{proof}

\section{Relation to Local Asymptotic Normality and Fisher information matrix}\label{ap:LAN}
Let 
$\mathbi{p}=(p_1,p_2,\dots,p_{k-1},1-\sum_{i=1}^{k-1}p_i)$ be a distribution (probability mass function) on $\cX$. 
Denote by $P(y^n;p_1,p_2,\dots,p_{k-1})$ the probability mass function of a random vector $Y^n$ formed of i.i.d. samples $Y^{(i)}$ 
drawn according to the distribution $\mathbi{p}\mathbi{Q}$. 
Recall that in the beginning of Section~\ref{Sect:Main}, we assumed that the output alphabet of $\mathbi{Q}$ is 
{$\cY=\{1,2,\dots,L'\}$}, and we defined $t'_i(y^n)$ to be the number of times that symbol $i$ appears in $y^n$. Therefore,
\begin{align}
 \log\, &P(y^n;p_1,p_2,\dots,p_{k-1}) 
 = \sum_{i=1}^{L'} t_i'(y^n) \log 
\Big(\mathbi{Q}(i|k)+\sum_{j=1}^{k-1}p_j\big(\mathbi{Q}(i|j)-\mathbi{Q}(i|k) \big) \Big)\nonumber \\
& = \sum_{i=1}^{L'} t_i'(y^n) \log q_i'  +
\sum_{i=1}^{L'} t_i'(y^n) \log 
\Big(\frac{\mathbi{Q}(i|k)}{q_i'}+
\sum_{j=1}^{k-1}p_j \Big(\frac{\mathbi{Q}(i|j)}{q_i'}-\frac{\mathbi{Q}(i|k)}{q_i'} \Big) \Big)\nonumber \\
& = \log P(y^n;1/k,1/k,\dots,1/k) +
\sum_{i=1}^{L} \sum_{a\in A_i} t_a'(y^n) \log 
\Big(\frac{\mathbi{Q}(a|k)}{q_a'}+
\sum_{j=1}^{k-1}p_j \Big(\frac{\mathbi{Q}(a|j)}{q_a'}-\frac{\mathbi{Q}(a|k)}{q_a'} \Big) \Big) \nonumber\\
& = \log P(y^n;1/k,1/k,\dots,1/k) +
\sum_{i=1}^{L}  t_i(y^n) \log 
\Big(\frac{q_{ik}}{q_i}+
\sum_{j=1}^{k-1}p_j \Big(\frac{q_{ij}}{q_i}-\frac{q_{ik}}{q_i} \Big) \Big), \label{eq:pt}
\end{align}
where the last equality follows from \eqref{eq:tq} and \eqref{eq:eqclass}.
Therefore, the function $g(\mathbi{u},y^n)$ defined in \eqref{eq:defg} can be written as
\begin{equation}\label{eq:guy}
 g(\mathbi{u},y^n) = \log \frac{P(y^n;1/k+u_1,1/k+u_2,\dots,1/k+u_{k-1})}{P(y^n;1/k,1/k,\dots,1/k)}.
\end{equation}

In Section~\ref{Sect:Main}, one of the main steps in the proof of \eqref{eq:indv} is to approximate
 $g(\mathbi{u},y^n)$ as $g_2(\mathbi{u},y^n)$.
According to \eqref{eq:defg2} and \eqref{eq:defPhi},
$$
g_2(\mathbi{u},y^n) =
\sum_{i=1}^L  v_i \sum_{j=1}^k \frac{u_j q_{ij}}{q_i}
-\frac{1}{2} \sum_{i=1}^L  n q_i  \Big(\sum_{j=1}^k \frac{u_j q_{ij}}{q_i} \Big)^2 =
\sum_{i=1}^L  v_i \sum_{j=1}^k \frac{u_j q_{ij}}{q_i}
- \frac{1}{2} \mathbi{u}^T \Phi(n,\mathbi{Q}) \mathbi{u}.
$$
The observable random variables in our problem are the coordinates of $Y^n$, and the unknown parameters are
$p_1,p_2,\dots,p_{k-1}.$ Denote by $I(p_1,p_2,\dots,p_{k-1})$ the Fisher information matrix of $\mathbi{p}$ vs. $Y^n$, where
   \begin{equation}\label{eq:Fim}
   [I(1/k,1/k,\dots,1/k)]_{ij} 
= - \mathbb{E}\Big[ \frac{\partial^2}{\partial p_i \partial p_j} \log P(Y^n;p_1,p_2,\dots,p_{k-1}) 
 \Big]_{p_1=p_2=\dots=p_{k-1}=1/k}.
 \end{equation}
We claim that
$$
\Phi(n,\mathbi{Q}) = I(1/k,1/k,\dots,1/k).
$$
Indeed, by \eqref{eq:Fim} we have
\begin{align*}
 [I(1/k,&1/k,\dots,1/k)]_{ij} \\
&= - \mathbb{E}\Big[ \frac{\partial^2}{\partial p_i \partial p_j} 
\Big( \sum_{m=1}^{L}  t_m(y^n) \log 
\Big(\frac{q_{mk}}{q_m}+
\sum_{j=1}^{k-1}p_j \Big(\frac{q_{mj}}{q_m}-\frac{q_{mk}}{q_m} \Big) \Big) \Big) 
 \Big]_{p_1=p_2=\dots=p_{k-1}=1/k} \\
& = \mathbb{E} \Big[ \sum_{m=1}^L t_m(Y^n) 
\frac{(q_{mi}-q_{mk})(q_{mj}-q_{mk})}{q_m^2} \Big]_{p_1=p_2=\dots=p_{k-1}=1/k} \\
& = n \sum_{m=1}^L  
\frac{(q_{mi}-q_{mk})(q_{mj}-q_{mk})}{q_m},
\end{align*}
where the first step follows from \eqref{eq:pt}.
Comparing this with \eqref{eq:altPhi}, we can easily see that $\Phi(n,\mathbi{Q}) = I(1/k,1/k,\dots,1/k)$.
Furthermore, it is also easy to check that
$$
\sum_{i=1}^L  v_i \sum_{j=1}^k \frac{u_j q_{ij}}{q_i}
= \sum_{i=1}^{k-1} u_i \Big( \frac{\partial}{\partial p_i}
\log P(y^n;p_1,p_2,\dots,p_{k-1}) \Big|_{p_1=p_2=\dots=p_{k-1}=1/k} \Big).
$$
Therefore,
$$
g_2(\mathbi{u},y^n) = \sum_{i=1}^{k-1} u_i \Big( \frac{\partial}{\partial p_i}
\log P(y^n;p_1,p_2,\dots,p_{k-1}) \Big|_{p_1=p_2=\dots=p_{k-1}=1/k} \Big) -
 \frac{1}{2} \mathbi{u}^T I(1/k,1/k,\dots,1/k) \mathbi{u}.
$$
Combining this with \eqref{eq:guy}, we conclude that $g(\mathbi{u},y^n) \approx g_2(\mathbi{u},y^n)$
for small $\mathbi{u}$ can be written as
\begin{align*}
& \log \frac{P(y^n;1/k+u_1,1/k+u_2,\dots,1/k+u_{k-1})}{P(y^n;1/k,1/k,\dots,1/k)} \\
\approx & \sum_{i=1}^{k-1} u_i \Big( \frac{\partial}{\partial p_i}
\log P(y^n;p_1,p_2,\dots,p_{k-1}) \Big|_{p_1=p_2=\dots=p_{k-1}=1/k} \Big) -
 \frac{1}{2} \mathbi{u}^T I(1/k,1/k,\dots,1/k) \mathbi{u}.
\end{align*}
This relation is a classic form of expressing Le Cam's result on local asymptotic normality (see \cite[Chapter 2, Theorem 1.1]{Ibrag81}).
As we already remarked in Sec.~\ref{Sect:ovr}, we could use this result to give a simpler proof that a relation of the form \eqref{eq:oblb} holds for an individual privatization scheme $\mathbi Q$; however along this route there seems to be no immediate way to establish the uniform bound
\eqref{eq:oblb}.

\bibliographystyle{IEEEtran}
\bibliography{differential}

\end{document}